\documentclass[11pt]{amsart}

\usepackage{amsfonts,a4wide,amsmath,amssymb,amsthm}
\usepackage[utf8]{inputenc}  
\usepackage[T1]{fontenc}
\usepackage{mathrsfs}
\usepackage[all]{xy}    
\usepackage{mathtools}  
\usepackage{multirow}    
\usepackage{csquotes}	
\usepackage{bm} 
\usepackage{tikz}   
\usepackage{float}
\usepackage{relsize}
\usepackage{tikz-cd}  
\usepackage{colonequals}
\usepackage{enumerate}
\usepackage{booktabs}
\usepackage{comment}
\usepackage{longtable}
\definecolor{amethyst}{rgb}{0.6, 0.4, .8}
\usepackage{breakurl}

\usepackage[pagebackref=true, colorlinks, allcolors = amethyst]{hyperref}

\renewcommand*{\backref}[1]{}
\renewcommand*{\backrefalt}[4]{%
  \ifcase #1 %
    \relax
  \or
    $\uparrow$#2.%
  \else
    $\uparrow$#2.%
  \fi%
}

\numberwithin{equation}{subsection}
\allowdisplaybreaks

\numberwithin{table}{section}
\theoremstyle{plain}\newtheorem{ithm}{Theorem}
\theoremstyle{plain}\newtheorem{iprop}[ithm]{Proposition}

\theoremstyle{plain}
\newtheorem{thm}{Theorem}[section]
\newtheorem{prop}[thm]{Proposition}
\newtheorem{cor}[thm]{Corollary}

\newtheorem{lem}[thm]{Lemma}

\theoremstyle{definition}
\newtheorem{defi}[thm]{Definition}

\newtheorem{remark}[thm]{Remark}

\newtheorem*{conj}{Conjecture}
\newtheorem{exe}[thm]{Example}

\theoremstyle{remark}
\newtheorem{rem}[thm]{Remark}

\newcommand{\gm}{{\mathfrak{m}}}

\newcommand{\gp}{{\mathfrak{p}}}
\newcommand{\gP}{{\mathfrak{P}}}

\newcommand{\Acal}{{\mathcal A}}
\newcommand{\Bcal}{{\mathcal B}}
\newcommand{\Ccal}{{\mathcal C}}

\newcommand{\Ecal}{{\mathcal E}}
\newcommand{\Fcal}{{\mathcal F}}

\newcommand{\Hcal}{{\mathcal H}}

\newcommand{\Jcal}{{\mathcal J}}

\newcommand{\Lcal}{{\mathcal L}}

\newcommand{\Ncal}{{\mathcal N}}
\newcommand{\Ocal}{{\mathcal O}}

\newcommand{\Qcal}{{\mathcal Q}}
\newcommand{\Rcal}{{\mathcal R}}
\newcommand{\Scal}{{\mathcal S}}

\newcommand{\Wcal}{{\mathcal W}}
\newcommand{\Xcal}{{\mathcal X}}
\newcommand{\Ycal}{{\mathcal Y}}

\newcommand{\N}{{\mathbb{N}}}
\newcommand{\Z}{{\mathbb{Z}}}
\newcommand{\Q}{{\mathbb{Q}}}
\newcommand{\R}{{\mathbb{R}}}
\newcommand{\C}{{\mathbb{C}}}
\renewcommand{\P}{\mathbb{P}}
\newcommand{\F}{\mathbb{F}}
\newcommand{\Fp}{{\mathbb{F}_{\! p}}}

\newcommand{\Aut}{\operatorname{Aut}}
\newcommand{\Card}{\operatorname{Card}}

\newcommand{\End}{\operatorname{End}}
\newcommand{\Gal}{\operatorname{Gal}}
\newcommand{\GL}{\operatorname{GL}}
\newcommand{\SL}{\operatorname{SL}}

\renewcommand{\Im}{\operatorname{Im}}

\newcommand{\Spec}{\operatorname{Spec}}

\newcommand{\Stab}{\operatorname{Stab}}

\newcommand{\Cot}{\operatorname{Cot}}

\newcommand{\Jac}{\operatorname{Jac}}

\newcommand{\Pic}{\operatorname{Pic}}

\newcommand{\Sym}{\operatorname{Sym}}

\newcommand{\defeq}{\colonequals}

\renewcommand{\d}{\operatorname{d}}
\newcommand{\mt}{\mapsto}	
\newcommand{\lmt}{\longmapsto}
\newcommand{\ra}{\rightarrow}
\newcommand{\lra}{\longrightarrow}

\newcommand{\Llra}{\Longleftrightarrow}

\newcommand{\fonction}[5]{\begin{array}{c|ccl}           
		#1: & #2 & \longrightarrow & #3 \\
		& #4 & \longmapsto & #5 \end{array}}

\newcommand{\fonctionsansnom}[4]{\begin{array}{ccl}      
		#1 & \lra & #2 \\\
		#3 & \lmt & #4 
\end{array}}

\newcommand{\Qb}{\overline{\Q}}
\newcommand{\kb}{{\overline{k}}}

\newcommand{\GalQ}{{\Gal(\Qb / \Q)}}

\newcommand{\Fpb}{\overline{\Fp}}

\renewcommand{\phi}{\varphi}
\newcommand{\wN}{\widetilde{N}}

\usepackage{cleveref}
\crefname{theorem}{Theorem}{Theorems}

\crefname{lem}{Lemma}{Lemmata}
\crefname{corollary}{Corollary}{Corollaries}
\crefname{proposition}{Proposition}{Propositions}
\crefname{proposition-definition}{Proposition-Definition}{Proposition-Definitions}
\crefname{lemma-definition}{Lemma-Definition}{Lemma-Definitions}
\crefname{definition}{Definition}{Definitions}
\crefname{conjecture}{Conjecture}{Conjectures}
\crefname{question}{Question}{Questions}
\crefname{example}{Example}{Examples}
\crefname{algorithm}{Algorithm}{Algorithms}
\crefname{remark}{Remark}{Remarks}
\crefname{assumption}{Assumption}{Assumptions}

\definecolor{darkgreen}{rgb}{0,0.5,0}
\definecolor{darkblue}{rgb}{0,0,0.8}
\definecolor{darkred}{rgb}{0.8,0,0}

\renewcommand{\c}{\mathcal{C}}

\title{Rational points on $X_0(N)^*$ when $N$ is non-squarefree}

\date{\today}

\author{Sachi Hashimoto}
\address{
Sachi Hashimoto,
  	Department of Mathematics,
  	Brown University,
  	Box 1917,
  	151 Thayer~St.,
  	Providence, RI, 02912, USA}
\email{\href{mailto:sachi_hashimoto@brown.edu}{sachi\_hashimoto@brown.edu}}
\urladdr{\url{https://sachihashimoto.github.io/}}

\author{Timo Keller}
\address{
Timo Keller,
Institut für Mathematik, Universität Würzburg, Emil-Fischer-Strasse 30, 97074,
Würz\-burg, Germany; 
Rijksuniveriteit Groningen, Bernoulli Institute, Bernoulliborg, Nijenborgh 9, 9747 AG Groningen, The Netherlands; 
Leibniz Universität Hannover, Institut für Algebra, Zahlentheorie und Diskrete Mathematik, Welfengarten 1, 30167 Hannover, Germany}
\email{\href{mailto:math@kellertimo.de}{math@kellertimo.de}}
\urladdr{\url{https://www.timo-keller.de}}

\author{Samuel Le Fourn}
\address{Samuel Le Fourn,
Institut Fourier, UMR 5582, Laboratoire de mathématiques Université Grenoble Alpes, CS 40700, 38058 Grenoble Cedex 9, France}
\email{\href{mailto:samuel.le-fourn@univ-grenoble-alpes.fr}{samuel.le-fourn@univ-grenoble-alpes.fr}}
\urladdr{\url{https://www-fourier.univ-grenoble-alpes.fr/~lefourns/}}

\begin{document}

\begin{abstract}
Let $N$ be a positive integer and $X_0(N)^*$ be the quotient of the modular curve $X_0(N)$ by its full group of Atkin--Lehner involutions.
Elkies conjectures that the rational points on $X_0(N)^*$ are only cusps or CM points when $N$ is large enough.
We establish an integrality result for the $j$-invariants of non-cuspidal rational points on $X_0(N)^*$ when $N$ is non-squarefree, representing a significant step toward resolving a key subcase of Elkies’ conjecture.
To this end, we prove the existence of rank-zero quotients of certain modular Jacobians $J_0(pq)$.  
Furthermore, in the non-squarefree case we provide a complete classification of the rational points on $X_0(N)^*$ of genus $1 \leq g \leq 5$, when they are finite. In the process we identify exceptional rational points on $X_0(75)^*$ and $X_0(147)^*$ which were not known before.
\end{abstract}

    
\maketitle

\section{Introduction}
Let $N \geq 1$ be an integer. The modular curve $X_0(N)$ is an algebraic curve over $\Q$ such that for any number field $K$, the non-cuspidal $K$-rational points of $X_0(N)$ classify pairs $(E,C_N)$ of elliptic curves $E/K$ and cyclic subgroups $C_N \subset E$ of order $N$ defined over $K$ (equivalently, isogenies between elliptic curves over $K$ with kernel cyclic of order $N$) up to isomorphism.

This curve is acted on by the group of \emph{Atkin--Lehner} involutions $\Wcal(N)$ (see~\cref{subsecAtkinLehner} for a precise definition). The main object of study in this paper is the quotient curve $X_0(N)^* \colonequals X_0(N)/\Wcal(N)$.
The  modular curve $X_0(N)$ and its Atkin--Lehner quotients have a rich geometry which is expected to constrain their possible rational points.

Elkies formulates the following conjecture, related to the study of $\Q$-curves, which correspond to rational points on the curves $X_0(N)^*$.

\begin{conj}[Boundedness conjecture]
 For all integers $N\gg0$, the points of $X_0(N)^*(\Q)$ are only CM points and cusps.
 \end{conj}
For example, Galbraith \cite{GalbriathQcurves} conjectured a  list of small levels for which there exist non-CM quadratic $\Q$-curves based on computational evidence and earlier work of Momose \cite{Momose1986,MomoseNplus}. 
Galbraith's conjecture has since been verified through a combination of theoretical and computational work  \cite{AraiMomose,ArulMuller22, AABCCKW,ACKP}. 

Elkies' conjecture is motivated by a broader desire to classify rational points on families of modular curves and a growing belief that such rational points should exist for a geometric reason \cite[p.~17]{Ogg75}.
Mazur \cite{Mazur1978} famously studied the curve $X_0(p)$ for $p$ prime, and showed that when $p>37$ the points of $X_0(p)(\Q)$ are all CM or cuspidal, thus classifying rational isogenies of elliptic curves over $\Q$ of prime degree.
Following Mazur's seminal result and building on Momose's work \cite{Momose1986}, Bilu, Parent, and Rebolledo \cite{BiluParent,BiluParentRebolledo} 
showed that the rational points on $X_0(p^k)^*$ are CM or cusps, for $p \geq 11, p \neq 13$ a prime and $k > 1$, whenever the curve has genus at least $1$.
Their result excluded $13$ due to the difficult case of $X_0(13^2)^*$, later determined using quadratic Chabauty in \cite{Balakrishnanetc}.

Mazur's method proceeds in two steps  and we focus in this paper on the first one, an \textit{integrality result}, obtained by Mazur through the so-called  \enquote{formal immersion method}. Crucially, Mazur constructs (when $p \notin \{2,3,5,7,13\}$) an abelian variety quotient $A$ of the Jacobian of $X_0(p)$ such that $A(\Q)$ is finite, i.e.\ a (Mordell--Weil) \enquote{rank zero quotient}, along with a formal immersion $f: X_0(p) \to A$, which he leverages to show that any point on $X_0(p)$ corresponds to an elliptic curve $E$ whose $j$-invariant $j(E) \in \Z[\frac{1}{2}]$ \cite[Corollary 4.4]{Mazur1978}.

In the present work, our main contributions are as follows. We prove an integrality result akin to Mazur's integrality theorem for $X_0(p)$ in the case of $X_0(N)^*$, when $N$ is non-squarefree. More precisely, our work generalizes the results of Momose \cite[Theorem (3.8)]{Momose1986} which dealt with the case of prime powers $N=p^k$ with $k >1$ (upon which the results of Bilu, Parent, and  Rebolledo are based).
\begin{ithm}[cf.\ Theorems \ref{thm:integraljinv} 
and \ref{prop:integralityexclevel}]
\label{introthm1}
Let $N \neq 99, 125, 147$ be a non-squarefree level such that $X_0(N)^*$ has positive genus. 
Let $P^* \in Y_0(N)^*(\Q)$ and $P = (E,C_{N})$ a choice of lift of $P^*$ in $Y_0(N)$, defined over a polyquadratic field $K$. 
Then $(2^2 \cdot 61^2)^N \cdot  j(E) \in \Ocal_K$.

Furthermore, if $p^2 |N$ for some prime $p \notin \{2,3,5,7,13\}$, $j(E) \in \Ocal_K$.
\end{ithm}

Between the full integrality statement of the last sentence of Theorem \ref{introthm1} and its more general but less precise result, we prove a range of intermediary results which can be found in Theorem \ref{thm:integraljinv}.
Since the star quotients are covered by every Atkin--Lehner quotient, this integrality result implies the following important corollary constraining the rational points that appear on Atkin--Lehner quotients of $X_0(N)$.

\begin{cor}
    Let $N \neq 99,125,147$ be a non-squarefree level such that $X_0(N)^
    *$ has positive genus. For any subgroup $G$ of the group of Atkin--Lehner involutions of $X_0(N)$, every non-cuspidal rational point $Q_G \in (X_0(N)/G)(\Q)$ is integral in the sense that for any lift $Q \in X_0(N)(\Qb)$ of $Q_G$, we have $(2^2 \cdot 61^2)^N \cdot  j(Q) \in \overline{\Z}$, and furthermore $j(Q) \in \overline{\Z}$ if $p^2|N$ for some prime number $p \notin \{2,3,5,7,13\}$. 
\end{cor}

\begin{rem}
The corollary can be made more precise, and slightly more general to cover cases where $X_0(N)/G$ has positive genus but $X_0(N)^*$ does not. 

Moreover, when $G$ is not the full group of Atkin--Lehner involutions, non-cuspidal points of $(X_0(N)/G)(\Q)$ can generally be proven to be integral even when $N$ is squarefree with similar techniques, but both these generalizations require yet more case-by-case analyses and careful computations, which would further increase the length of this paper. We chose to focus on the case $X_0(N)^*$ itself since it is the most difficult in that sense.
\end{rem}

The most difficult part of Theorem \ref{introthm1} concerns the integrality of the $j$-invariant at primes above 2, which requires a delicate analysis of the group of components of the Néron model of the Jacobian $J_0(p)$ for $p$ prime such that $p^2|N$, in the fashion of \cite[\S 2.3]{LeFourn1}. This also relies on Mazur's proof of Ogg's conjecture describing fully the group $J_0(p)(\Q)_{\rm{tors}}$ \cite[Theorem 1]{Mazur1977}, and proving the integrality above 2 in full generality would require the generalization of Ogg's conjecture to Jacobians $J_0(pq)(\Q)_{\rm{tors}}$ for $p,q$ distinct primes (see  \cref{subsec:improvementsvaluation2} for further discussion).

To prove Theorem \ref{introthm1} using Mazur's method, we need to build rank zero quotients of the Jacobian of $X_0(N)^*$.
When $N$ is squarefree, assuming BSD, the Jacobian of $X_0(N)^*$ over $\Q$ will not have any rank zero quotient (see Remark \ref{rem:whynonsquarefree}), which is why we restrict  to non-squarefree levels $N$. 
When $N$ is \emph{not} squarefree,  one can degenerate the Jacobian of $X_0(N)^*$ to a quotient of some $J_0(M)$, with $M$ dividing $N$, with a minus sign for one of the Atkin--Lehner involutions, and in this situation we can generally find a rank zero quotient. For example, Mazur's original result relies on the Eisenstein quotient of $J_0(p)^{-}$  when $p \notin \{2,3,5, 7, 13\}$ and the fact that it is of rank zero \cite{Mazur1977}.
We obtain the following result by an analysis of non-vanishing of central values of $L$-functions of modular forms, proving the existence of our main new family of rank zero quotients.
\begin{ithm}[cf.\ \cref{proprankzeroquotient}]
\label{ithm:rkzeroq}
    Let $p$ and $q$ be prime numbers with $p \in \{2,3,5, 7, 13\}$ 
   and $q >23$. Then $J_0(pq)^{-_p,+_q}$ (see Definition \ref{defi:quotientswithsigns})  has a rank zero quotient.
\end{ithm}
For a stronger statement that gives the exact cases when $J_0(pq)^{-_p,+_q}$ has a rank zero quotient, see \cref{proprankzeroquotient}.
To prove Theorem \ref{introthm1}, we are led to the study of rational points of $X_0(N)^*$ on small levels $N$.
Combining computational and theoretical methods, we classify the rational points in these small levels.
We were able to fully algorithmically find Heegner points on the Atkin--Lehner quotient (we develop and implement an algorithm to enumerate rational Heegner points on $X_0(N)^*$, discussed in \cref{sec:conditionb}), but 
the non-Heegner CM points were more difficult.
The interplay between Atkin--Lehner and Galois actions needed for a CM point to be a rational point on the quotient proved to be surprisingly subtle in the non-squarefree case: although these considerations should ultimately boil down to Shimura's reciprocity for CM elliptic curves, no algorithm doing the necessary computations for general CM points exists in the literature, and we had to resort to case-by-case analyses based on isogenies between CM elliptic curves with small field of definition.

The excluded levels of Theorem \ref{introthm1} relate to interesting cases in the analysis of the points on the corresponding modular curves.
For $N = 99$, the curve $X_0(99)^*$ is an elliptic curve of rank 1.
When $N = 125$, the curve $X_0(125)^*$ has one non-CM non-cuspidal (exceptional) point, studied in \cite{ArulMuller22}.

The case $N = 147$ is a new case of interest. In fact, by classifying rational points on small levels, we discovered the following new exceptional points. 
\begin{iprop}
    The curve $X_0(147)^*$ has exactly two non-CM non-cuspidal rational points whose $j$-invariants are the roots of the quartic minimal polynomials in equations \eqref{eq:exceptional1471}, and \eqref{eq:exceptional1472}. 
    The curve $X_0(75)^*$ has exactly one non-CM non-cuspidal rational point whose $j$-invariant satisfies the quartic minimal polynomial of equation \eqref{eq:x075}.
\end{iprop}
Finally, we classify rational points on all modular curves $X_0(N)^*$ with genus $1 \leq g \leq 5$, $N \neq 99$, and study when these curves have exceptional points. 
For $N \leq 67$, Elkies \cite{ElkiesKCurves} studied these levels and found an exceptional point in the case $N = 63$. 

For the following classification result, we computed and classified rational points on over one hundred modular curves.
\begin{iprop}[cf.\ Tables \ref{table:heegnerpts} and \ref{tab:smallcurves}]
\label{prop:exceptional}
    Let $N \geq 1$ non-squarefree such that $X_0(N)^*$ has genus $1 \leq g \leq 5$. Then $X_0(N)^*(\Q)$ has only cusps and CM points, except in the cases:
    \begin{itemize}
    \item (Exceptional points) $N=63, 75,  125, 147$: there are exceptional rational points that do not correspond to CM elliptic curves or cusps.
    \item (Rank 1 elliptic curve) $N = 99$.  \end{itemize} 
\end{iprop}

The present work is the first step in a broader program to classify rational points on $X_0(N)^*$, for $N$ non-squarefree. The authors plan to carry out this classification  following the strategy of Bilu, Parent, and Rebolledo. 
The integrality part is the object of the present paper and isogeny theorems are fully understood following Gaudron and R\'{e}mond \cite{GaudronRemond14,GaudronRemond23}, so there remains to be done a detailed analysis of Runge's method to (hopefully) bound  sharply enough the heights of integral points on $X_0(N)^*$ for $N$ non-squarefree, which the authors have started tackling in a second part of this work.

Many of the theorems contained in this paper rely on computational results and code contained in our GitHub repository~\cite{CodeForOurArticle}. See~\cref{subsec:algorithms and code} for more details.

In the next section, we give a more detailed overview of the proof strategy for the reader's convenience. We end that section with a collection of notation used in the paper. In \cref{sec:preliminaries} we discuss background material on modular curves, degeneracy maps, and Atkin--Lehner operators. We prove the existence of rank zero quotients in \cref{sec:analysis-of-vanishing-of-the-l-functions}. In \cref{sec:formalimmersion}, we construct the formal immersions and give the integrality proof for most levels. We also give special attention to the integrality at $2$. In \cref{sec:exclevels} we discuss the exceptional levels, which require more detailed case analysis. In \cref{sec:heegnerpts}, we discuss the theory of Heegner points and  provide an explicit algorithm for determining rational Heegner points on $X_0(N)^*$. This section also contains results about non-Heegner CM points, and a classification of the rational points on exceptional levels.
Finally, in \cref{sec:smalllevels} we discuss the classification of rational points on $X_0(N)^*$ of genus $1 \leq g \leq 5$.

\section{Overview of the strategy}
\label{sec:overview}
\subsection{Overview}
In this section, we discuss several assumptions that can be made on the level $N$ which help us with the analysis of the rational points on $X_0(N)^*$. We put off some technical proofs and an analysis of the exceptional levels to give the reader a broad overview of our strategy for tackling the generic case.
\begin{defi}
	For any level $N \geq 1$, a rational point $P \in X_0(N)^*(\Q)$ is said to be \emph{trivial} if it is a cusp or a CM point and \emph{nontrivial} otherwise, and $X_0(N)^*(\Q)$ is called \emph{trivial} if it only contains trivial points and \emph{nontrivial} otherwise.
\end{defi}
First, recall the following result when the level $N$ is a prime power.
\begin{thm}
	When $N = p^k$ with $p$ prime, $k >1$ and the genus of $X_0(p^k)^*$ is nonzero, $X_0(p^k)^*(\Q)$ is trivial except when $N=5^3$ (and in that case, it contains exactly one nontrivial rational point).
\end{thm}

\begin{proof}
	This is almost completely contained in \cite[Theorem 1.1]{BiluParentRebolledo} and the discussion that follows, except for $N = 13^2$ where $X_0(13^2)^+ \cong X_{\rm{split}}(13)^+$ which has been famously solved in \cite{Balakrishnanetc}, and for $N=5^3$ solved in \cite{ArulMuller22}, both with the recent developments of the quadratic Chabauty method.
\end{proof}

\begin{defi}[Rank zero quotient]
	Let $B$ an abelian variety over $\Q$.
	
	A \emph{rank zero quotient of} $B$ is a quotient abelian variety $A$ of $B$ with $\dim A>0$ such that $A(\Q)$ is finite (i.e.\ of Mordell--Weil rank $0$).
\end{defi}
We can thus assume by default in what follows (and we will) that the level $N$ is \emph{not} a prime power and not squarefree.
For each level $N \geq 1$, we wish to apply Mazur's method and construct a formal immersion from $X_0(N)^*$ to a rank 0 quotient of $J_0(M)$ for some $M|N$.

\begin{rem}
\label{rem:whynonsquarefree}
	For $N$ squarefree, the simple factors over $\Q$ of the Jacobian $J_0(N)^*$ correspond by Eichler--Shimura  theory to the normalised newforms $f \in S_2(\Gamma_0(M))^*$ for all divisors $M$ of $N$. 
Their $L$-functions have sign $-1$ in their functional equations so $L(f,1)=0$, therefore assuming the BSD conjecture, all the simple factors of $J_0(N)^*$ have nonzero rank. 
This is why the traditional Mazur's method (and the classical Chabauty method) are not expected to work in this case. However, that is not the topic of the current paper, where we assume $N$ is non-squarefree precisely because some old Hecke eigenforms might provide a rank zero quotient, as will be shown below.
\end{rem}

We now discuss the conditions on divisors $M$ of $N$ that make it possible to apply Mazur's method.

\begin{defi}[Admissible divisors]\label{def:adimissible}
Let $N \geq 1$ be a level.  An \emph{admissible divisor of $N$} is a proper divisor $M|N$ such that  $(HV)$ and $(RZQ)$ hold, where those conditions are defined as follows:

\begin{enumerate}
    \item[$(HV)$] For every prime $p|N$, $v_p(M) \leq \lceil v_p(N)/2 \rceil$ \quad (``at most half valuation'').

    \item[$(RZQ)$] $J_0(M)^{\rm{new}}$ has an \emph{admissible} rank zero quotient, i.e., a positive-dimensional rank zero quotient such that for some common prime divisor $q$ of $N/M$ and $M$, the Atkin--Lehner involution $w_q^{(M)}$ acts as $-1$ on this quotient (and furthermore, if $2|\gcd(M,(N/M))$, then $w_2^{(M)}$ must act by $-1$). 
\end{enumerate}
\end{defi}

\begin{rem}
	The additional technical restrictions in (RZQ) on Atkin--Lehner signs are designed to help prove that the maps $X_0(N)^* \ra J_0(M)$ we build are formal immersions at the relevant cusps. To give a somewhat abusive example, for any prime $p$ there is a natural map $X_0(p)^+ \ra J_0(p) \ra J_0(p)^-$ (and the latter is often of rank 0), but it turns out to be constant precisely because of the Atkin--Lehner signs, so of course it is not a formal immersion.
\end{rem}
\begin{defi}[Almost squarefree level]
\label{defi:almostsquarefree}
    A level $N \geq 1$ is called \emph{almost squarefree} if it can be written for some $s \geq 0$ as
  \begin{equation}
	\label{eqshapeN}
N = p^k q_1 \cdots q_s
\end{equation}
with $p,q_1, \cdots, q_s$ pairwise distinct primes and $k=2$ or 3.  The prime $p$ is called the \emph{powerful prime} of $N$.
\end{defi}
When $N$ is almost squarefree, any proper squarefree divisor $M|N$ satisfies $(HV)$.
We will solve most cases by reducing the level to an almost squarefree level, which we do with the following result.
\begin{prop}[Going down (Proposition \ref{prop:constructionofdegeneracymappsi})]
\label{propredalmostsqf}
	For every level $N \geq 1$ and every prime number $q$ such that $q^2  |  N$, we can define a morphism 
	\begin{equation}
	    	\psi_{N,N/q^2}\colon X_0(N)^* \rightarrow X_0(N/q^2)^*
	\end{equation}
	sending cusps to cusps, CM points to CM points and exceptional points to exceptional points. More precisely, for each $P \in Y_0(N)^*(\Qb)$, the $\Qb$-isogeny classes of $\psi_{N,N/q^2}(P)$ and of $P$ are the same.
\end{prop}

\begin{cor}
\label{cor:trivialityandabovelevels}
	For any level $\widetilde{N} \geq 1$ and any $N | \widetilde{N}$ such that $\widetilde{N}/N$ is a square, if $X_0(N)^*(\Q)$ is trivial, then $X_0(\widetilde{N})^*(\Q)$ is trivial too. 
\end{cor}
 Thus, starting from any non-squarefree level $\widetilde{N}$, we can reduce  to an almost squarefree level $N$ (it can be a prime power), such that $\widetilde{N}/N$ is a square.  
 This corollary allows us to reduce our investigation to computing rational points on almost squarefree levels (which will work except for finitely many cases).

To check whether some $M|N$ satisfies (RZQ), we investigate the quotients of factors of $J_0(M)$ with fixed Atkin--Lehner signs.

\begin{defi}
\label{defi:quotientswithsigns}
	For any level $M \geq 1$ and any prime number $p  |  M$, define $J_0(M)^{-_p,+_{(p)}}$ the abelian subvariety of $J_0(M)$ fixed by all Atkin--Lehner involutions $w_q^{(M)}$ ($q$ prime dividing $M$ distinct from $p$) and with eigenvalue $-1$ for $w_p^{(M)}$ (for details on Atkin--Lehner involutions, see section \ref{subsecAtkinLehner}).
\end{defi}

We will be looking for divisors $M|N$ such that $J_0(M)^{-_p,+_{(p)}}$ has a rank zero quotient. The following description of exceptional numbers and levels  determines the cases where it is impossible to find such a rank zero quotient (for $M=p,pq$ or $pqq'$ respectively).

\begin{defi}[Exceptional tuples of primes]
\hspace*{\fill}

A prime $p$ is called \emph{exceptional} if $J_0(p)^-=0$, i.e.\ $p \in \{2,3,5,7,13\}$.

	A pair of distinct primes $(p,q)$ is called an \emph{exceptional pair} if it belongs to the following table.
			\begin{equation}
					\begin{array}{|c|c|c|c|c|c}
					\hline
					p & 2 & 3 & 5 & 7 \\
					\hline
					q & \in \{3,5,7,11,23\} & \in \{2,5,11\} & 2 & 3 \\
					\hline
					\end{array}.
			\end{equation}
				
The only \emph{exceptional triples} of distinct primes are the triples $(p,q,q')$ with $p=2$ and  $(q,q')= (3,5)$ or $(5,3)$ (and there are no exceptional $n$-tuples for $n \geq 4$, by definition).
\end{defi}

\begin{defi}[Exceptional levels]
	\label{defexceptionallevel}
	Let $\widetilde{N} \geq 1$ be a non-squarefree integer that is not a prime power, such that the genus of $X_0(\widetilde{N})^*$ is positive.
	
	We say $\widetilde{N}$ is \emph{an exceptional level} if for every prime $p$ such that $p^2  |  \wN$, the prime $p$ is exceptional, and for every $N  |  \wN$ almost squarefree of the shape $N = p^k q_1 \cdots q_s$ such that $\wN/N$ is a square, every subtuple of $(p,q_1, \cdots, q_s)$ that starts with $p$ is exceptional (in particular $s\leq 2$ and $p=2$ if $s=2$).
\end{defi}

\begin{rem}
    The non-exceptional levels are those for which, up to lowering the level by squares, one will be able to build formal immersions $X_0(N)^* \rightarrow J_0(M) \rightarrow A$ for some $M$ dividing $N$ and some rank zero quotient $A$ of $J_0(M)$. 
    For exceptional levels, there are generally still rank zero quotients in the Jacobian. However, those rank zero quotients do not come with an easy recipe to build a formal immersion. This is why we chose to go through this dichotomy, and deal with exceptional levels separately.
\end{rem}

This theorem is our main result on the existence of rank 0 quotients. It relies on analytic results about non-vanishing of $L$-values and will be proved in section \ref{sec:analysis-of-vanishing-of-the-l-functions}.
	
\begin{thm} \label{thm:rank0quotient}
	For every $\widetilde{N}  \geq 1$ non-squarefree, not a prime power, and non-exceptional such that the genus of $X_0(\widetilde{N})$ is positive, 
 there are integers $M$ and $N$ with $M  |  N  |  \widetilde{N}$ such that the following holds.
	
	\begin{enumerate}[(i)]
		\item $\widetilde{N}/N$ is a square and $N$ is almost squarefree \eqref{eqshapeN}. 
		
		\item $M$ is a squarefree divisor of $N$ and $p  |  M$ (where $p$ is the powerful prime factor of $N$).
		
		\item There exists a rank zero new quotient $A$ of $J_0(M)^{-_p,+_{(p)}}$.
	\end{enumerate}
 In particular, $M$ is an admissible divisor of $N$. 
\end{thm}

\begin{rem}
     To compare it with the definition of admissible divisors, condition $(ii)$ automatically implies $(HV)$ since $M$ is squarefree, and $(iii)$ provides the admissible rank zero quotient (if $2|(M,N/M)$, $2$ is the powerful prime of $N$ so we indeed have $w_2^{(M)}$ acting as $-1$ on $A$), hence implies $(RZQ)$.

    As will be clear later, the almost squarefree case is easier to manipulate, but we need more general definitions and hypotheses to be able to also handle exceptional levels.
\end{rem}

Using Mazur's method on formal immersions in section \ref{sec:formalimmersion} we obtain the following.
\begin{thm}\label{thm:integraljinv}
Let $N$ be a non-exceptional non-squarefree and non-prime power level. Let $P^* \in Y_0(N)^*(\Q)$ and $P = (E,C_{N})$ a choice of lift of $P^*$ in $Y_0(N)$, defined over a polyquadratic field $K$. Then: 

$(a)$ For any prime ideal $\gp$ of $\Ocal_K$ of odd residue characteristic, $E$ has potentially good reduction at $\gp$.

$(b)$ For any prime ideal $\gp$ of $\Ocal_K$ above 2, $v_\gp(j(E)) \geq -N v_\gp(2)$.

$(c)$ If $p^2 |N$ with $p \notin \{2,3,5,7,13\}$, $j(E) \in \Ocal_K$. 
\end{thm}

Theorem \ref{thm:integraljinv} is the most general version of the theorem, but we must also deal with exceptional levels. This must be done by studying a finite set of the minimal exceptional levels and apply Mazur's method carefully in each of those cases.

After small modifications (see section \ref{subsec:detexclevels} for the precise explanation), we obtain the following finite list of minimal exceptional levels, to be studied separately: 
	\begin{eqnarray}
		\Lcal & = & \{40, 48, 72, 80, 88, 96, 100, 108, 112, 120, 135, 144, 147, 162,\notag
        \\ & & 176,180, 184, 196, 200, 216, 224, 225, 240, 250, 297, 368, 396, 405,\\ & &  441, 450, 486, 
		500, 891, 1029, 1125, 1225, 1250\}.\notag\end{eqnarray}
The levels in this list are dealt with in section \ref{sec:exclevels}.

The following theorem gives a broad integrality result for all levels, including the exceptional levels. We show a stronger integrality result in section \ref{sec:exclevels} that gives a more precise bound on the denominator depending on the individual level. 
\begin{thm}
Let $N \neq 147, 125, 99$ be a non-squarefree level such that $X_0(N)^*$ has positive genus. Let $P^* \in Y_0(N)^*(\Q)$ and $P = (E,C_{N})$ a choice of lift of $P^*$ in $Y_0(N)$, defined over a polyquadratic field $K$. 
Then $(2^2 \cdot 61^2)^N \cdot  j(E) \in \Ocal_K$. 
When $N = 125$ and $147$, the curves $X_0(125)^*$ and $X_0(147)^*$ have respectively one and two non-CM, non-cuspidal rational points, whose $j$-invariants are given in \cite{ArulMuller22}, \eqref{eq:exceptional1471}, and \eqref{eq:exceptional1472}. 
\end{thm}

\section*{Acknowledgments}
We are very grateful to Eran Assaf for help computing $j$-maps, Maarten Derickx for helpful comments and feedback on a first version of this paper (including correction of a mistake regarding Atkin--Lehner signs) and the basic idea exploited in~\Cref{ssec:rat pts via RZQ},
Noam Elkies for advice on CM points on genus 1 curves,
Elvira Lupoian for a very helpful discussion on generalized Ogg's conjecture,
Jinzhao Pan for discussions about ``$r < g$'' in an early stage of the project, Michael Stoll for a suggestion on computing models of $X_0(N)^*$ in genus 1 and for comments on our draft, and Andrew Sutherland for assistance with CM point computations. We are  thankful to the organizers of the PCMI 2022 Research Program ``Number Theory informed by Computation'' and Samir Siksek for connecting the authors at the beginning of the project. 
We also thank Barry Mazur, Ken Ribet, and Preston Wake for their expertise and advice on Ogg's conjecture.
We are very grateful to the referee for pointing out several mistakes in the first draft of this article.

We used Claude Fable 5.1 (High effort) to review the article for mistakes and typos during the revision process.

SH was supported by an AMS-Simons Travel Grant. TK was supported by the 2021 MSCA Postdoctoral Fellowship 01064790 -- Ex\-pli\-cit\-Rat\-Points while working on this article. SH and SLF were supported by IEA PARIALPP. SLF was also supported by IRGA PointRatMod and ANR JINVARIANT.

\subsection{Notation}

Here, we give some notation used throughout the paper:

\paragraph{Divisors}
\begin{itemize}
	\item $N,M,\widetilde{N},d$ will always denote positive integers. 
    \item $\omega(N)$ is the number of distinct prime factors of $N$.
	\item For any $d$, $d \| N$ means that $d$ is a Hall divisor of $N$, i.e.\ $d|N$ and $\gcd(d,N/d)=1$. 
    \item For two integers $N,\widetilde{N}$, one writes $N |_{\square} \widetilde{N}$ when $N |\widetilde{N}$ and $\widetilde{N}/N$ is a square integer, and one says that $N$ is square-below $\widetilde{N}$, and respectively $\widetilde{N}$ is square-above $N$.
\end{itemize}

\paragraph{Modular curves}
\begin{itemize}
    \item For $C$ an abelian group or commutative group scheme and $Q \geq 1$ an integer, $C[Q]$ is the $Q$-torsion subgroup (resp.\ subgroup scheme) of $C$. 
    \item $\Hcal$ is the complex upper half plane with $\GL_2^+(\mathbb{R})$ acting on it via fractional linear transformations.
	\item $X_0(N)$ is the usual modular curve of level $N$ over $\Q$ associated to the congruence subgroup 
    \begin{equation}
    \Gamma_0(N) = \left\{ \begin{pmatrix} a & b \\ c & d \end{pmatrix}  \in \SL_2(\Z) \,   \Big\vert \, N|c \right\}.
    \end{equation}
    \item $\Wcal(N)$ is the Atkin--Lehner group of involutions on $X_0(N)$ (Definitions \ref{defi:atkinlehnerinvolutions} and \ref{defnotationAL}).
    \item The curves $X_0(N)^+$ and $X_0(N)^*$ are respectively the plus quotient and star quotient, i.e., the quotient by the Fricke involution $w_N$ and by all Atkin--Lehner involutions, respectively (Definition \ref{defi:starquotients}). We tacitly use the interpretation of $X_0(N)$ as a moduli space.
    \item For any point $P^* \in X_0(N)^* (\Q)$, a lift $P \in X_0(N)(K)$ of $P^*$ is a point $P \in X_0(N)(\Qb)$ and a minimal polyquadratic field $K$ such that $P$ is $K$-rational and the image of $P$ in $X_0(N)^*$ is $P^*$ (see Definition \ref{defi:liftofPstar}). 
	\item $S_2(\Gamma_0(N))$ is the complex vector space of cusp forms of weight $2$ for $\Gamma_0(N)$.
	\item The cusp $\infty$ is the cusp associated to the orbit of $(1:0) \in \P^1(\Q)$ in $X_0(N)$.
    \item $\iota_N : X_0(N) \rightarrow J_0(N)$ is the canonical map from $X_0(N)$ to its Jacobian with base point $\infty$.
    \item For $M|N$, $i_{N,M} : X_0(N) \rightarrow X_0(M)$ is the degeneracy map which on pairs $(E,C_N)$ gives $(E,C_N[M])$ (see Definition \ref{defi:degeneracymaps} for more general degeneracy maps).
    \item $\Xcal_0(N)$ is the compactified coarse moduli scheme over $\Z$ associated to $\Gamma_0(N)$ (Definition \ref{defi:Xcal0N}).
\end{itemize}

\section{Background material on modular curves}\label{sec:preliminaries}

In this section, we brief\-ly explain how to relate algebraic degeneracy morphisms between modular curves to operators between cuspidal modular forms. All this is quite well-known but not always stated explicitly, and we need it for computations later.

Recall that for every $z \in \Hcal$, we define $q(z) \defeq e^{2i \pi z}$ (shortened to $q$ if there is no ambiguity).

\subsection{Degeneracy maps}
\begin{prop}
\label{prop:isommodularforms1forms}
	For every level $N \geq 1$: 
	
	$(a)$ There is a canonical isomorphism
	\begin{equation}
	    	\fonction{\psi_N}{S_2(\Gamma_0(N))}{H^0(X_0(N)_\C,\Omega^1)}{f}{2 i \pi f(z) dz}
	\end{equation}
	where $H^0(X_0(N),\Omega^1)$ is the $\C$-vector space of holomorphic 1-forms on $X_0(N)$.
	
	$(b)$ For every $f \in S_2(\Gamma_0(N))$, if $f = \sum_{n \geq 1}a_n q^n$ at the cusp $\infty$, 
 \begin{equation}
 \psi_N(f) = \sum_{n \geq 0} a_{n+1}(f) q^{n} dq  = f(q) dq/q.
 \end{equation}
\end{prop}

\begin{proof}
	See \cite[Theorem~3.3.1 and Exercise~3.3.6]{DiamondShurman}.
\end{proof}

We also have the explicit description of complex points of $Y_0(N)$:
\begin{equation}
    \fonction{\varphi_N}{\Gamma_0(N) \backslash \Hcal}{Y_0(N)(\C)}{\Gamma_0(N)\tau}{(E_\tau := \C/(\Z+\Z\tau), \langle 1/N \rangle)}
\end{equation}

\begin{defi}[Degeneracy morphisms]
\label{defi:degeneracymaps}
	For levels $M  |  N$ and $n \geq 1$ dividing $N/M$, we 
	define the \emph{degeneracy map}
	\begin{equation}
	\fonction{i_{N,M}^{(n)}}{Y_0(N)}{Y_0(M)}{(E,C_N)}{(E/(C_N[n]),C_N[nM]/C_N[n])}.
	\end{equation}
    which extends naturally to $X_0(N) \rightarrow X_0(M)$. 
	In the special case $n=1$, we will denote by $i_{N,M} = i_{N,M}^{(1)}$.
\end{defi}

\begin{rem}
    It can be immediately checked with the definition that for any levels $L|M|N$, any $d|(N/M)$ and any $d'|(M/L)$, $i_{M,L}^{(d')} \circ i_{N,M}^{(d)} = i_{N,L}^{(dd')}$. In particular, for $d=d'=1$, $i_{N,L} = i_{M,L} \circ i_{N,M}$.
\end{rem}
On the Poincaré half-plane, define $A_n\colon \tau \mt n \tau$ and observe that it factors through a morphism $\Gamma_0(N) \backslash \Hcal \ra \Gamma_0(M) \backslash \Hcal$, which we also denote by $A_n$. On complex points, we check immediately that the diagram below commutes.
\begin{equation}
    \label{eqcomdegenAn}
    \begin{gathered}
    \xymatrix{
	\Gamma_0(N) \backslash \Hcal \ar[r]^-{\varphi_N} \ar[d]_{A_n} & Y_0(N) (\C) \ar[d]^{i_{N,M}^{(n)}} \\ 	\Gamma_0(M) \backslash \Hcal \ar[r]^-{\varphi_M} & Y_0(M) (\C)
}
\end{gathered}
\end{equation}
This also allows us to see how $i_{N,M}^{(n)}$ extends to cusps and in particular sends $\infty$ on $X_0(N)$ to $\infty$ on $X_0(M)$ (for more on the cusps and degeneracy maps, see \S \ref{subseccuspsdegeneracymaps}).

\begin{defi} \label[definition]{action of GL on functions}
	Following the notation of Atkin and Lehner, for any even integer $k \geq 0$, any function $f$ on $\Hcal$, and any $\gamma = \begin{pmatrix} a & b \\c & d \end{pmatrix} \in \GL_2^+(\R)$, 
	\begin{equation}
	f_{|_k [\gamma]}(\tau) := \frac{(\det \gamma)^{k/2}}{ (c \tau + d)^{k}} f \left( \frac{a \tau + b}{c \tau + d} \right). 
	\end{equation}
	The map $(\gamma,f) \mapsto f_{|_k [\gamma]}$ defines a right action of $\GL_2^+(\R)$ on the complex-valued functions on $\Hcal$. When $k$ is not specified, it is by default $k=2$ and we simply write $f_{[\gamma]}$.
	
	In particular, for the matrix $A_n = \begin{pmatrix} n & 0 \\ 0 & 1 \end{pmatrix}$, we have $f|_{[A_n]}(\tau) = n f(n \tau)$.
\end{defi}

\begin{lem}
\label{lemcotangentpullback}
 Let $f: X \rightarrow Y$  be a finite morphism of (smooth projective algebraic geometrically connected) curves over a field $k$. 
 Let $J_X$ and $J_Y$ be the respective Jacobians of $X$ and $Y$, and $f_*: J_X \ra J_Y$ the pushforward map induced by $f$.
 Let $x\in X(k)$ and $y=  f(x)$, $\iota_X : X \ra J_X$ and $\iota_Y: Y \ra J_Y$ the Abel--Jacobi maps based respectively at $x$ and $y$.
 We have the commutative diagram
\begin{equation}
\begin{gathered}
 \xymatrix{
 \Cot_0(J_Y) \ar[r]^{\Cot_0(f_*)} & \Cot_0(J_X)  \\
 H^0(J_Y, \Omega^1) \ar[r]^{(f_*)^*} \ar[d]_{(\iota_Y)^*}  \ar[u]^{-_{|0}} & H^0(J_X, \Omega^1) \ar[d]_{(\iota_X)^*} \ar[u]^{-_{|0}}\\
 H^0(Y, \Omega^1) \ar[r]^{(f^*)} & H^0(X,\Omega^1)}
 \end{gathered}
\end{equation}
where all vertical maps are isomorphisms.
\end{lem}

\begin{proof}
First, restriction to 0 of global closed 1-forms of abelian varieties induces an isomorphism with the cotangent space by \cite[p.~40, item~$(iii)$]{MumfordAbVar08}, so the top vertical maps are isomorphisms, and the upper square commutes by restriction to the cotangent spaces at 0. Second, the pullbacks by Abel--Jacobi maps are isomorphisms by \cite[Proposition 2.2]{Milne86}. Finally, for every point $P \in X(\kb)$, 
\begin{equation}
 (f_* \circ \iota_X) (P) = f_* ([P] - [x]) = [f(P)] - [f(x)] = [f(P)] - [y] = (\iota_Y \circ f)(P)  
\end{equation}
so $f_* \circ \iota_X =  \iota_Y \circ f$, and by pullback we obtain that the lower square commutes.
\end{proof}

\begin{prop}
\label{propdiagramiNM}
	The pullback morphism $(i_{N,M}^{(n)})^*\colon H^0(X_0(M),\Omega^1) \ra H^0(X_0(N),\Omega^1)$ is identified to the operator $A_n$ on modular forms, i.e., we have the commutative diagram 
	\begin{equation}
    \begin{gathered}
	\xymatrix{
		H^0(X_0(M)_\C,\Omega^1) \ar[rr]^-{(i_{N,M}^{(n)})^*} \ar[d]_{\psi_M^{-1}} & & H^0(X_0(N)_\C,\Omega^1) \ar[d]^{\psi_N^{-1}} \\
		S_2(\Gamma_0(M)) \ar[rr]^-{[A_n]} & & S_2(\Gamma_0(N)).
	} \qedhere
    \end{gathered}
	\end{equation}
\end{prop}

\begin{proof}
	Since the holomorphic 1-forms are determined by their $q$-expansions (so by their restriction around $\infty$), it is enough to compute $(i_{N,M}^{(n)})^*$ on the $q$-expansions, and on $q$-charts $i_{N,M}^{(n)}$ is the $n$-th power map $p_n\colon q \mt q^n$ by \eqref{eqcomdegenAn}. For any 1-form $\omega \in \Omega^1_{X_0(M)}$, writing $\omega = f(q) \frac{dq}{q}$, we have 
	\begin{equation}
	p_n^* \omega = f(q^n) \frac{n q^{n-1} dq}{q^{n}} = n f(q^n) \frac{dq}{q},
	\end{equation}
	so $\psi_N^{-1}(A_n^* \omega) = n f(q^n) = f_{[A_n]} = (\psi_M^{-1}(\omega))_{[A_n]}.$
\end{proof}

\begin{cor}
 With the same notation as in Proposition \ref{propdiagramiNM}, we have the following commutative diagram
 	\begin{equation}
    \begin{gathered}
	\xymatrix{
		\Cot_0(J_0(M)_\C) \ar[rr]^-{((i_{N,M}^{(n)})_*)^*} \ar[d]_{\cong} & & \Cot_0(J_0(N)_\C) \ar[d]^{\cong} \\
		S_2(\Gamma_0(M)) \ar[rr]^-{[A_n]} & & S_2(\Gamma_0(N))
	}
    \end{gathered}
	\end{equation}
where the vertical arrows are obtained by composition of the vertical arrows in Proposition \ref{propdiagramiNM} and Lemma \ref{lemcotangentpullback}.
\end{cor}

\subsection{Atkin--Lehner operators algebraically}
\label{subsecAtkinLehner}

We recall a useful lemma.

\begin{lem}
	\label{lemactAL}
	Let $\gamma = \begin{pmatrix} a &  b \\ c & d \end{pmatrix} \in M_2(\Z) \cap \GL_2(\R)$ and $\tau \in \Hcal$. Then, for $\tau' = \gamma \tau$,
	\begin{equation}
	\fonctionsansnom{E_\tau}{E_{\tau'}}{z}{\frac{z \det(\gamma)}{c \tau +d}}
	\end{equation}
	is an isogeny with degree $|\det (\gamma)|$.
\end{lem}

\begin{proof}
	The multiplication by $\frac{\det(\gamma)}{c \tau +d}$ sends $1$ to $a - c \tau'$ and $\tau$ to $-b + d \tau'$, which proves that it factors through an isogeny $E_\tau \ra E_{\tau'}$. To compute its kernel, we need to know for which $(x,y) \in \Q^2$ we have $(x + y \tau) \frac{\det \gamma}{c \tau +d} \in \Z+ \Z \tau'$. Looking at the multiplication map, this amounts to having 
	\begin{equation}
	\begin{pmatrix} a & -c \\ -b & d \end{pmatrix} \begin{pmatrix} x \\ y \end{pmatrix} \in \Z^2 \Llra   \begin{pmatrix} x \\ y \end{pmatrix} \in \frac{1}{\det \gamma} \begin{pmatrix} d & c \\ b  & a \end{pmatrix} \Z^2,
	\end{equation}
	which defines a lattice in which $\Z^2$ is contained with index $|\det \gamma |$.
\end{proof}

\begin{defi}[Hall divisor]
	For any integer $N \geq 1$, a \emph{Hall divisor} of $N$ is a positive $Q|N$ such that $\gcd(Q,N/Q)=1$ (also denoted by $Q \| N$). 
\end{defi}

\begin{defi}[Atkin--Lehner involutions]
\label{defi:atkinlehnerinvolutions}
Let $N \in \Z_{\geq 1}$ and $Q$ a Hall divisor of $N$.
 
$\bullet$ An \emph{Atkin--Lehner matrix $W_Q^{(N)}$} is defined as
\begin{equation}
W_Q ^{(N)}:= \begin{pmatrix} Qa & b \\ Nc & Qd \end{pmatrix} \textrm{ such that } a,b,c,d \in \Z, \quad \det W_Q = Q.
\end{equation}
 Then, $f \mapsto f_{[W_Q^{(N)}]}$ is an involution on $S_2(\Gamma_0(N))$ that does not depend on the choice of $a,b,c,d$ \cite[p.~138]{AtkinLehner70}.
 
$\bullet$ Functorially, we define the \emph{Atkin--Lehner involution with index $Q$ on $X_0(N)$} as
\begin{equation}
\fonction{w_Q^{(N)}}{Y_0(N)}{Y_0(N)}{(E,C_N)}{(E/C_N[Q],(C_N + E[Q])/C_N[Q])}
\end{equation}
(which then extends to $X_0(N) \rightarrow X_0(N)$).

$\bullet$ For any divisor $d|N$, we define by abuse of notation $W_d^{(N)} := W_Q^{(N)}$ and $w_d^{(N)} := w_Q^{(N)}$ with $Q := \prod_{p|d} p^{v_p(N)}$ the unique Hall divisor of $N$ with the same prime factors as $d$. In particular, for $p|N$ prime, $W_p^{(N)} := W_{p^{v_p(N)}}^{(N)}$ and $w_p^{(N)} := w_{p^{v_p(N)}}^{(N)}$.
\end{defi}

\begin{prop}
\label{propALalgebraically}
For any level $N \geq 1$ and any positive $d|N$, we have commutative diagrams
	\begin{equation}
    \begin{gathered}
	\xymatrix{
		\Gamma_0(N) \backslash \Hcal \ar[r]^-{W_d^{(N)}} \ar[d]_{\varphi_N} & \Gamma_0(N) \backslash \Hcal \ar[d]^{\varphi_N} \\
		Y_0(N) (\C) \ar[r]^-{w_d^{(N)}}&  Y_0(N) (\C) 
	}, 	
		\xymatrix{
	H^0(X_0(N)_\C,\Omega^1) \ar[r]^-{(w_d^{(N)})^*} \ar[d]_{\psi_N^{-1}} & H^0(X_0(N)_\C,\Omega^1) \ar[d]^{\psi_N^{-1}} \\
		S_2(\Gamma_0(N)) \ar[r]^-{[W_d^{(N)}]} & S_2(\Gamma_0(N))
	}
    \end{gathered}
	\end{equation}
\end{prop}

\begin{proof}
	By definition, it is enough to prove it for $d = Q$ a Hall divisor of $N$. By Lemma \ref{lemactAL} and its proof, for any $\tau \in \Hcal$ and $\tau' = W_Q \tau$, the kernel of the isogeny $\varphi\colon E_\tau \ra E_{\tau'}$ given by $z \mt \frac{Qz}{Nc \tau + Qd} = \frac{z}{(N/Q)\tau + d}$ is of order $Q$, and in this particular case the subgroup generated by $1/Q$. Furthermore, it can also be checked immediately that  $\varphi(\langle 1/N \rangle) = \langle 1/(N/Q) \rangle$ and $\varphi(\langle \tau/Q \rangle)= \langle 1/Q \rangle$.
Since $Q$ and $N/Q$ are coprime, this implies that the pair $(E_{\tau'}, \langle 1/N \rangle)$ can be seen as $(E_\tau/ \langle 1/Q \rangle, \varphi(E_\tau[Q] + \langle 1 /N \rangle))$.

As before for the $A_n$, this implies that for $w_Q$, the action of $(w_Q^{(N)})^* :H^0(X_0(N),\Omega^1) \ra H^0(X_0(N),\Omega^1)$ seen on modular forms is through the operator $[W_Q^{(N)}]$. Indeed, the pullback by $(W_Q^{(N)})^*$ of $f(\tau) d\tau$ is 
\begin{equation}
f(W_Q^{(N)} \tau) (W_{Q}^{(N)})'(\tau) d\tau = f(W_Q^{(N)} \tau)  \frac{Q}{(Nc\tau + Qd)^2} \tau = f_{[W_Q]}(\tau) d \tau. \qedhere
\end{equation}
\end{proof}

\subsection{The main lemma of Atkin--Lehner}
We begin by fixing some notation for subgroups of Atkin--Lehner operators.
\begin{defi}
\label{defnotationAL}
 For any level $N \geq 1$: 
 
 \begin{itemize}
  \item Let $\Wcal(N)$ be the subgroup of $\Aut(X_0(N))$ generated by the Atkin--Lehner involutions of level $N$ (also acting on modular forms by the previous results).
  
  \item  For any $M|N$, and any Hall divisor $Q$ of $N$, define $Q_M \| M$ as the Hall divisor of $M$ with the same prime factors as $\gcd(Q,M)$. Then, the map
  \begin{equation}
  \fonction{\varphi_{N,M}}{\Wcal(N)}{\Wcal(M)}{w_{Q}^{(N)}}{w_{Q_M}^{(M)}}
  \end{equation}
  is a surjective group morphism between $\Wcal(N)$ and $\Wcal(M)$ (see \cite[Lemma 9]{AtkinLehner70} for details on these involutions), with kernel generated by the $w_p^{(N)}$ for the prime numbers $p$ dividing $N$ but not $M$.
    
\item For any $M|N$, the group $\Wcal(N)$ acts on the positive divisors $d$ of $N/M$ by $ w_Q^{(N)} \cdot d \colonequals d'_Q$ where for each prime $p|(N/M)$, 
 \begin{equation}
  v_p(d'_Q) = \begin{cases}
  	v_p(N/M) - v_p(d) &  \textrm{if  } p | Q, \\
	               v_p(d) &  \textrm{otherwise.}
	              \end{cases}
 \end{equation}
 In particular, it permutes the Hall divisors of $N/M$.
 \end{itemize}

\end{defi}

\begin{lem}
\label{lem:atkinlehnerandactionNM}
    With the previous notations, for every $w \in \Wcal(N)$ and every Hall divisor $d$ of $N/M$: 
    \begin{equation}
     \label{eqvarphiandactionofWN}
w_{w \cdot d}^{(N)} = w \cdot w_d^{(N)} \cdot \varepsilon
    \end{equation}
    where $\varepsilon$ is a product of some $w_p^{(N)}$ restricted to prime factors  $p$ of $N$ such that $v_p(N) = v_p(M)$ (i.e.\ $p \nmid N/M$).
\end{lem}

\begin{proof}
Let us fix $w= w_Q^{(N)}$, $d \nmid (N/M)$ and $d' = w \cdot d$ which is by definition another Hall divisor of $N/M$.

    To obtain the equality of the Lemma, one simply needs to check that for every prime factor $p$ of $N/M$, $p$ divides $d'$ if and only if $p$ does not simultaneously divide (or not divide) $d$ and $Q$.

    This amounts to a case-by-case analysis:
    \begin{itemize}
        \item If $p|Q$, $v_p(d') = v_p(N/M) - v_p(d)$ so as $d$ and $d'$ are Hall divisors of $N/M$, exactly one of them is divisible by $p$ (so $w_p^{(N)}$ appears on both sides of \eqref{eqvarphiandactionofWN} if $p \nmid d$ and on no side if $p |d$).
        \item If $p \nmid Q$, $v_p(d')=v_p(d)$ so $w_p^{(N)}$ appears on both sides of \eqref{eqvarphiandactionofWN} if $p|d$ and on no side if $p \nmid d$.

    In all cases, writing each side of $\eqref{eqvarphiandactionofWN}$ as a product of distinct $w_p^{(N)}$, for any prime factor $p$ of $N/M$, $w_p^{(N)}$ appears either on both sides or on none of them, as claimed, which proves \eqref{eqvarphiandactionofWN}.        \qedhere
    \end{itemize}
\end{proof}
We now state a slightly more general version of Lemma 26 of \cite{AtkinLehner70}.

\begin{lem}
	\label[lemma]{mainlemmaAtkinLehner}

	Let $M$ and $N$ be positive integers with $M  |  N$ and $k \geq 1$ even.
	
	For any $d|(N/M)$, any $f \in S_k(\Gamma_0(M))$ and any $w \in \Wcal(N)$, 
	\begin{equation}
	 \label{eq:formulamainlemmaAtkinLehner}
	 	 (f_{|_k [A_d]})_{|_k [w]} = (f_{|_k [\varphi_{N,M}(w)]})_{|_k [A_{w \cdot d}]}.
	\end{equation}

\end{lem}
\begin{proof}
	 Let us here drop for clarity the subscript $k$, having fixed the weight.
	 
	Let $q$ be a prime dividing $N$ and $\alpha = v_q(N), \beta = v_q(N/M), \gamma = v_q(d)$. Define $d' = w_q^{(N)} \cdot d =  d q^{\beta - 2 \gamma}$, so that $v_q(d') = \beta - \gamma$ and $v_p(d') = v_p(d)$ for all other prime factors of $N/M$ (and $d'  |  (N/M)$ as well). By \cite[Lemma 26]{AtkinLehner70} translated with our notation,
	\begin{equation}  
	(f_{|[A_d]})_{|[W_q^{(N)}]} = (f_{|[W_{q_M}^{(M)}]})_{|[A_{d'}]} = (f_{|[W_{q_M}^{(M)}]})_{|[A_{w_q^{(N)} \cdot d}]}.
	\end{equation}
	Now, assume that \eqref{eq:formulamainlemmaAtkinLehner} holds for given $w,w' \in \Wcal(N)$ and every $d |(N/M)$. Let us show that it holds for $w w'$ and every $d|(N/M)$ as well.
	
	Then, for every $f \in S_k(\Gamma_0(M))$, 
	\begin{eqnarray}
	 \left(f_{|[A_d]}\right)_{[ww']} = ((f_{[A_d]})_{[w]})_{[w']} & = & (f_{[\varphi_{N,M}(w)})_{[A_{w \cdot d}]})_{[w']} \notag\\
	 & = & ((f_{[\varphi_{N,M}(w)})_{[\varphi_{N,M}(w')]})_{[A_{w' \cdot (w \cdot d)}]} \\ 
	 & =& (f_{[\varphi_{N,M}(ww')]})_{[A_{(ww')\cdot d}]} \notag
	\end{eqnarray}
	as $\Wcal(N)$ is commutative.

	Combining the case of $w = w_q^{(N)}$ with $q$ prime and the previous property, we obtain the lemma for all $w_Q^{(N)}$ in $\Wcal(N)$ by induction on the number of prime factors of $Q$.
\end{proof}

\begin{rem}
Lemma \ref{mainlemmaAtkinLehner} has the following consequences in particular cases, for $Q$ a Hall divisor of $N$ and $w = w_Q^{(N)}$: 
\begin{itemize}
    \item If $(Q,N/M)=1$, for any $d|(N/M)$, $w \cdot d = d$ so 
    \begin{equation}
     (f_{|_k[A_d]})_{|_k[w]} = (f_{|_k [\varphi_{N,M}(w)]})_{|_k {[A_{d}]}}.
    \end{equation}
In particular, if $f$ is an $\varepsilon$-eigenvector for $\varphi_{N,M}(w)$, $f_{|_k [A_d]}$ is an  $\varepsilon$-eigenvector for $w$.

	\item If $(Q,M)=1$, $w$ exchanges $f_{|_k [A_1]}$ and $f_{|_k [A_Q]}$.

	\item If $f$ is an $\varepsilon$-eigenvector for $w' \in \Wcal(M)$ and $\varphi_{N,M}(w) = w'$, $w$ exchanges $f_{|_k [A_1]}$ and $\varepsilon f_{|_k [A_{w \cdot 1}]}$.
\end{itemize}
\end{rem}

\begin{prop}
	\label{propdiagcommALandiNM}
 With $d,M,N$ positive integers such that $M|N$ and $d|(N/M)$, and $w \in \Wcal(N)$, the following diagram commutes
 \begin{equation}
 \begin{gathered}
  \xymatrix{
  X_0(N) \ar[d]_{i_{N,M}^{(w \cdot d)}} \ar[rr]^{w} & & X_0(N) \ar[d]^{i_{N,M}^{(d)}}\ar[d] \\X_0(M) \ar[rr]^{\varphi_{N,M}(w)} & & X_0(M) \\
  }
  \end{gathered}
   \end{equation}
\end{prop}

\begin{proof}
By scalar extension, it is enough to prove the commutativity on the complex modular curves.  Comparing the pullbacks by $f:=i_{N,M}^{(d)} \circ w$ and $g:=\varphi_{N,M}(w) \circ i_{N,M}^{(w \cdot d)}$ leads to the respective maps $f \mt (f_{[A_d]})_{[w]}$ and $f \mt (f_{[\varphi_{N,M}(w)})_{[A_{w \cdot d}]}$ on forms of weight 2 by Propositions \ref{propdiagramiNM} and \ref{propALalgebraically}. Those operators on modular forms are equal by Lemma \ref{mainlemmaAtkinLehner}, so $f_*$ and $g_*$ induce the same endomorphisms $\Cot_0(J_0(M)) \rightarrow \Cot_0(J_0(N))$ by Lemma \ref{lemcotangentpullback}, hence $f_*-g_* = 0$ since the tangent map of $f_*-g_*$ at 0 is 0 and $J_0(N)$ is a connected algebraic group. Therefore, $f_*=g_*$ and by restriction to $X_0(N)$ via the Abel--Jacobi map (if the genus of $X_0(M)$ is positive), we obtain $f=g$.
 
 If the genus of $X_0(M)$ is 0, let us choose a large prime $\ell \nmid N$ such that $X_0(M \ell)$ has positive genus, define for $w = w_Q^{(N)} \in \Wcal(N)$ its lift $\widetilde{w} := w_Q^{(N \ell)}$ to $\Wcal(N\ell)$. We can then use the following diagram
 \begin{equation}
 \begin{gathered}
       \xymatrix{
X_0(N \ell) \ar@{->}[rrr]^{\widetilde{w}} \ar@{->}[ddd]_{{i_{N \ell,M\ell}^{(\widetilde{w} \cdot d)}}} \ar@{->}[rd]^{{i_{N \ell,N}^{(1)}}} &  &  & X_0(N \ell) \ar@{->}[ddd]^{{i_{N \ell,M\ell}^{(d)}}} \ar@{->}[ld]_{{i_{N \ell,N}^{(1)}}} \\
 & X_0(N) \ar@{->}[d]_{{i_{N,M}^{\widetilde{w} \cdot(d)}}} \ar@{->}[r]^{w} & X_0(N) \ar@{->}[d]^{{i_{N,M}^{(d)}}} &  \\
 & X_0(M) \ar@{->}[r]_{{\varphi_{N,M}(w)}} & X_0(M) &  \\
X_0(M \ell) \ar@{->}[rrr]_{{\varphi_{N\ell,M\ell}(\widetilde{w})}} \ar@{->}[ru]^{{i_{M\ell,M}^{(1)}}} &  &  & X_0(M \ell) \ar@{->}[lu]_{{i_{M \ell,M}^{(1)}}}
}
\end{gathered}
 \end{equation}
The external square commutes by the previous argument, and each side of the trapezoids commute (it is a straightforward verification using the modular interpretation of the maps and that nothing acts on the $\ell$-part among those maps), so the inner square commutes by using that $i_{N\ell,N}^{(1)}$ is surjective. 
\end{proof}

\begin{rem}
 It is possible to prove directly the commutativity of the diagram of the proposition by the modular interpretation, but it is in our opinion quite technical and not very enlightening to do so, so we preferred to use the action of modular forms based on \cite{AtkinLehner70}, which will also be useful to us later. 
\end{rem}

Finally, let us define the main protagonist of our story.
\begin{defi}[Star quotients]
\label{defi:starquotients}
	Let $N \geq 1$ be an integer. The \emph{star quotient of $X_0(N)$}, denoted by $X_0(N)^*$, is the quotient of $X_0(N)$ by the full Atkin--Lehner group $\Wcal(N)$.

We also denote the \emph{plus quotient} $X_0(N)^+ := X_0(N) / \langle w_N^{(N)} \rangle$, so that we have natural quotient maps $X_0(N) \ra X_0(N)^+$ and $X_0(N) \ra X_0(N)^*$ of respective degrees 2 and $2^{\omega(N)}$.
\end{defi}

We now prove the technical lemmas needed to reduce to working with almost squarefree \eqref{eqshapeN} levels $N$. As outlined in section \ref{sec:overview}, these levels are desirable to apply easily Mazur's formal immersion method. We then show how to apply Mazur's method to obtain proof of integrality of the $j$-invariant in this case.

\begin{lem}
\label{lem:degeneneracyandwforq2}
	Let $N \in \Z_{\geq 1}$ and $q$ a prime number such that $q^2  |  N$. For any prime $p  |  N$, the diagram 
	\begin{equation}
	    \begin{gathered}   
	\xymatrix{X_0(N) \ar[r]^-{w_p^{(N)}} \ar[d]_{i_{N,M}^{(q)}} & X_0(N) \ar[d]^{i_{N,M}^{(q)}} \\
		X_0(N/q^2) \ar[r]^-{w_p^{(M)}} & X_0(N/q^2) }
        	    \end{gathered}
	\end{equation}
	is commutative, where by definition above $w_p^{(M)}$ is the identity if $p=q$ and $q^2 \| N$.
\end{lem}

\begin{proof} 
This is  an immediate corollary of 
	 Proposition \ref{propdiagcommALandiNM} (which can also be proven directly).
\end{proof}
\begin{prop}
 \label{prop:constructionofdegeneracymappsi}
		For any integer $N \geq 1$ and any prime $q$ such that $q^2|N$, the morphism $i_{N,N/q^2}^{(q)}$ factors through the Atkin--Lehner quotients and defines a morphism \begin{equation}\psi_{N,N/q^2}\colon X_0(N)^* \ra X_0(N/q^2)^*.\end{equation} This morphism sends cusps to cusps and for any point $P \in Y_0(N)^*$ its image is in the same isogeny class as $P$.
	\end{prop}
	
	\begin{proof}
		This is immediate with the lemma. If $P = w_Q^{(N)} P'$ for some points $P,P' \in X_0(N)$ and $Q \| N$,  then $i_{N,N/q^2}^{(q)}(P) = w_{Q'}^{(M)} (i_{N,M}^{(q)}(P'))$ where $Q' = Q$ if $q\nmid Q$ and $Q' = Q/q^2$ if $q |Q$. In particular, $i_{N,N/q^2}^{(q)}(P)$ and $i_{N,N/q^2}^{(q)}(P')$ are in the same Atkin--Lehner orbit so both cusps if one of them is, and otherwise the underlying elliptic curves are isogenous.
	\end{proof}
	
	\begin{cor}
		If $P \in X_0(\widetilde{N})^* (\Q)$ is a nontrivial point where $\widetilde{N}$ has a square factor, there exists an almost squarefree  divisor $N$ of $\widetilde{N}$ such that $N$ is square-below $\widetilde{N}$ and $P$ is isogenous to a nontrivial point of $X_0(N)^* (\Q)$.
	\end{cor}
	This allows us to assume that there is only one square factor in the level, which makes things easier for Mazur's argument.
    \begin{defi}
    \label{defi:liftofPstar}
        For any $N \geq 1$ and any non-cuspidal $P^* \in X_0(N)^* (\Q)$, there is a poly\-quadratic field $K$ of degree at most $2^{\omega(N)}$ which is the minimal field of definition of every lift $P \in X_0(N)$ above $P^*$. 
        Therefore, by a small abuse of notation, a \emph{lift $P \in X_0(N)(K)$ of $P^*$} will denote the datum of both $P$ and such a polyquadratic field $K$ (of degree as small as possible).
    \end{defi}

    \begin{proof}
    Let us choose $P \in X_0(N)(\Qb)$ whose image in $X_0(N)^*$ is $P^*$. By definition, this implies that $\GalQ \cdot P \subset \Wcal(N) \cdot P$. 

    Defining $H = \Stab_{\Wcal(N)}(P)$, $G := \Wcal(N)/H$ acts simply transitively on $\Wcal(N) \cdot P = G \cdot P$, so for every $\sigma \in \GalQ$, there exists a unique $g_\sigma \in G$ such that 
    \begin{equation}
\sigma(P) = g_\sigma \cdot P.
    \end{equation}
    The function $\sigma \mapsto g_\sigma$ is a group morphism $\rho: \GalQ \rightarrow G$ and the fixed field of its kernel is some finite Galois extension $K$ with Galois group $G' = \Im(\rho) \subset G$ by definition. As $G$ is isomorphic to $(\Z/2\Z)^r$ for some $r \leq \omega(N)$, $G'$ also is for some $r' \leq r$, therefore the field of definition of $P$ is a Galois extension of $\Q$ with Galois group $G'$, hence a polyquadratic field of degree $2^{r'}$, and the same holds for all Atkin--Lehner conjugates of $P$ by commutativity of $\Wcal(N)$.
    \end{proof}

\begin{rem}
    One can see  $P^*$ comes from a rational point in a quotient $X_0(N)/W$ with $W \subsetneq \Wcal(N)$ if and only if $G'$ is of smaller cardinal than $\Wcal(N)$ (which means that either $H \neq \{1\}$ and $P$ has CM, or $H=\{1\}$ but $G' \subsetneq \Wcal(N)$).
\end{rem}

\begin{prop}
\label{prop:fieldsofdefsquarebelow}
    For any integer $N \geq 1$, any prime $q$ with $q^2|N$, any $P^* \in X_0(N)^*(\Q)$ noncuspidal non-CM and any lift $P \in X_0(N)(K)$ of $P^*$, the field of definition of $P' = i_{N,N/q^2}^{(q)}(P)$ is $K' \subset K$ with $[K:K'] \leq 2$ and $K'=K$ if $q^3|N$.
\end{prop}

\begin{proof}
    As $P$ is a non-CM and non-cuspidal point, $\Stab_{\Wcal(N)} (P)$ is trivial (and similarly for $P'$). We also have $K' \subset K$ because $i_{N,N/q^2}^{(q)}$ is defined over $\Q$. By the previous proof, there is a subgroup $H$ of $\Wcal(N)$ of order $[K:\Q]$ such that 
    \[
\GalQ \cdot P = H \cdot P.
    \]
   Applying $i_{N,N/q^2}^{(q)}$, we get
    \[
\GalQ \cdot P'  = i_{N,N/q^2}^{(q)} (\GalQ \cdot P) = i_{N,N/q^2}^{(q)} (H \cdot P).
    \]
    Now, by Lemma \ref{lem:degeneneracyandwforq2}, for every $Q \|N$, $i_{N,N/q^2}^{(q)} \circ w_Q^{(N)} = w_{Q'}^{(N/q^2)} (i_{N,N/q^2}^{(q)}(P))$ with $Q' = \gcd(Q,N/q^2)$. We thus obtain 
    \[
\GalQ \cdot P'  = H' \cdot P'
    \]
    with $H' \subset \Wcal(N/q^2)$ the image of $H$ via $\varphi_{N,N/q^2} : w_Q^{(N)} \mapsto w_{Q'}^{(N/q^2)}$.

    If $q^3|N$, this group morphism is injective hence $|\GalQ \cdot P'| = |H'| = |H|$, and thus $K'=K$. If $q^2 \| N$, the kernel of $\varphi_{N,N/q^2}$ is $\langle w_{q^2}^{(N)} \rangle$, and $H' = H /H \cap \langle w_{q^2}^{(N)} \rangle$, therefore $|H|/|H'| \leq 2$ and $[K:K'] \leq 2$.
\end{proof}
\subsection{Cusps and Galois action}\label{sec:number-of-cusps-and-galois-action}

Before describing the behavior of cusps on $X_0(N)$ with respect to diverse actions and maps, to avoid any ambiguity we  will use the following.

\begin{defi}[Cusps]
 For any $a/b \in \P^1(\Q)$ (with $a,b \in \Z$ coprime, allowing $a=1,b=0$ for $\infty)$, we define
 \begin{equation}
  \Ccal_N(a/b) \colonequals \Gamma_0(N) \frac{a}{b}
 \end{equation}
the cusp of $X_0(N)$ associated to $a/b$ (dropping the index $N$ if the level is clear from context).

The \emph{width of the cusp} $\c$, denoted by $\omega_{\c}^{(N)}$, is the index of the stabilizer of $a/b$ in $\Gamma_0(N)$ as a subgroup of the stabilizer of $a/b$ in $\SL_2(\Z)$ (independent of the choice of $a/b$ representing $\Ccal$).

\end{defi}

\begin{prop}[Basic properties of cusps]
\label{propbasecusps}
\hspace*{\fill}

\begin{enumerate}[(i)]
 \item The map $(a:b) \mapsto \Ccal_N(a/b)$ on $\P^1(\Q)$ induces a bijection between $B_0(N)\backslash \Scal(N)$  and the cusps of $X_0(N)$, where $\Scal(N)$ is the set of pairs of order exactly $N$ in $(\Z/N\Z)^2$  (seen as column vectors) and $B_0(N)$ the (Borel) subgroup of upper-triangular matrices in $\SL_2(\Z/N\Z)$ acting canonically on those vectors by left multiplication. 
 
 \item Two elements $a/b, a'/b' \in \P^1(\Q)$ define the same cusp on $X_0(N)$ if and only if there is $d|N$ such that $ \gcd(b,N) = \gcd(b',N) = d$ and $a(b/d) \equiv a'(b'/d) \mod \gcd(d,N/d)$.
  
  \item 
  Consequently, a set of representatives of $B_0(N) \backslash \Scal(N)$ (and thus of cusps of $X_0(N)$ via $\Ccal_N$) is given by
$\Rcal_N := \bigsqcup_{b|N} \frac{\Rcal_{N,b}}{b}$, where $\Rcal_{N,b}$ is a set of integers prime to $N$ such that reduction modulo $\gcd(b,N/b)$ defines a bijection from $\Rcal_{N,b}$ to  $(\Z/\gcd(b,N/b)\Z)^\times$ (in particular, $\Card \Rcal_{N,b} = \varphi (\gcd(b,N/b))$). If $b$ is a Hall divisor of $N$, we choose $\Rcal_{N,b} = \{1\}$.

Consequently, there are in total 
\begin{equation}
\sum_{b|N} \varphi((b,N/b))
\end{equation}
cusps on $X_0(N)$.
  \item The Galois action of $\Gal(\Qb/\Q)$ on the cusps is seen via the bijection in $(i)$ as the action on $B_0(N) \backslash \Scal(N)$ given by 
  \begin{equation}
  \sigma \cdot \begin{pmatrix} a \\ b \end{pmatrix} := \begin{pmatrix} \chi(\sigma) a \\ b \end{pmatrix}.
  \end{equation}
  with $\chi$ the cyclotomic character of order $N$ on $\Gal(\Qb/\Q)$.
  Consequently, two cusps $c = \Ccal_N(a/b)$ and $c' = \Ccal_N(a'/b')$ with $(a,b),(a',b') \in \Rcal_N$ are in the same $\GalQ$-orbit if and only if $b=b'$, which makes up exactly $\tau(N)$ Galois orbits of cusps where $\tau$ is the divisor counting function. Furthermore, the rational cusps of $X_0(N)$ are exactly the $\Ccal_N(1/Q)$, with $Q$ a Hall divisor of $N$ (e.g.\  $\infty$ is $\Ccal_N(1/N)$), and more generally the $\Gal(\Qb/\Q)$-orbit of the cusp $\Ccal_N(a/b)$ with $b|N$ is in bijection with $\Rcal_{N,b}$, therefore it has cardinality $\varphi((b,N/b))$.
  
\item  The width of $\Ccal = \Ccal_N(a/b)$ with $b|N$ is exactly $N/\gcd(N,b^2)$, in particular it is 1 if $N|b^2$.  Furthermore, for $\gamma \in \SL_2(\Z)$ such that $\gamma \cdot \infty = a/b$, the function 
  	\begin{equation}
  	      	      	q_\c^{(N)} : z \mapsto e^{\frac{2 i \pi \gamma^{-1} z}{ \omega_\c^{(N)}}}
  	\end{equation}
  	induces on a neighborhood of $\c$ a uniformizer of $X_0(N)(\C)$ at $\c$, independent of the choice of $\gamma$ up to multiplication by a $\omega_\c^{(N)}$-th root of unity.
\end{enumerate}
\end{prop}

\begin{exe}
    For item $(iii)$, we cannot simply choose $\Rcal_{N,b}$ to be the set of integers between 1 and $\gcd(b,N/b)$ prime to this gcd, as they might not be prime to $b$ itself. For example, taking $N = 108$, $a =5, b= 12$, the reduction process would suggest replacing 5 by its remainder mod $\gcd (12,108/12) = 3$ which would be 2, but then 2 is not prime to 12. In that precise case, a correct choice of $\Rcal_{108,12}$ is thus $\{1,5\}$. 

   Notice that this problem does not appear when $N$ is squarefree (because then $\gcd(b,N/b)$ is always 1) nor when $N = p^\alpha$ for some prime $p$ because any integer prime to $p$ is automatically prime to $N$.
\end{exe}

\begin{proof}
    For $(i)$, by a straightforward computation, $a/b$ and $a'/b'$ are in the same orbit for $\Gamma(N)$ if and only if the elements $(a,b)$ and $(a',b')$ mod $N$ (which belong to $\Scal(N)$ by coprimality) are equal up to $\pm 1$. Determining if $a/b$ and $a'/b'$ are in the same orbit for $\Gamma_0(N)$ thus depends only on $(a,b)$ and $(a',b')$ mod $N$, and $(i)$ follows as $B_0(N)$ is exactly $\Gamma_0(N)$ mod $N$.
    
    For $(ii)$, multiplication by a matrix of $B_0(N)$ does not change the order of the second term of $\overline{(a,b)} \in \Scal(N)$, and this order is exactly $N/\operatorname{gcd}(b,N)$. Assuming $a/b$ and $a'/b'$ define the same cusp, we thus have $\gcd(b,N) = \gcd(b',N)$ and writing it $d$, up to multiplication by a diagonal matrix with lower-right term $\overline{d/b}$ thus upper-left term $\overline{b/d}$, we can reduce to $b=d$ (and similarly $b'=d$), which multiplies $a$ by $\overline{a_1} = \overline{b/d} \overline{a}$ and similarly $\overline{a'_1} =\overline{b'/d} \overline{a'}$, we have to check when $\overline{(a_1,d)}$ and $\overline{(a'_1,d)}$ are in the same orbit for $B_0(N)$, and this holds if and only if $a_1 \equiv a'_1 \mod \gcd(d,N/d)$.

    For $(iii)$, this is a direct consequence of $(ii)$ (in particular we can choose for any cusp a representative $(a,b)$ such that $a$ is prime to $N$). For $(iv)$, the Galois action is given by the description of the Galois action on the cusps given in \cite[Chapter VI.5]{DeligneRapoport}. The consequence on Galois orbits is immediate when we start with $(a,b)$ and $a$ coprime to $N$ as we have just proven possible (so that its multiples by the image of $\chi$ go through all $(\Z/N\Z)^\times$.
    
    Finally, for $(v)$, let $\gamma = \begin{pmatrix}  a & c \\ b & d \end{pmatrix} \in \SL_2(\Z)$ such that $\gamma \cdot \infty = a/b$ and $\Gamma_\infty = \pm \left\langle \begin{pmatrix} 1 & 1 \\ 0 & 1 \end{pmatrix} \right\rangle$ the stabilizer of $\infty$. By definition, the width $\omega_\c^{(N)}$ is the index of $\Gamma_\infty \bigcap \gamma \Gamma_0(N) \gamma^{-1} $ in $\Gamma_\infty$. For $k \in \Z$, 
    	\begin{equation}
    	\gamma \begin{pmatrix} 1 & k \\ 0 & 1 \end{pmatrix} \gamma^{-1} = \begin{pmatrix} \ast & \ast \\ -b^2 k & \ast \end{pmatrix}
    	\end{equation}
    	so this belongs to $\Gamma_0(N)$ if and only if $N  |  kb^2$, hence $\omega_\c^{(N)} = N/\operatorname{gcd}(N,b^2)$,
    	    	
    	Next, see \cite[Figure 2.7]{DiamondShurman} for the fact that $q_\c^{(N)}$ is a uniformizer at $\c$. If we choose another $\gamma' \in \SL_2(\Z)$ such that $\gamma' \infty = a/b$, $\gamma^{-1} \gamma'$ fixes $\infty$ so we can write $\gamma' = \gamma (\pm T^k)$ for $T = \begin{pmatrix} 1 & 1 \\ 0 & 1 \end{pmatrix}$ and some $k \in \Z$, and then for all $z \in \Hcal$,
    	\begin{equation}
    	(\gamma')^{-1} z = T^{-k} (\gamma^{-1}(z)) = \gamma^{-1}(z)-k,
    	\end{equation}
    	hence the result.
\end{proof}

\begin{lem}[Images of cusps by certain degeneracy maps]
\label{lem:imagecuspsdegeneracymaps}
Let $N \geq 1$ a level, $M$ a Hall divisor of $N$ and $d$ a Hall divisor of $N/M$. 
\begin{itemize}
    \item [$(i)$] For any cusp $\Ccal_N(a/b)$ with $(a,b) \in \Rcal_N$, $i_{N,M}^{(d)} (\Ccal_N(a/b))= \Ccal_M(a'/b')$ with the unique choice of $(a',b') \in \Rcal_M$ given by 
    \begin{eqnarray*}
       b' & = & \gcd(b,M) \\
       a' & \equiv & \frac{d\gcd(b,N/(Md))}{\gcd(b,d)} a \mod \gcd(b,N/b,M)=\gcd(b',M/b'). 
    \end{eqnarray*}
    \item[$(ii)$] The map $i_{N,M}^{(1)} \times i_{N,N/M}^{(1)}$ induces a bijection between cusps of $X_0(N)$ and pairs of cusps of $X_0(M)$ and $X_0(N/M)$, given for any $(a,b) \in \Rcal_N$ by 
    \[
\Ccal_N\left( \frac{a}{b} \right) \longmapsto \left( \Ccal_M\left( \frac{a \gcd(b,N/M)}{\gcd(b,M)} \right), \Ccal_{N/M} \left( \frac{a \gcd(b,M)}{ \gcd(b,N/M)} \right) \right)
    \]
Conversely, given two cusps $\Ccal_M(a_M/b_M)$ and $\Ccal_N(a_{N/M}/b_{N/M})$ with $(a_M,b_M) \in \Rcal_M$ and $(a_{N/M},b_{N/M} \in \Rcal_{N/M})$, their preimage via this product map is $\Ccal_N(a/b)$ with $(a,b) \in \Rcal_N$ such that $b = b_M b_{N/M}$ and $a \equiv a_M b_{N/M}^{-1} \mod \gcd(b_M,M/b_M)$ and $a \equiv a_{N/M} b_{M}^{-1} \mod \gcd(b_{N/M},N/(M b_{N/M})$. 
    
\end{itemize}
    
\end{lem}

\begin{rem}
    Notice that in part $(i)$, we might prefer $a'$ to not be chosen equal to the right-hand side product, as the latter is not coprime to $N$ (although it is always coprime with $M$): this is why we only fix the congruence condition.
\end{rem}

\begin{proof}
    \begin{itemize}
        \item[]
        \item[$(i)$] Start with $A_d \cdot (a/b) = (da / b)$, which simplifies to $\frac{d/\gcd(b,d) a}{b/\gcd(b,d)}$ and this fraction is reduced. As $d \| N/M \| N$, one can write $b = b_d b_M b''$ with $b_d = \gcd(b,d)$ and $b_M = \gcd(b,M)$, and then $b'' = \gcd(b, (N/Md))$. By Proposition \ref{propbasecusps} applied in level $M$, we get 
        \[
\Ccal_M \left( \frac{(d/\gcd(b,d)) a}{b/\gcd(b,d)} \right) = \Ccal_M \left( \frac{(d/\gcd(b,d)) a}{b_M b''} \right) = \left( \frac{(d/\gcd(b,d)) a b''}{b_M} \right) 
        \]
        as claimed.

        \item[$(ii)$] is a direct consequence of $(i)$ for the case $d=1$: in that case $a' = \gcd(b,N/M) a$ and replace $M$ by $N/M$ when degenerating in $X_0(N/M)$.
    \end{itemize}
\end{proof}

We can now describe the action of the Atkin--Lehner operators on cusps.

\begin{prop}
\label[proposition]{propactALoncusps}
	Let $N \geq 1$ be an integer and $Q \| N$.

    For every cusp $\c = \Ccal_N(a/b)$ of $X_0(N)$ with $(a,b) \in \Rcal_N$, $w_Q^{(N)}(\Ccal) = \Ccal_N(a'/b')$ where $(a',b')$ is the unique element of $\Rcal_N$ given by :

    \begin{itemize}
        \item[$(i)$] for every prime $p |N$, $v_p(b) = v_p(b')$ if $p \nmid Q$ and $v_p(b) = v_p(N) - v_p(b')$ if $p|Q$.
        \item[$(ii)$] $a' \equiv a \mod \gcd(b,N/b,N/Q)$ and $a' \equiv - a \mod \gcd(b,N/b,Q)$.
    \end{itemize}

 Consequently, two $\GalQ$-orbits of cusps of $X_0(N)$ (thus associated to divisors $b$ and $b'$ of $N$ by Proposition \ref{propbasecusps} $(iv)$) are related by an Atkin--Lehner involution if and only if for every prime $p|N$, $v_p(b) = v_p(b')$ or $v_p(b) + v_p(b') = v_p(N)$. The number of $\Wcal(N)$-orbits of $\GalQ$-orbits of cusps of $X_0(N)$ (so of $\GalQ$-orbits of cusps of $X_0(N)^\ast$) is thus 
 \begin{equation}
\prod_{\substack{p|N \\ p \textrm{ prime}}} \left\lceil \frac{v_p(N)+1}{2} \right\rceil.
 \end{equation}
Among those, the rational cusps of $X_0(N)^*$: 
 (written as $\Wcal(N)$-orbits of cusps of $X_0(N)$) are are exactly the following: 
\begin{itemize}
    \item The cusp made up by the $\Ccal_N(1/Q)$ (with $Q  \| N$) which are already rational cusps in $X_0(N)$, i.e.\ $\Wcal(N) \cdot \infty$.
    \item (If $4|N$) The cusp $\Wcal(N) \cdot \Ccal_N(1/2)$, made up with cusps already rational in $X_0(N)$.
    \item (If $9 \| N$) The cusp $\Wcal(N) \cdot \Ccal_N(1/3)$, made up with cusps defined over $\Q(\zeta_3)$ in $X_0(N)$.
    \item (If $16  \| N$) The cusp  $\Wcal(N) \cdot \Ccal_N(1/4)$, made up with cusps defined over $\Q(i)$ in $X_0(N)$.
    \item (If $4 |N$ and $9 \| N$) The cusp $\Wcal(N) \cdot \Ccal_N(1/6)$, made up with cusps defined over $\Q(\zeta_3)$ in $X_0(N)$.
    \item (If $144 \| N$) The cusp $\Wcal(N) \cdot \Ccal_N(1/12)$, made up with cusps defined over $\Q(\zeta_6)$ in $X_0(N)$.
\end{itemize}
\end{prop}

\begin{proof}

Let us fix $(a,b) \in \Rcal_N$, $\Ccal = \Ccal_N(a/b)$ and $Q \| N$, define $(a^{(Q)}, b^{(Q)}) \in \Rcal_N$ such that 
\[
\Ccal^{(Q)} := w_Q^{(N)}(\Ccal) = \Ccal_N(a^{(Q)}/b^{(Q)}).
\]
By Lemma \ref{lem:imagecuspsdegeneracymaps}, it is enough to identify the images of  $\Ccal^{(Q)}$ by degeneracy in $X_0(Q)$ and $X_0(N/Q)$. By Proposition \ref{propdiagcommALandiNM}, we have 
\begin{eqnarray*}
i_{N,N/Q}^{(1)} \circ w_Q^{(N)} & = & i_{N,N/Q}^{(N/Q)} \\
i_{N,Q}^{(1)} \circ w_Q^{(N)} & = & w_Q^{(Q)} \circ i_{N,Q}^{(1)}
\end{eqnarray*}
and we deduce by the lemma that 
\[
i_{N,N/Q}^{(1)} (\Ccal^{(Q)}) = \Ccal_{N/Q} \left( \frac{a_{N/Q}}{b_{N/Q}} \right),
\]
with admissible choice for representatives
\[
b_{N/Q} = \gcd(b,N/Q), \quad a_{N/Q} =  \frac{Q}{\gcd(b,Q)} a.
\]
For $w_Q^{(Q)}$ one can use the matrix $\begin{pmatrix} 0 & -1 \\ Q & 0 \end{pmatrix}$, and this gives 
\begin{eqnarray*}
i_{N,Q}^{(1)} (\Ccal^{(Q)}) = w_Q^{(Q)} \bigl( i_{N,Q}^{(1)} (\Ccal) \bigr) & = & w_Q^{(Q)} \Ccal_Q \left( \frac{a \gcd(b,N/Q)}{\gcd(b,Q)} \right) \\
& = & \Ccal_Q \left( - \frac{1}{a \gcd(b,N/Q) Q/\gcd(b,Q)}\right) \\ & = & \Ccal_Q (a_Q/b_Q), \quad a_Q = - a \gcd(b,N/Q),\quad b_Q =  \frac{Q}{\gcd(b,Q)}.
\end{eqnarray*}

We thus deduce by the lemma that 
\begin{eqnarray*}
    \gcd(b^{(Q)},N/Q)  & = & b_{N/Q} = \gcd(b,N/Q), \\
    \gcd(b^{(Q)},Q) & = & b_Q = Q/\gcd(b,Q), \\
    a_{N/Q} & = & a^{(Q)} \gcd(b^{(Q)}, Q) \mod (b^{(Q)}, N/b^{(Q)},N/Q), \\
    a_Q & = & a^{(Q)} \gcd(( b^{(Q)},N/Q) \mod (b^{(Q)},N/b^{(Q)},Q).
\end{eqnarray*}

Combining the first two lines, $b^{(Q)} = b_Q b_{N/Q}$ with the claimed properties. Combining the last two with the found values of $a_Q$ and $a_{N/Q}$, an admissible $a^{(Q)}$ is any choice which is $a$ mod $N/Q$ and $-a$ mod $Q$, as claimed.
\end{proof}

\begin{cor}
\label{cor:fixedcuspsbyAtkinLehner}
    With the previous notation, for  any cusp $\Ccal = \Ccal_N(a/b)$ with $(a,b) \in \Rcal_N$, the stabilizer of $\Ccal$ in $\Wcal(N)$ is trivial, except when $4$ is a Hall divisor of $N$ and then for each cusp with $v_2(b) = 1$, the stabilizer of $\Ccal$ is $\langle w_4^{(N)} \rangle $ (the stabilizers of all other cusps being trivial).
\end{cor}

\begin{proof}
    Applying the formula of the Proposition with $w_Q^{(N)} (\Ccal) = \Ccal$, we obtain first that for each prime $p |Q$ , $2 v_p(b) = v_p(N) = v_p(Q)$ so $Q$ is a square number and $\gcd(b,Q) = \sqrt{Q}$, and second that $a \equiv - a \mod \gcd(b,N/b,Q)$. As $a$ is invertible modulo $\gcd(b,N/b)$ this implies $2 = 0$ modulo $\gcd(b,N/b,Q)$ so this gcd is 1 or 2. The previous condition on the primes dividing $Q$ thus imposes $Q = 4$ and $v_2(b) = 1$. Checking the converse ($w_4^{(N)}$ stabilising these cusps) is immediate.
\end{proof}

\subsection{Cusps and degeneracy maps}
\label{subseccuspsdegeneracymaps}
\begin{prop}[Ramification of degeneracy maps at cusps]
\label{propramdegmaps}

For any levels $M  |  N$, for any $d | N/M$ and any cusp $\c = \Ccal_N(a/b)$ of $X_0(N)$ (with $(a,b) \in \Rcal_N$), denoting by $\Ccal' = i_{N,M}^{(d)}(\Ccal)$ in $X_0(M)$, for choices of uniformisers $q_\Ccal$ and $q_{\Ccal'}$ as in Proposition \ref{propbasecusps}, 
\[
(i_{N,M}^{(d)})^* (q_{\Ccal'}) = \zeta q_\Ccal^{\left(\frac{\gcd(b,d)^2 \omega_\Ccal^{(N)}}{d \omega_{\Ccal'}^{(M)}}\right)}
\]
for some $N$-th root of unity $\zeta$. Consequently:
	\begin{itemize} 
	 \item [$(i)$] The ramification degree of the map $i_{N,M}^{(1)}$ at $\c$ is the ratio $\omega_\c^{(N)}/\omega_{\c'}^{(M)}$ of the widths, hence $i_{N,M}^{(1)}$ is unramified at $\c$ if and only if for every prime factor $p$ of $N$ such that $v_p(N/M)>0$, $2 v_p(b) \geq v_p(N)$ (in particular, it is unramified at every width 1 cusp $\c$ of $X_0(N)$, see \ref{propbasecusps}\,$(v)$).
	\item[$(ii)$] If $\Ccal_N(a/b)$ is of width one, for $d|(N/M)$ $i_{N,M}^{(d)}$ is unramified at $\c$ if and only if $d \| (N/M)$ and for every prime $p|d$,$v_p(b) = v_p(N)/2$.
	\end{itemize}
 
\end{prop}

\begin{proof}

Define $d'= d/\gcd(b,d)$ and $b' = b/\gcd(b,d)$ so that $d'a$ and $b'$ are coprime and $\Ccal' = \Ccal_M (d'a/b')$ (notice that what follows does not assume that $(d'a,b')$ belongs to $\Rcal_N$). Choose $\gamma_\Ccal = \begin{pmatrix} a & e \\ b & f \end{pmatrix}$ and $\gamma_{\Ccal'} = \begin{pmatrix} d'a & g \\ b' & h \end{pmatrix}$ in $\SL_2(\Z)$ such that $\gamma_{\Ccal} \cdot \infty = a/b$ and $\gamma_{\Ccal'} \cdot \infty = (d'a)/b'$. Now, 

\begin{eqnarray}
 (\gamma_\c')^{-1} \circ A_d \circ \gamma_\c & = & \begin{pmatrix} hda - gb & hde - gf \\ -b' d a + d' a b & - b' d e + d' a f \end{pmatrix} \\
& = & \begin{pmatrix} \gcd(b,d) & h de - gf \\ 0 & d' \end{pmatrix} =: B \notag
\end{eqnarray}

Consequently, $\gamma_{\Ccal'}^{-1} \circ A_d = B \gamma_{\Ccal}^{-1}$, and
\begin{eqnarray*}
(i_{N,M}^{(d)})^* q_\Ccal' (z) & = & e^{\frac{2 i \pi}{\omega_{\Ccal'}^{(M)}} (\gamma_{\Ccal'}^{-1} \circ A_d) \cdot z} =  e^{\frac{2 i \pi}{\omega_{\Ccal'}^{(M)}} (B \cdot (\gamma_{\Ccal}^{-1} \cdot z))}  \\
& = & \\
& = & e^{\frac{2 i \pi}{\omega_{\Ccal'}^{(M)}} \left(\frac{\gcd(b/d)}{d'} (\gamma_{\c}^{-1} \cdot z) + \frac{hde-gf}{d'} \right)} \\
& = & q_{\Ccal}^{\frac{\gcd(b,d)\omega_\Ccal^{(N)}}{d' \omega_{\Ccal'}^{(M)}} } e^{\frac{2 i \pi (hde-gf)}{d' \omega_{\Ccal'}^{(M)}}},
\end{eqnarray*}
which gives the claimed identity after rearranging the exponents.

$(i)$ 
As a direct consequence of that formula, for $d=1$ the ramification index $e$ of $i_{N,M}^{(1)}$ at $\Ccal$ is $e= \omega_\c^{(N)}/ \omega_\c^{(M)}$. Consequently, $i_{N,M}^{(1)}$ is unramified at $c$ if and only if for all primes $p|N$,
\begin{equation}
 v_p\left( \frac{N}{\gcd(N,b^2)} \right) = v_p\left( \frac{M}{\gcd(M,b^2)} \right).
\end{equation}
If $v_p(M) = v_p(N)$, this is automatic, otherwise separating in three cases whether $2v_p(b)$ belongs to $]0,v_p(M)[,]v_p(M),v_p(N)[$ or $[v_p(N),+\infty[$, we obtain in the first two situations than $v_p(\omega_\c^{(N)}) > v_p(\omega_\c^{(M)})$ and in the third one that $v_p(\omega_\c^{(N)}) = v_p(\omega_\c^{(M)}) = 0$.

Consequently, $e=1$ if and only if for every prime factor $p$ of $N/M$, $2 v_p(b) \geq v_p(N)$.

$(ii)$ Assume now that $\c = \Ccal_N(a/b)$ is a width 1 cusp and $d|(N/M)$. As $\c$ is assumed of width $1$, $\omega_\c^{(N)} = 1$ and $\gcd(b,d)^2/d$. Consequently, $i_{N,M}^{(d)}$ has ramification index $\gcd(b,d)^2 /(d \omega_{\c'}^{(M)})$ at $\c$, and it is equal to 1 if and only if $d \omega_{\c'}^{(M)} = \gcd(b,d)^2$.
This means that for every prime $p|d$, we must have 
\begin{equation}
v_p(d)+ v_p(M) - \min(2v_p(b'),v_p(M)) = 2 \min(v_p(b),v_p(d)).
\end{equation}

By case by case analysis, the solutions of these equalities give either $v_p(d) = 0$ or $v_p(d) = v_p(N/M)$ and $v_p(b) = v_p(N)/2$.
\end{proof}

\begin{cor}
\label{cor:width1cuspsandAL}
	With the previous notation:

  $(i)$ If $N$ is almost squarefree with powerful prime $p$ and $M$ a squarefree divisor of $N$ with $p|M$, for any width 1 cusp $\c$ of $X_0(N)$ and any Hall divisor $d \neq 1$ of $(N/M)$, $i_{N,M}^{(d)}$ is unramified at $\c$ if and only if $d=p^{v_p(N/M)}$ and $v_p(b) = v_p(N)/2$. Conversely, for any cusp $\Ccal$ with $v_p(b) = v_p(N)/2$, and any $d \neq 1, p^{v_p(N/M)}$, $i_{N,M}^{(d)}$ is ramified at $\Ccal$ with ramification order at least the minimum of non-powerful prime factors of $N$.

	$(ii)$ For any level $N \geq 1$ and any cusp $\c = \Ccal_N(a/b)$ of $X_0(N)$, there is an Atkin--Lehner operator $w$ on $X_0(N)$ such that $w \cdot \c$ is of width 1 in $X_0(N)$ (and then $i_{N,M}^{(1)}$ is unramified at $w \cdot \c$ for any $M  |  N$).
\end{cor}

\begin{proof}
	$(i)$ the first part is immediate from Proposition \ref{propramdegmaps} as $N$ has no primes with multiplicity except $p$. Let us fix $\Ccal = \Ccal_N(a/b)$ with $v_p(b) = v_p(N)/2$ and check the ramification order of $i_{N,M}^{(d)}$ for $d \neq 1, p^{v_p(N/M)}$. 
By hypothesis, we can write $N = p^k q_1 \cdots q_s$ such that $ M = p q_1 \cdots q_r$ with pairwise distinct primes $p,q_1, \cdots, q_s$ and $k \in \{2,3\}$. The existence of $\Ccal$ as above implies $k=2$ and $v_p(b)=1$, and by the computations of the proposition, the ramification index of $i_{N,M}^{(d)}$ at $\Ccal$ is, for $b' = b/\gcd(b,d)$,
\[
\frac{\gcd(b,d)^2}{d \omega_{\Ccal'}^{(M)}} = \frac{\gcd(b,d)^2}{d M/\gcd(b',M)^2}.
\]
For $q = q_i$, as we must have $v_q(b)=1$, the $q$-adic valuation of the ratio is 0 when $q \nmid d$ and $1$ when $q |d$. As $d \neq 1, p^{v_p(N/M)} = p$, it must have a prime factor distinct from $p$, hence the ramification order it at least the minimum of these prime factors.

	$(ii)$ Let $ \c= \Ccal_N(a/b)$ be a cusp with $(a,b) \in \Rcal_N$. 
 Using the Atkin--Lehner involution $w_p^{(N)}$, we can replace $v_p(b)$ by $v_p(N)-v_p(b)$ (and not change the valuations for other prime factors) by Proposition \ref{propactALoncusps}, so we can change $b$ such that $v_p(b) \geq v_p(N)/2$ prime by prime, and thus the cusp obtained at the end of the process is of width 1.
\end{proof}

\begin{prop}[Cotangent maps in the unramified case]
\label{propcotangentunramified}
 Assume here the hypothesis 
 
$(HV)$ For every prime $p|N$, $v_p(M) \leq \lceil v_p(N)/2 \rceil$.
 
 For every width 1 cusp $\c = \Ccal_N(a/b)$ of $X_0(N)$ (with $b|N$) and every Hall divisor $d$ of $N/M$ such that $i_{N,M}^{(d)}$ is unramified at $\c$, if $d_M = \prod_{p|d} p^{v_p(M)}$ and $d_N = \prod_{p|d} p^{v_p(N)}$, 
 \begin{equation}
 w_{d_M}^{(M)} \circ i_{N,M}^{(d)}(\c) = \infty = i_{N,M}^{(1)}(\c) = i_{N,M}^{(1)}(w_{d_N}^{(N)}(\c))
 \end{equation}

 on $X_0(M)$ and the cotangent maps from $\Ccal$ to $\infty$ are related by
 \begin{equation}
  \Cot_{\c}(i_{N,M}^{(1)} \circ w_{d_N}^{(N)}) = \zeta \Cot_{\c} (i_{N,M}^{(1)}), 
 \end{equation}
with $\zeta$ a primitive root of unity of exact order $\prod_{p|d} p^{v_p(N)/2} = \gcd(b,d)$.
\end{prop}

\begin{proof}
 By Proposition \ref{propdiagcommALandiNM}, $w_{d_M}^{(M)} \circ i_{N,M}^{(d)} = i_{N,M}^{(w_{d_N}^{(N)} \cdot d)} \circ w_{d_N}^{(N)}$ where $d_N = \prod_{p|d} p^{v_p(N)}$. For a width 1 cusp $\c$ and $d \neq 1$, such that $i_{N,M}^{(d)}$ is unramified at $\c$ we must have $v_p(b) = v_p(N)/2$ and $v_p(d) = v_p(N/M)$ for every prime $p|d$ (and $v_p(b) \geq v_p(N)/2$ for all other primes as $\c$ is of width 1), by Proposition  \ref{propramdegmaps}). Consequently, applying $w_{d_N}^{(N)}$ sends $\c$ to a cusp $\c'$ with the same properties (and same denominator exactly), let us write it $\c' = \Ccal(a'/b)$.
 
 Now, let us compare the cotangent maps by lifting everything to $\Hcal$. One can fix a representative matrix of $w_{d_N}^{(N)}$ sending $a/b$ to $a'/b$, denoted by 
 \begin{equation}
  W_{d_N}^{(N)} := \begin{pmatrix} d_N x & y \\ N z & d_N t \end{pmatrix},
 \end{equation}
 
with integers $x,y,z,t$ such that $\det W_{d_N}^{(N)} = d_N$. Let us fix matrices of $\SL_2(\Z)$ defined as 
\begin{equation}
     \gamma := \begin{pmatrix} a & e \\ b & f \end{pmatrix},  \gamma' := \begin{pmatrix} a' & e' \\ b & f' \end{pmatrix},
\end{equation}
so that $\gamma \cdot  \infty = a/b$ and $\gamma' \cdot  \infty = a'/b$.

We now have to compare, on $\Hcal \cup \P^1(\Q)$, the maps $[\gamma'] \circ [W_{d_N}^{(N)}]$ and $[\gamma]$, both sending $a/b$ to $\infty$. By a straightforward but lengthy matrix computation, we can write 
\begin{equation}
 (\gamma')^{-1} W_{d_N}^{(N)} \gamma= \begin{pmatrix} a f d_N x + a e' N z + b f' y + b e' d_N t & e f d_N x + e e' N z + ff'y +  e'f d_N t\\ - a b d_N x + a a' N z - b^2  d_N x + a'bNz & -bey + a'ed_Nt -bfy+a'fd_Nt \end{pmatrix}. 
\end{equation}
By construction, as this matrix sends $\infty$ to $\infty$, the lower-left term is 0 (which can also be seen using the fact that $W_{d_N}^{(N)} (a/b) = a'/b$). Now, define $r_d = \prod_{p|d} p^{v_p(N)/2}$. By Proposition \ref{propramdegmaps}, we have for every $p|d$ that $v_p(N)/2 = v_p(b)$ and by $(HV)$ that $v_p(M) \leq v_p(b)$ so $v_p(d) = v_p(N/M) \geq v_p(b)=v_p(N)/2$, so $r_d$ is the square root of $d_N$ and divides $b$. Therefore, $r_d$ divides both diagonal terms of the matrix, but the determinant of the matrix is $d_N = r_d^2$, hence up to sign, they are both $r_d$. Now, modulo $r_d$, the upper-right coefficient is equal to $f f' y$ but by construction, each of those factors is prime to $r_d$. In other words, we have 
\begin{equation}
 (\gamma')^{-1} W_{d_N}^{(N)} \gamma  = \pm \begin{pmatrix} r_d & k \\ 0 & r_d \end{pmatrix} =: B, 
\end{equation}
with $k$ prime to $r_d$. As a consequence, in the same fashion as for Proposition \ref{propramdegmaps} for every $z \in \Hcal$, in a neighbourhood of $\Ccal$, in $X_0(N)$, 
\begin{eqnarray}
\label{eq:rootsofunitywd}
(w_{d_N}^{(N)})^* q_{\Ccal'} & = & e^{2 i \pi (\gamma')^{-1} ((W_{d_N})^{(N)} \cdot z)} = e^{ 2i \pi (B \gamma^{-1}) \cdot z} \\
& = & e^{ 2 i \pi \left(\frac{r_d \gamma^{-1} z + k}{r_d}\right)} = e^{ \frac{2 i \pi k}{r_d}} q_\Ccal.
\end{eqnarray}
Notice that here, $q_\Ccal,q_\Ccal'$ are canonically defined as the cusps are of width 1.
Starting now with a neighbourhood of $\infty$ in $X_0(M)$ and pulling back $q_\infty$ with $(i_{N,M}^{(1)})^*$, we obtain, in a neigbhourhood of $\Ccal$:
\[
e^{\frac{2 i \pi k}{r_d}} (i_{N,M}^{(1)})^* q_\infty =  (i_{N,M}^{(1)} \circ w_{d_N}^{(N)})^* q_\infty.
\]
In particular,
\begin{equation}
   \Cot_{\c}(i_{N,M}^{(1)} \circ w_{d_N}^{(N)}) = \zeta \Cot_{\c} (i_{N,M}^{(1)})
\end{equation}
and this concludes the proof of the proposition. 
\end{proof}

\section{\texorpdfstring{Existence of rank zero quotients}{Existence of rank zero quotients: analysis of vanishing of the L-functions}}\label{sec:analysis-of-vanishing-of-the-l-functions}

In this section, we prove that $J_0(M)^{-_p, +_{(p)}}$ has a non-zero rank zero new quotient away from finitely many of the levels $M \geq 1$ considered. When $M$ is a non-exceptional prime, then the rank 0 new quotient of $J_0(p)^-$ is Mazur's Eisenstein quotient. More generally, we do an analysis on the vanishing of the $L$-functions $L(f,1)$ of $f \in S_2(\Gamma_0(M))$ with prescribed Atkin--Lehner signs using analytic estimates on their sums.

\subsection{Analysis of vanishing of the \texorpdfstring{$L$}{L}-functions}

To achieve this, we first introduce some analytic tools that will help us bound sums of $L(f,1)$ central values of $L$-functions.
\begin{defi}[Kloosterman sums]
	For any integers $a,b$ and $c$ with $c>1$, the associated Kloosterman sum is defined as 
	\begin{equation}
	S(a,b;c) := \sum_{x \in (\Z/c\Z)^\times} \exp \left( 2 i \pi \frac{m x + n \bar{x}}{c} \right)
	\end{equation}
	where $\bar{x}$ is the inverse of $x$ modulo $c$. (Kloosterman sums with $c=1$ will be equal to $1$ by convention.)
	
	This sum satisfies the Weil bound \cite[Corollary 11.12]{IwaniecKowalski}:
	\begin{equation}
		\label{eqWeilbound}
		|S(m,n;c)| \leq (m,n,c)^{1/2} \tau(c) \sqrt{c}. 
	\end{equation}
\end{defi}

\begin{defi}[First Bessel function]
	The Bessel function $J_1$ is the function on $\R$ defined as the power series
	\begin{equation}
	J_1(x) = \sum_{k=0}^{+ \infty} \frac{(-1)^k}{2^{2k+1} k! (k+1)!} x^{2k+1}.
	\end{equation}
	The Bessel function is bounded as follows: $|J_1(x)| \leq |x|/2$ for all $x \in \R$ and $J_1$ oscillates towards $0$ while converging towards $0$ at speed $1/|x|^{1/2}$ \cite[Formula 9.2.1]{AbramowitzStegun}.
\end{defi}

We will make use of the following version of Petersson trace formula, adapted to multiple eigenvalues as proven in \cite[Proposition 5.6]{LeFournLfunctions}.
\begin{prop}[Restricted Petersson trace formula with multiple eigenvalues]
	\label{propformulesdestraceslaplusgenerale}
	\hspace*{\fill}
	
	Let $m,n,M$ be three fixed positive integers. Let $E$ be a group morphism from a subgroup $H$ of $\Wcal(M)$ to $\{\pm 1 \}$. For every Hall divisor $Q$ of $M$, let us define 
	\begin{equation}
	S_Q \colonequals  2 \pi \sqrt{\frac{m}{n}} \sum_{\substack{c>0 \\ (M/Q)  |  c \\(Q,c) = 1}} \frac{S(m,nQ^{-1} ;c)}{c \sqrt{Q}}J_1 \left( \frac{4 \pi \sqrt{mn}}{c\sqrt{Q}} \right),
	\end{equation}
	and for $\Bcal$ an eigenbasis of $S_2(\Gamma_0(M))$,
	\begin{equation}
		\label{eq:trace}
		(a_m,a_n)_M^E \colonequals \sum_{\substack{f \in \Bcal \\ \forall w \in H,\\ f_{|w} = E(w) f }} \frac{\overline{a_m(f)}a_n(f)}{\|f\|^2}.
	\end{equation}
	Then, we have 
	\begin{equation}
	\frac{|E|}{4 \pi \sqrt{mn}} (a_m,a_n)_M^{E} \colonequals \delta_{mn} - \sum_{\substack{Q||M \\ w_Q^{(M)} \in H}} E(w_Q^{{(M)}}) S_Q.
	\end{equation}
\end{prop}

This helps us to compute the central values of $L$-functions thanks to the following result. For any $f \in S_2(\Gamma_0(M'))$ and any $x$, 
\begin{equation}
	\label{eqapproxLfchi}
	L(f,1) = \sum_{n=1}^{+ \infty} \frac{a_n(f)}{n} e^{- \frac{2 \pi n}{x}} - \sum_{n=1}^{+ \infty} \frac{a_n(f_{|w_{M'}})}{n} e^{- \frac{2 \pi n x}{M'}}.
\end{equation}
(This is well-known, see e.g.\  \cite[(2.11.1)]{Cremona97} in the particular case $x = \sqrt{M'}$ the level.)

We will always take $x = \sqrt{M}$, and summing on an (orthogonal) eigenbasis $\Bcal$ of the modular forms space $S_2(\Gamma_0(M))^{-_p,+_{(p)}}$, we obtain using Proposition \ref{propformulesdestraceslaplusgenerale} the following formula for the trace \eqref{eq:trace}
\begin{equation}
	\label{eqmaintermpluserror}
	T_p(M) := \frac{2^{\omega(M)}}{8 \pi} \sum_{\substack{f \in \Bcal \\ \forall q \neq p, f_{|w_q} = f \\\ f_{|w_p} = -f}} \frac{a_1(f)L(f,1)}{\|f\|^2} = e^{- 2 \pi/\sqrt{M}} - R
\end{equation}
with 
\begin{equation}
R =  2 \pi \sum_{Q \in W} \frac{(\pm 1)}{\sqrt{Q}} \sum_{n \geq 1} \frac{e^{- 2 \pi n / \sqrt{M}}}{\sqrt{n}} \sum_{ \substack{c>0 \\ (M/Q)  |  c \\(Q,c) = 1}} \frac{S(1,nQ^{-1};c)}{c}  J_1 \left( \frac{4 \pi \sqrt{n}}{c\sqrt{Q}} \right).
\end{equation}

Using Weil bounds, we could already obtain a threshold for $M$ after which $T_p(M)>0$, but it would not be very good, so we also make use of Polya--Vinogradov type estimates explained below. Define for $Q$ a Hall divisor of $M$ and $c>0$ multiple of $M/Q$ 
\begin{equation}
S_Q(c) \defeq \frac{1}{c} \sum_{n \geq 1}  \frac{e^{- 2 \pi n / \sqrt{M}}S(1,nQ^{-1};c)}{\sqrt{n}} J_1 \left( \frac{4 \pi \sqrt{n}}{c\sqrt{Q}} \right),
\end{equation}
so that 
\begin{equation}
R = 2 \pi \sum_{Q} \frac{\pm 1}{\sqrt{Q}}  \sum_{ \substack{c>0 \\ (M/Q)  |  c \\(Q,c) = 1}} S_Q(c).
\end{equation}

\begin{lem}
	\label{lemboundsSQ}
	For every $Q|M$ such that $(Q,M/Q)=1$ and every $c>0$ multiple of $M/Q$:
	\begin{equation}
	|S_Q(c)| \leq \min \left( \frac{\tau(c) \sqrt{M}}{c^{3/2} \sqrt{Q}},\frac{5.7}{c\sqrt{Q}} (\log(c)+1.5)\right) 
	\end{equation}
	
\end{lem}

\begin{proof}
	Weil bounds give 
	\begin{equation}
	|S_Q(c)| \leq \frac{\tau(c) \sqrt{M}}{c^{3/2} \sqrt{Q}}.
	\end{equation}
	
	Now, by \cite[Lemma 5.9]{LeFourn1}, for any $c >1$ and any $k$ prime to $c$
	\begin{equation}
	\sup_{K,K' \in \N} \left| \sum_{n=K}^{K'} S(1,nk,c) \right| \leq \frac{4c}{\pi^2} (\log(c) + 1.5).
	\end{equation}
	By Abel transform, we can thus write 
	\begin{equation}
	|S_Q(c)|  \leq \frac{4(\log(c) + 1.5)}{\pi^2} \sum_{n \geq 1} |f_{Q,c}(n) - f_{Q,c}(n+1)| \leq  \frac{4(\log(c) + 1.5)}{\pi^2} \operatorname{Totvar}(f_{Q,c})
	\end{equation}
	with $f_{Q,c}(x) = e^{ - 2 \pi x/ \sqrt{M}}/\sqrt{x} J_1(4 \pi \sqrt{x}/(c \sqrt{Q}))$ and $\operatorname{Totvar}$ the total variation of $f_{Q,c}$ on $[1,+ \infty[$. For this total variation, notice that we can take out the exponential factor (it only decreases the total variation), and rewrite 
	\begin{equation}
	\frac{ J_1(4 \pi \sqrt{x}/(c \sqrt{Q}))}{\sqrt{x}} = \frac{4 \pi }{c \sqrt{Q}} J_1 (4 \pi \sqrt{x}/(c \sqrt{Q}))/(4 \pi \sqrt{x}/(c \sqrt{Q})) 
	\end{equation}
	The total variation of $f_{Q,c}$ on $[1,+ \infty[$ is thus bounded by $4 \pi/ (c \sqrt{Q})$ times the one of $J_1(y)/y$ on $[0,+ \infty]$, which can be computed to be less than 1.1.
	
	We obtain 
	\begin{equation}
	|S_Q(c)| \leq \frac{5.7}{c\sqrt{Q}} (\log(c)+1.5). \qedhere
	\end{equation}
\end{proof}

Let $\lambda \geq 1$ be a choice of threshold for the bounds in Lemma \ref{lemboundsSQ}. Denoting $c = (M/Q) c'$ and using the latter bound for $c' \leq  \lambda$ and the Weil bound for $c' >  \lambda$, we obtain
\begin{eqnarray}
	\label{eqboundsSQ}
	\frac{1}{\sqrt{Q}} \sum_{ \substack{c>0 \\ (M/Q)  |  c \\(Q,c) = 1}} |S_Q(c)| & \leq & \frac{5.7}{M} \sum_{1 \leq c' \leq \lambda } \frac{(\log(Mc'/Q) + 1.5)}{c'} + \frac{\tau(M/Q)}{M/\sqrt{Q}} \sum_{c' >  \lambda} \frac{\tau(c')}{(c')^{3/2}}  \\
	& \leq & \frac{5.7( (\log(M/Q)+ 1.5) (\log(\lambda) + 1) + (1 + \log(\lambda)^2)/2 )}{M} \\
	& + & \frac{\tau(M/Q)}{M/\sqrt{Q}} \frac{2 \log(\lambda) + 8}{\sqrt{\lambda}} 
\end{eqnarray}
by \cite[Lemma 5.11]{LeFourn1}. 
For $Q >36$, we choose $\lambda = Q/36$ in the previous bound.

For $Q \leq 36$, we will only use the Weil bounds, and this gives as stated before
\begin{equation}
\frac{1}{\sqrt{Q}} \sum_{ \substack{c>0 \\ (M/Q)  |  c \\(Q,c) = 1}} |S_Q(c)| \leq \frac{\tau(M/Q)}{M/\sqrt{Q}} \zeta(3/2)^2 \leq 6.9 \frac{\tau(M/Q)}{M/\sqrt{Q}}.
\end{equation}

Gathering all the bounds, we obtain the following estimate.

\begin{lem}
	\label{lemboundsforpq}
	For every $M = pq$ with $p \in \{2,3,5,7,13\}$ and $q > 36$ prime, we have for $\Bcal$ an orthogonal basis of $S_2(\Gamma_0(M))^{-_p,+_q}$:
	\begin{equation}
		\label{eqwitherrorterm}
		\left| \frac{1}{2 \pi} \sum_{f \in \Bcal} \frac{a_1(f)L(f,1)}{\|f\|^2} - 1 \right| \leq (1 - e^{- 2 \pi/\sqrt{pq}}) + (2 \pi) \cdot \left(6.9 \left( \frac{4}{pq} + \frac{2}{q \sqrt{p}}\right) + f_1(q) + f_2(q)\right)
	\end{equation}
	where 
	\begin{eqnarray}
		f_1(x) & = & \frac{5.7 (\log(p) + 1.5)(\log(x/36)+1) + (1+ \log(x/36)^2)/2}{px} \notag\\
		& + & \frac{12(2 \log(x/36) + 8)}{px} \\
		f_2(x) & = &  \frac{5.7 (1.5)(\log(px/36)+1) + (1+ \log(px/36)^2)/2}{px} \notag\\
		& + & \frac{6(2 \log(px/36) + 8)}{px}. \notag
	\end{eqnarray}
	Moreover, for fixed $p$, the right-hand side of \eqref{eqwitherrorterm} is a strictly decreasing function of $q$ on $[36,+\infty[$.
\end{lem}

\begin{proof}
	The estimates are simply given by collecting the previous bounds, so we only have to prove the claim about decreasing functions of $q$. It is enough to prove that $f_1$ and $f_2$ are decreasing for all $x \geq 36$ (as the other terms are clearly decreasing), and for those they are a sum of obviously decreasing terms and of terms up to constant equal to  $\log(px)/x$ or $\log(px)/36)^2/x$ or $\log(x/36)+1/x$ or $(\log(x/36)+4)/x$. Each one of those is a decreasing function for $x$ large, and computing the derivatives the worst case is $(\log(x/36)+1)/x$ which is decreasing exactly after $36$.
\end{proof}

This allows us to compute these values precisely for some $q$'s and figure out the thresholds for which the error term is strictly less than $1$.

\begin{thm}
	\label[theorem]{proprankzeroquotient}
	In each of the following cases, the abelian variety $J = J_0(M)^{-_p,+_{(p)}}$ has a rank zero new quotient $A$.
	
	\begin{enumerate}
		\item For $M = p \notin \{2,3,5,7,13\}$ prime  (and then $J = J_0(p)^-$).
		
		\item For $M = pq$ with $p \in \{2,3,5,7,13\}$ and $q \neq p$ (then $J = J_0(pq)^{-_p,+_q}$), except precisely in the following cases. 
  
		\begin{equation}
		\begin{array}{|c|c|c|c|c|c}
			\hline
			p & 2 & 3 & 5 & 7 \\
			
			q & (3,5,7,11,23) & (2,5,11) & 2 & 3 \\
		
			\hline
		\end{array}
		\end{equation}

		\item For $M=pqr$ with  $p \in \{2,3,5,7,13\}$ and $(p,q,r)$ are distinct primes (then  $J = J_0(pqr)^{-_p,+_q,+_r}$) 
     and such that $(p,q)$ and $(p,r)$ are in the table of exceptions of case 2, 
        $J$ has a rank zero quotient except for $(p,q,r)$ an exceptional triple, i.e.\ $(2,3,5)$ or $(2,5,3)$.
	\end{enumerate}
\end{thm}

\begin{proof}
	For the first case, this holds by Mazur's work on the Eisenstein quotient \cite[Theorem~4]{Mazur1977} since $\dim J_0(p)^- > 0$ for every prime $p \notin \{2,3,5,7,13\}$. 	
	For the two other cases, by the Gross--Zagier--Kolyvagin--Logach\"ev theorem~\cite{GrossZagier1986,KolyvaginLogachev}, it is enough to prove that there is a newform $f$ of level $M$ and weight $2$ with the claimed Atkin--Lehner eigenvalues and such that $L(f,1) \neq 0$, and we will start by using the estimates of Lemma~\ref{lemboundsforpq}.
	
	Since the error term in this lemma is decreasing for $q>36$, it is enough to find a value for which the right hand side of~\eqref{eqwitherrorterm} is less than $1$, for each $p \in \{2,3,5,7,13\}$.  
	We thus obtain rank zero quotients (which are automatically new as the oldforms cannot satisfy $L(f,1) \neq 0$ in this case) when  for $p=2$, $q>1700$, for $p=3$, $q>1100$, for $p=5$, $q>600$, for $p=7$, $q>450$ and for $p=13$, $q>250$. 
	
	Afterwards, an LMFDB search on the finitely many remaining cases with $q$ at most the given bound gives a rank zero quotient for $J_0(pq)^{-_p,+_q}$ (and it is automatically new here) except for the cases $(p,q)$ in the table.  
	
	For case 3, we proceed similarly with the finitely many possible cases to check.
\end{proof}

Theorem 
\ref{thm:rank0quotient} is now a consequence of \cref{proprankzeroquotient}, as proven below. 
\begin{proof}[Proof of Theorem ~\ref{thm:rank0quotient}.]
	Since $\widetilde{N}$ is not exceptional, by Definition \ref{defexceptionallevel}, there exists a prime $p$ such that $p^2  |  \wN$, satisfying one of two possibilities:

    \begin{enumerate}[(1)]
  
	\item If $p $ is not exceptional (i.e.\ $p \notin \{ 2,3,5,7,13\}$) we can then divide $\wN$ by a square to obtain $N$ almost squarefree \eqref{eqshapeN} with $v_p(N) = 2$ or 3. Then $M =p$, and we choose $A$ to be the Eisenstein quotient of $J_0(p)^-$. The result follows by \cite[Theorem~4]{Mazur1977}.
	
	   \item  If $p$ is exceptional, we can find $N$ almost squarefree \eqref{eqshapeN} with another prime divisor $q$ of $N$ such that $(p,q)$ is not exceptional, or two primes $q \neq q'$ different from $p$ with $qq'  |  N$ and $(p,q,q')$ not exceptional. In the first case, one can fix $M = pq$, in the second $M=pqq'$; we find a rank zero quotient in each case by  \cref{proprankzeroquotient}. \qedhere
    \end{enumerate}
\end{proof}

\section{Construction of the formal immersions and proof of integrality for nonexceptional levels}
\label{sec:formalimmersion}

\subsection{Preliminary results on formal immersions and reductions of points}

\begin{defi}[Reduction of points]
    For any proper scheme $\Xcal$ over $\Z$, any number field $K$ and any prime ideal $\lambda$ of $\Ocal_K$ (of residual characteristic $\ell$ prime and residue field $\F_\lambda$), a point $P \in \Xcal(K)$ extends uniquely by the valuative criterion of properness to a section $\widetilde{P}: \Spec \Ocal_K \rightarrow \Xcal$, and we denote by $P_\lambda$ the set-theoretic image of $\lambda \in \Spec \Ocal_K$ in $\Xcal$, which is thus a point of $\Xcal$ above $\ell \Z \in \Spec \Z$. By abuse of notation, it will also sometimes be considered as an element of $\Xcal_{\F_\ell}(\F_\lambda)$.
\end{defi}

\begin{defi}[Formal immersions]
\label{defformalimmersion}
Let $\Xcal$ and $\Ycal$ be two locally noetherian schemes and $f: \Xcal \ra \Ycal$. For $x \in \Xcal$ with image $y \in \Ycal$ by $f$, $f$ is a \emph{formal immersion at $x$} if the induced morphism on completed local rings 
\begin{equation}
\widehat{f_x} : \widehat{\Ocal_{\Ycal,y}} \rightarrow \widehat{\Ocal_{\Xcal,x}}
\end{equation}
is surjective.

Here is a well-known criterion to determine whether $f$ is a formal immersion. 

\end{defi}

\begin{lem}[Formal immersion criterion]
\label{lem:formimmcriterion}
With the same notation as in Definition \ref{defformalimmersion}, $f$ is a formal immersion at $x$ if and only if the two following conditions hold:

$\bullet$ The induced map $\kappa(y) \ra \kappa(x)$ between residue fields is an isomorphism.

$\bullet$ The induced cotangent map $\Cot_x (f) : \gm_{\Ycal,y} /\gm_{\Ycal,y}^2 \rightarrow \gm_{\Xcal,x} / \gm_{\Xcal,x}^2$ is surjective.
\end{lem}

\begin{proof}
    See e.g.\  \cite[Proposition I.2.10]{LeFournthese2} or \cite[Proposition 17.4.4]{EGAIV}.
\end{proof}

Formal immersions are the main tool to prove rational points on modular curves have potentially good reduction, via the following key result \cite[Proposition 2.4]{LeFourn1}.

\begin{prop}
\label{prop:formimmkeyresult}
   Let $K$ be a number field and $\lambda$ a prime ideal of $\Ocal_K$ with residual characteristic $\ell$ and residue field $\F_\lambda$.

   Let $X$ be a projective algebraic curve over $\Q$ with a proper model $\Xcal$ over $\Z$, and $A$ an abelian variety over $\Q$ with Néron model $\Acal$ over $\Z$. 

   Let $f : X \ra A$ be a morphism, extending by Néron mapping property to $f_\Z: \Xcal^{\rm{smooth}} \rightarrow \Acal$.

   Let $x, y \in X(K)$ such that: 

   $\bullet$ The point $x$ and $y$ have the same reduction modulo $\lambda$ and it belongs to 
$\Xcal^{\rm{smooth}}$. 

$\bullet$ The map $f_\Z$ is a formal immersion at this common reduction point $x_\lambda=y_\lambda$.

$\bullet$ The difference $f(y)-f(x)$ is a torsion point in $A(\Q)$.

Then, $x=y$ unless $\ell=2$ and $f(y)-f(x) \in A(\Q)$ is of order exactly 2 and generates a $\mu_2$ over $\Z_2$.
\end{prop}

To deal with possible non-smooth reductions (in which case we cannot use the formal immersion criterion as above), we will use the following result.
\begin{lem}
\label{lemregularsmooth}
Let $\mathscr{C}/\Z$ be a curve with generic fiber $C_\Q$. Let $Q \in \mathscr{C}$ above $p$ for some prime $p$. If $Q$ is a regular point of $\mathscr{C}/\Z$ but not a smooth point of $\mathscr{C}_{\F_p}$, then no rational point of $C_\Q$ reduces modulo $p$ to $Q$ (i.e.\ there are no $\Z$-points of $\mathscr{C}$ with reduction $Q$ on $\mathscr{C}_{\F_p}$).

In particular, if $\mathscr{C}/\Z$ is a proper regular model of a curve $C_\Q/\Q$, then $C_\Q(\Q) = \mathscr{C}^{\rm sm}(\Z)$, where $\mathscr{C}^{\rm sm}$ is the smooth locus of $C/\Z$.
\end{lem}

\begin{proof}
It suffices to argue locally around $Q$, and show there is no $\Z_p$-point of $\mathscr{C}$ reducing to $Q$ modulo $p$. Let $\gm_Q$ be the maximal ideal corresponding to the closed point $Q$ of $\mathscr{C}$. Because $\mathscr{C}$ is regular at $Q$, we have that $\dim \gm_Q / \gm_Q^2 = 2$. Since $Q$ is singular in $\mathscr{C}_{\F_p}$, the tangent space to $\mathscr{C}_{\F_p}$ at $Q$ has dimension $ \geq 2$. The fiber $\mathscr{C}_{\F_p}$ is cut out by the vanishing of $p$, this tangent space is to the fiber is also cut out by the vanishing of $p$ inside $ \gm_Q / \gm_Q^2$, i.e.
\begin{equation} \dim (\gm_Q / \gm_Q^2)/(p) = 2 =  \dim (\gm_Q / \gm_Q^2) \end{equation}
and therefore $p \in (\gm_Q / \gm_Q^2)$.

Suppose for contradiction that we have some $\Q_p$-point whose reduction modulo $p$ is $Q$, i.e.\ a $\Z_p$-point. This takes place in an affine chart $\Spec \Z_p [x,y] / (f(x,y)) \subset \mathscr{C}_{\Z_p}$ with coordinates such that $\gm_Q = (x,y,p)$ such that our $\Z_p$-point is $(X,Y) = (0,0)$ modulo $p$. Since $p \in \gm_Q$, we can write 
\begin{equation} p = \lambda f(x,y) + (p, x,y)^2 \end{equation} 
in $\Z_p[x,y]$. Evaluating at the point $(X,Y)$, the left hand side has $p$-adic valuation 1, but $\lambda f(X, Y) = 0$ and $(p,x,y)^2$ has $p$-adic valuation at least 2. This is a contradiction, so such a $\Z_p$-point cannot exist.
\end{proof}

For our explicit power series expansions, we need the following notion of good uniformizer (called well-behaved uniformizer in \cite[p.~771]{Siksek08} where the details can be found).
\begin{defi}[Good uniformizer]
    \label{def:goodunif}
    Let $R$ be a discrete complete valuation ring with maximal ideal $\gm$ and residue field $k$, $\mathscr{C}/R$ be a proper curve with generic fiber $C$ over  the fraction field $K$ of $R$.

    If $P \in C(K)$ has reduction $P_\gm$ at $\gm$ which is a smooth point of $\mathscr{C}_k$, a \emph{good uniformizer} at $P$ is a regular function $t \in \Ocal_{\mathscr{C},P}$ whose restriction in $\Ocal_{\mathscr{C}_k,P_\gm}$ is a uniformizer at $P_\gm$ and with $t(P)=0$. 
Such a good uniformizer always exists under these hypotheses, and we have a canonical isomorphism
\begin{equation}
\widehat{\Ocal_{\mathscr{C},P}} \cong R[[t]],
    \end{equation}
    in particular every regular 1-form $\omega \in H^0(\mathscr{C}, \Omega^1_{\mathscr{C}/R})$ admits an expansion at $P$ of the shape 
    \begin{equation}
\omega = \sum_{n \geq 0} a_n(\omega) t^n dt,
    \end{equation}
    with each coefficient $a_n(\omega)$ belonging to $R$.

\end{defi}
\subsection{\texorpdfstring{Models of modular curves over $\Z$}{Models of modular curves over Z}}

We will use the following integral models of the modular curves.

\begin{defi}
\label{defi:Xcal0N}
  For any level $N \geq 1$, $\Xcal_0(N)$ is the compactified coarse moduli scheme over $\Z$ of (generalized) elliptic curves together with a cyclic subgroup of order $N$ (\cite[\S 8.6]{KatzMazur} or \cite[\S 1.1]{Edixhoven90}), and $\Xcal_0(N)^*$ is its quotient by the Atkin--Lehner group $\Wcal(N)$.
\end{defi}

\begin{prop}[\cite{Edixhoven90}, Theorem 1.2.3.1]
For any cusp $\c$ of $X_0(N)$ defined over some number field $K$, and every maximal ideal $\gp$ of $\Ocal_K$, the reduction of $\c$ modulo $\gp$ (seen as a point of $\Xcal_0(N)$ belonging to the fiber at $p$) is a regular point of $\Xcal_0(N)$.
\end{prop}

We will need (in particular for the reduction at 2) some precise knowledge of the irreducible components at the bad fiber, mostly based on results of \cite{Momose1986}.

\begin{prop}
\label{prop:irredcomponentsX0pr}
For any $r \geq 1$, the irreducible components of $\Xcal_0(p^r)$ as denoted in Definition \ref{defi:Xcal0N} are all defined over $\F_p$ and intersect each other only at supersingular points. 
They can be numbered $\Ecal_0^{(p^r)}, \cdots,\Ecal_r^{(p^r)}$ (resp.\ $\Ecal_0^{h, (p^r)}, \cdots, \Ecal_r^{h, (p^r)}$for these components minus their supersingular points), and the following properties hold (we drop the superscript $(p^r)$ when it is obvious):

\begin{itemize}
    \item[$(a)$]For each $i \in \{0, \cdots, r\}$, a pair $(E,C) \in \Xcal_0(p^r)(k)$ with $E$ ordinary over a field $k$ in characteristic $p$ belongs to $\Ecal_i^h$ if and only if $C_{\overline{k}}$ fits into an exact sequence of the following shape:
    \begin{equation}
0 \rightarrow (\mu_{p^i})_{\overline{k}} \rightarrow C_{\overline{k}} \rightarrow (\Z/p^{r-i} \Z)_{\overline{k}} \rightarrow 0.
    \end{equation}
    \item[$(b)$] The cusp $\infty$ (resp.\ $0 = \Ccal_{p^r}(0)$) reduces modulo $p$ in $\Ecal_r^h$ (resp.\ $\Ecal_0^h$).
        \item[$(c)$] The Atkin--Lehner involution $w_{p^r}^{(p^r)}$ exchanges $\Ecal_i^{(p^r)}$ and $\Ecal_{r-i}^{(p^r)}$ for every $i \in \{0, \cdots, r\}$, and for every $r > s \geq 1$, the map $i_{p^r,p^{s}}^{(1)} : \Xcal_0(p^r) \rightarrow \Xcal_0(p^{s})$ satisfies
    \begin{equation}
i_{p^r,p^{s}}^{(1)}(\Ecal_i^{h,(p^r)}) = \Ecal_{\min(i,s)}^{h,(p^s)}.
    \end{equation}
    \item[$(d)$] Let $E$ be an elliptic curve over a local field $K$ with maximal ideal $\gp$ and residual characteristic $p$ and inertia degree $e_K$, and $C$ a cyclic subgroup of order $p^r$ in $E$. Then, if $p \geq \max(3e_K,5)$, $P= (E,C) \in \Xcal_0(p^r)(K)$ does not have potentially supersingular reduction modulo $\gp$, and if $p>  3e_K+1$, $P$ must reduce in $\Ecal_0^h$ or $\Ecal_r^h$ modulo $\gp$.

    \item[$(e)$] With the notation of the previous item, if $P' = (E',C') \in \Xcal_0(p^r)(K)$ such that for some isogeny $\varphi: E \rightarrow E'$, $\varphi(C) = C'$ and $p \nmid \deg(\varphi)$, assuming again $p>3e_K+1$, $P$ and $P'$ reduce simultaneously in $\Ecal_0^h$ or $\Ecal_r^h$ modulo $\gp$.
\end{itemize}

\end{prop}
\begin{proof}
    Items $(a)$, $(b)$ and the first part of item $(c)$ are given in \cite[(1.7)]{Momose1986}. The second part of item $(c)$ can be obtained by immediate induction on $r-s$ from the same reference, treating the case $r=s+1$ (recall Momose's notation for $i_{p^r,p^{r-1}}$ is $\pi_r$).

Regarding the first part of $(d)$, one has $p \geq 5$ and $p> e_K e' -1$ with the notations of Corollary (2.3) $(ii)$ of \cite{Momose1986}, so by contraposition, we obtain that $P$ does not have potentially supersingular reduction mod $\gp$. Then, its reduction belongs to $\Ecal_i^h$ for some $i \in \{0, \cdots, r\}$. Now, by  \cite[Corollary 2.3 $(i)$]{Momose1986}, if $i \in \{1, \cdots, r-1\}$, $p-1 \leq e' e_K \leq 3 e_K$ which contradicts the hypothesis, therefore $i=0$ or $r$.

Finally, item $(e)$ is a consequence of $(a)$, as an isogeny of degree prime to $d$ will induce a group scheme isomorphism between the reduction of the cyclic group schemes modulo $\gp$ therefore they fit into an exact sequence of the same shape, and $P'$ lands in the same irreducible component as $P$ modulo $\gp$.
\end{proof}
\begin{cor}
\label{cor:key0caseforextremecomponents}
    With the same notation as in Proposition \ref{prop:irredcomponentsX0pr}, on the formal sums of irreducible components $\Ecal_i^{h,(p^r)}$ ($0 \leq i \leq r$), one has 
    \begin{equation}
(1 - w_{p^s}^{(p^s)})\circ \left( i_{p^r,p^s}^{(1)} - i_{p^r,p^s}^{(p^{r-s})} \right) [(\Ecal_{i}^{h,(p^r)})] = 0
    \end{equation}
    when $i=0$ or $r$.
\end{cor}

\begin{proof}
Let us first recall that 
\begin{equation}
 \label{eq:relationsignscomponents}
i_{p^r,p^s}^{(p^{r-s})} = w_{p^s}^{(p^s)} \circ i_{p^r,p^s}^{(1)} \circ w_{p^r}^{(p^r)}
\end{equation}
by Proposition \ref{propdiagcommALandiNM} applied to $N=p^r$, $M=p^s$, $d = 1$, $w = w_{p^r}^{(p^r)}$. Therefore, one can rewrite 
\begin{equation}
(1 - w_{p^s}^{(p^s)})\circ \left( i_{p^r,p^s}^{(1)} - i_{p^r,p^s}^{(p^{r-s})} \right)  = i_{p^r,p^s}^{(1)} - w_{p^s}^{(p^s)} \circ i_{p^r,p^s}^{(1)} - w_{p^s}^{(p^s)} \circ  i_{p^r,p^s}^{(1)} \circ w_{p^r}^{(p^r)} +  i_{p^r,p^s}^{(1)} \circ w_{p^r}^{(p^r)}.
\end{equation}
    Following the properties of Proposition \ref{prop:irredcomponentsX0pr} $(c)$, one has for every $i \in \{0, \cdots, r\}$
    \begin{eqnarray}
i_{p^r,p^s}^{(1)} (\Ecal_i^{h,(p^r)}) & = & \Ecal_{\min(i,s)}^{h,(p^s)}, \notag\\  w_{p^s}^{(p^s)} \circ i_{p^r,p^s}^{(1)} (\Ecal_i^{h,(p^r)}) & = & \Ecal_{s-\min(i,s)}^{h,(p^s)}, \\
 i_{p^r,p^s}^{(1)} \circ w_{p^r}^{(p^r)} (\Ecal_i^{h,(p^r)}) & = & \Ecal_{\min(r-i,s)}^{h,(p^s)}, \notag \\
w_{p^s}^{(p^s)} \circ i_{p^r,p^s}^{(1)} \circ w_{p^r}^{(p^r)} (\Ecal_i^{h,(p^r)})& = &  \Ecal_{s-\min(r-i,s)}^{h,(p^s)}. \notag
    \end{eqnarray}
    Assuming $i=0$, the indices of the connected components in the four equalities above are respectively $0,s,s,0$. In the $i=r$ case, we obtain $s,0,0,s$. In both cases, the sum of four terms in \eqref{eq:relationsignscomponents} is exactly 0.
\end{proof}

The following recalls the main properties of minimal models of $\Xcal_0(M)$.

\begin{defi}[Minimal model of $X_0(M)$]
\label{defi:minimalmodel}
Let $M$ be a squarefree integer, $p|M$ prime, and $(R,\gp)$ a complete discrete valuation ring of characteristic 0, with fraction field $K$, finite residue field $k$ with residual characteristic $p$. Let $\pi$ be a uniformizer of $R$ and $e = v_\pi(p)$ the absolute ramification of $R$. Then: 

\begin{itemize}
    \item First, $\Xcal_0(M)_k$ is the union of two (geometrically) irreducible components, each isomorphic to $\Xcal_0(M/p)_{k}$, glued at supersingular points of $\Xcal_0(M)_{k}$, and $\Xcal_0(M)_R$ is smooth everywhere outside of the supersingular points in the fiber above $\gp$.

    In the fashion of Proposition \ref{prop:irredcomponentsX0pr} above, we denote by $(\Ecal_1)_\gp^{(M)}$ the irreducible component containing the reduction of cusp $\infty$ and $(\Ecal_0)_\gp^{(M)}$ the irreducible component containing the reduction of cusp $0 = \Ccal_M(0)$ modulo $\gp$.
    \item Let $s=(E,C_M)$ be a supersingular point of $\Xcal_0(M)(\overline{k})$. Denoting 
    \begin{equation}
    k_s = \frac{1}{2}\left|\Aut(E,C_{M/p})\right| \in \{1,2,3\},
    \end{equation}the scheme $\Xcal_0(M)_k$ is nonregular at $s$ if and only if $e k_s>1$, and the completed local ring of $\Xcal_0(M)_R$ at $s$ is isomorphic to 
    \begin{equation}
R[[X,Y]]/(XY-\pi^{ek_s}).
    \end{equation}
    \item The special fiber of the minimal regular model $\widetilde{\Xcal_0(M)}_R$ of $X_0(M)$ over $R$ is obtained by blowing up in $\Xcal_0(M)_{k}$ each nonregular (hence supersingular) point $s$ to a chain of $e k_s-1$ projective lines.
\end{itemize}
\end{defi}

\begin{proof}
    This is written in \cite[Appendix, Theorem 1.1]{Mazur1977} with $\overline{k}$ instead of $k$, but the formation of minimal models commutes with étale base changes which preserves the properties and definitions above (see also \cite[Appendix, $\S$3]{BertoliniDarmon97}). 
\end{proof}

The motivation for this minimal model computation is that it directly provides the description of the group of components of the Néron model of the Jacobian at a fiber of characteristic dividing the level, by the following result.
\begin{prop}[Theorem 9.6.1 of \cite{BLR}]
\label{prop:Neronmodelminmodel}
    With the same notations $(R,\gp),K,p$ as above, for any smooth projective curve $C$ over $K$ with minimal regular model $\Ccal$ over $R$ (and each irreducible component of $\Ccal$ assumed of multiplicity 1), denote by $J$ the Jacobian of $C$ and $\mathcal{J}$ its Néron model over $R$. 

    Then, the group of components $\Phi := \Jcal_k/\Jcal_k^0$ of the special fiber of $\mathcal{J}$ is isomorphic to the group generated by all formal sums with total degree 0 of irreducible components of $\Ccal_k$, with relations given by 
    \begin{equation}
\sum_{D'} (D \cdot D') [D'] = 0
    \end{equation}
    for each irreducible component $D$ of $\Ccal_k$ and $D'$ going through those irreducible components as well.

    With this description, for two points $c_1,c_2 \in C(K)$, the image of $[c_1-c_2]_k \in \Jcal_k(k)$ in $\Phi$  is precisely $[D_{c_1}] - [D_{c_2}]$ where $D_c$ is the irreducible component containing the reduction of $c$ modulo $\gp$.
\end{prop}

\subsection{Building the formal immersions for Mazur's method}\label{subsec:formal immesions}

In this section we build the formal immersions  $X_0(N) \to A$ where $A$ is a rank $0$ quotient of the new part of $J_0(M)$ for some divisor $M | N$, and use them to prove integrality of the $j$-invariant by Mazur's method.

\begin{prop}
\label{propdefGNf}
 Let $N, M \geq 1$ be integers with $M |N$ and $f$ be an eigenform of $S_2(\Gamma_0(M))^{\rm{new}}$. Let $\varepsilon_f :\Wcal(M) \rightarrow \{ \pm 1\}$ the character associated to the system of eigenvalues of $f$ for $\Wcal(M)$, assume $\varepsilon_f(w_q^{(M)})=-1$ for some prime $q$ dividing $\gcd(M,N/M)$ and $\varepsilon_f(w_p^{(M)})=1$ for all primes $p$ such that $v_p(N)=v_p(M)$.
 
 Define $\iota_M\colon X_0(M) \to J_0(M)$ the Abel--Jacobi map with respect to the base point $\infty$. Let $A_f$ the modular abelian variety associated to $f$ and  $\pi_f  : J_0(M) \rightarrow A_f$ the quotient map.
 
The map $G_{N,f} : X_0(N) \rightarrow A_f$ is defined by 
 \begin{equation}
 \label{eqn:formalimmersion}
  G_{N,f} = \sum_{{\substack{d \| (N/M)}}} \varepsilon_f (\varphi_{N,M}(w_d^{(N)})) \cdot \pi_f \circ \iota_M \circ i_{N,M}^{(d)}.
 \end{equation}
Then, $G_{N,f}$ factors through $X_0(N)^*$, and the induced map $X_0(N)^* \ra A_f$ is denoted by $G_{N,f}^*$. 
\end{prop}

\begin{rem}
    This definition generalizes Momose's definition for powers of a prime $p$. We flesh out the exact correspondence in this case.  Following the notations of \cite[\S\,1.4, p.~444]{Momose1986}, we fix here $N=p^r$ and $M=p^s$ for some prime $p$. In that paper, $\pi=\pi_{r,s} =i_{p^r,p^s}^{(1)}$ and the map $f_{r,s}: X_0(p^r) \rightarrow X_0(p^s)$ is defined by $f = [w_{p^s}^{(p^s)} \circ \pi] - [\pi \circ w_{p^r}^{(p^r)}]$, which functorially on pairs $(E,C)$ with $C$ cyclic of order $p^r$ means 
    \begin{equation}
    f(E,C)= (E/(C[p^s]),E[p^s]/C[p^s]) - (E/C,(E[p^s]+C)/C).
    \end{equation}
    With our choice above, as $\varepsilon_f(w_p^{(p^s)})=-1$, we get after a straightforward computation 
    \begin{equation}
G_{N,f}(w_{p^r}^{(p^r)} (E,C)) = -\pi_f (f_{r,s}(E,C)).
    \end{equation}
\end{rem}
\begin{proof}
	By construction, for any $w' \in \Wcal(M)$, $\pi_f \circ (w')_* = \varepsilon_f(w') \cdot  \pi_f$, and 
	\begin{equation}
	\iota_M \circ w' = (w')_* \circ \iota_M + [w'(\infty)] - [\infty] 
	\end{equation}
	so 
	\begin{eqnarray}
		\pi_f \circ \iota_M \circ w' - \varepsilon_f \cdot \pi_f \circ \iota_M & = & \pi_f \circ (w')_* \circ \iota_M + \pi_f([w'(\infty)] - [\infty]) - \varepsilon_f \cdot \pi_f \circ \iota_M \notag\\
		& = & \varepsilon_f \cdot\pi_f \circ \iota_M +  \pi_f([w'(\infty)] - [\infty])  - \varepsilon_f \cdot \pi_f \circ \iota_M \\
		& = & \pi_f([w'(\infty)] - [\infty]). \notag
	\end{eqnarray}
	Now, for any $w \in \Wcal(N)$, defining $w' = \varphi_{N,M}(w)$, by Proposition \ref{propdiagcommALandiNM},
	\begin{eqnarray}
	G_{N,f} \circ w &  =  & \sum_{{\substack{d \| (N/M)}}} \varepsilon_f(\varphi_{N,M}(w_d^{(N)})) \pi_f \circ \iota_M \circ i_{N,M}^{(d)} \circ w \notag \\
	& = & \sum_{{\substack{d \| (N/M)}}} \varepsilon_f(\varphi_{N,M}(w_d^{(N)})) \pi_f \circ \iota_M \circ w' \circ i_{N,M}^{(w \cdot d)}  \\
	& = & \sum_{{\substack{d \| (N/M)}}} \varepsilon_f(\varphi_{N,M}(w_d^{(N)})) (\varepsilon_f(w') \pi_f \circ \iota_M + [w'(\infty)] - [\infty]) \circ i_{N,M}^{(w \cdot d)} \notag \\
	& = & \sum_{{\substack{d \| (N/M)}}} \varepsilon_f(\varphi_{N,M}(w_d^{(N)})) \varepsilon_f(\varphi_{N,M}(w)) \pi_f \circ \iota_M \circ i_{N,M}^{(w \cdot d)} \notag\\
	& + & \sum_{{\substack{d \| (N/M)}}} \varepsilon_f(\varphi_{N,M}(w_d^{(N)}))\pi_f ([w'(\infty)] - [\infty]). \notag
	\end{eqnarray}
The second sum is zero since, by hypothesis, the character $\varepsilon_f \circ \varphi_{N,M}$ restricted to the subgroup of $\Wcal(N)$ made up with the $w_d^{(N)}, d \|(N/M)$ is nontrivial.

Now, for every $w \in \Wcal(N)$ and every Hall divisor $d$ of $(N/M)$, by Lemma \ref{lem:atkinlehnerandactionNM} and as $\varphi_{N,M}$ is a group morphism vanishing on the $w_p^{(N)}$ such that $p \nmid N/M$ by hypothesis on Atkin--Lehner signs, for $\varepsilon_f$, we have
\begin{equation}
\varphi_{N,M} (w_{w \cdot d}^{(N)}) = \varphi_{N,M}(w_d^{(N)}) \varphi_{N,M} (w). \qedhere
\end{equation}
We thus obtain
\begin{equation}
	G_{N,f} \circ w = \sum_{{\substack{d \| (N/M)}}} \varepsilon_f(\varphi_{N,M}(w_{w \cdot d}^{(N)})) \cdot \pi_f \circ \iota_M \circ i_{N,M}^{(w \cdot d)} = G_{N,f}. \qedhere
\end{equation}
\end{proof}

\begin{prop}
\label{propcotangentGNf}
	With the previous notation, let $\c = \Ccal_N(a/b)$ be a width 1 cusp of $X_0(N)$, and assume hypothesis $(HV)$.
	
	Let $\Rcal_{\c}$ the set of Hall divisors $d$ of $N/M$ such that $i_{N,M}^{(d)}$ is unramified at $\c$ (notice that for any $d \in \Rcal_\c$, $i_{N,M}^{(d)}(\c) = \infty_M \in X_0(M)$ by Proposition \ref{propcotangentunramified}). We have, for $S' = G_{N,f}(\c)$
	\begin{equation} \label{eq:sums of roots of unity}
	 \Cot_\c (T_{-S'} \circ G_{N,f}) = \left( \sum_{d \in \Rcal_\c} \zeta_{\gcd(b,d)} \right) \Cot_\c (\pi_f \circ \iota_M \circ i_{N,M}^{(1)}),
	\end{equation}
	as cotangent maps from the cotangent spaces respectively of 0 in $J_0(M)$ and $\c$ in $X_0(N)$, where for each $n$, $\zeta_{n}$ is a primitive $n$th root of unity. 
\end{prop}

\begin{proof}
	For any $d$ not in $\Rcal_\c$, the cotangent map of $i_{N,M}^{(d)}$ at $\c$ is 0 as $i_{N,M}^{(d)}$ is ramified by definition, hence the computation of the cotangent map of $G_{N,f}$ reduces to the computation on the subsum made up with the $d$'s of $\Rcal_\c$. For each of those, $i_{N,M}^{(d)} = w_{d_M}^{(M)} \circ i_{N,M}^{(1)} \circ w_{d_N}^{(N)}$ and by hypothesis, $\c' = w_{d_N}^{(N)} (\c)$ is still a width 1 cusp (and more precisely, for each $p|d$, $v_p(b) = v_p(N)/2$ so $\c' = \Ccal_N(a'/b)$ for some $a'$) by Corollary \ref{cor:width1cuspsandAL}. The image by $i_{N,M}^{(1)}$ is then $\infty$, and we use that 
	\begin{equation}
	\pi_f \circ \iota_M \circ w_{d_M}^{(M)} = \pi_f \circ (w_{d_M}^{(M)})_* \circ \iota_M + \pi_f(t) = \varepsilon_f (w_{d_M}^{(M)}) \pi_f \ \iota_M + \pi_f(t) 
	\end{equation} where $t$ is a degree zero sum of cusps in $X_0(M)$.

	By these arguments, up to translation (which corresponds to the sum of the constant terms $\pi_f(t)$), we can make use of Proposition \ref{propcotangentunramified} for the unramified terms and obtain the result (notice that for each term, another factor $\varepsilon_f(w_{d_M}^{(M)})$ comes from the definition of $G_{N,f}$, so they are squared in the sum for the computation of the cotangent map of $G_{N,f}$).
\end{proof}

\begin{cor}[Almost squarefree case]
\label{corasqfcotangentmaps}
    With the same notation as Proposition \ref{propcotangentGNf}, 
    assume furthermore that $N$ is an almost squarefree number with powerful prime $p$ and $M|N$ squarefree such that $p|M$. 

    $\bullet$ If $v_p(N)=3$ or $\c=\infty$, we have $\Rcal_\c = \{1\}$ so 
    \begin{equation}
\Cot_{\c} (T_{-S'} \circ G_{N,f}) = \Cot_\c(\pi_f \circ \iota_M \circ i_{N,M}^{(1)}).
    \end{equation}

    $\bullet$ If $v_p(N) = 2$ and $v_p(b) = 1$, we have $\Rcal_\c = \{1,p\}$ and 
    \begin{equation}
    \label{eq:cotmapasqfworstcase}
\Cot_{\c} (T_{-S'} \circ G_{N,f}) = (1 + \zeta_p) \Cot_\c(\pi_f \circ \iota_M \circ i_{N,M}^{(1)}).
    \end{equation}
\end{cor}

\begin{rem}
\label{rem:badcase2powerfulprime}
    With this corollary, when $4 \| N$ for $N$ almost non-squarefree (or more generally when $\Rcal_\Ccal = \{1,2\}$), the cotangent map of $G_{N,f}$ is actually 0 at cusps $\Ccal = \Ccal_N(a/b)$ with $v_2(b)=1$. This is consistent with the fact that such cusps are fixed by $w_4^{(N)}$ (see Corollary \ref{cor:fixedcuspsbyAtkinLehner}), thus the map $X_0(N) \rightarrow X_0(N)^\ast$ is ramified with order 2 at $\Ccal$, hence $G_{N,f}$ is also ramified at $\Ccal$ as it factors through the star quotient map.
\end{rem}
\begin{prop}\label{prop:formalimmersionsec2}

Assume and define the following: 

$\bullet$ $N$ is a nonexceptional almost squarefree level with powerful prime $p$ (Definitions \ref{defexceptionallevel} and \ref{defi:almostsquarefree}) such that $p \neq 2$ or ($p=2$ and $8 |N$).

$\bullet$ $M$ is a squarefree factor of $N$ such that $p|M$ and admitting a simple rank zero quotient $A_f$ in $J_0(M)^{-_p,+_ {(p)}}$ associated to some new eigenform $f \in S_2(\Gamma_0(M))$. We also define $\pi_f : J_0(M) \rightarrow A_f$ the quotient map.

$\bullet$ $G_{N,f} : X_0(N) \ra A_f$ is the associated map defined in Proposition \ref{propdefGNf} and $G^*_{N,f}: X_0(N)^* \ra A_f$ its factorisation through $X_0(N)^*$. 

By Néron mapping property, defining $\Jcal_0(N), \Jcal_0(M),\Acal_f$ the respective Néron models of $J_0(N),J_0(M),A_f$ over $\Z$, the maps $G_{N,f}$ and $G^*_{N,f}$ extend respectively to $\Xcal_0(N)^{\rm{smooth}}$ and $(\Xcal_0(N)^*)^{\rm{smooth}}$, towards $\Acal_f$.

Let $\c$ be a width 1 cusp of $X_0(N)$ defined over some number field $K$, $\lambda$ a prime ideal of $\Ocal_K$ with residual characteristic $\ell$ and $\c_\lambda \in \Xcal_0(N)_{\F_\ell}$ the reduction of $\c$ modulo $\lambda$.

Then, we have the following trichotomy if $\ell \neq 2$: 

$\bullet$  $\c_\lambda$ is not a smooth point of $\Xcal_0(N)_{\F_\ell}$.

$\bullet$ $\c_\lambda^*$ is a smooth point of $\Xcal_0(N)^*_{\F_\ell}$ but it is not defined over $\F_\ell$.

$\bullet$ $\c_\lambda$ is a smooth point of $\Xcal_0(N)_{\F_\ell}$, its image $\c_\lambda^*$ in $\Xcal_0(N)^*$ is a smooth point of $\Xcal_0(N)^*_{\F_\ell}$ defined over $\F_\ell$ and $G^*_{N,f}$ is a formal immersion at $\c_\lambda^*$.
\end{prop}

\begin{proof}
    If $\c_\lambda$ is not a smooth point of $\Xcal_0(N)_{\F_\lambda}$, we are in the first case and there is nothing to prove.

    If it is, as $\Xcal_0(N)^*$ is the quotient of $\Xcal_0(N)$ by $\Wcal(N)$ which is a 2-group made up with involutions, its image in $\Xcal_0(N)^*$ is also a smooth point as $\ell \neq 2$, by a result of Serre as the order of $\Wcal(N)$ is prime to $\ell$ (for the details, see \cite[Chapter V, Exercise 7 (p.~149)]{BourbakiLie}). 

We can assume from now on that $\c_\lambda^*$ is defined over $\F_\ell$ (otherwise we are immediately in the second case, nothing to prove either). The cotangent map of $G_{N,f}^*$ at $\c_\lambda^*$ is a factor of the cotangent map of $G_{N,f}$ at $\c_\lambda$, which is by Corollary \ref{corasqfcotangentmaps} on the geometric generic fiber (up to translation) the cotangent map of $\pi_f \circ \iota_M \circ i_{N,M}^{(1)}$ multiplied by some number which is either 1 or $1 + \zeta_p$ with $\zeta_p$ a primitive $p$-th root of unity. If $p \neq 2$,  notice that $1 + \zeta_p = \frac{1-\zeta_p^2}{1 - \zeta_p} \in \Z[\zeta_p]^\times$. If $p=2$ and $8\mid N$, then $\Rcal_\Ccal = \{1\}$ so the factor $1 + \zeta_2$ does not appear. In both cases,   the geometric cotangent map of $G^*_{N,f}$ at $\c_\lambda$ is surjective if and only if the geometric cotangent map of $\pi_f \circ \iota_M \circ i_{N,M}^{(1)}$ at $\c_\lambda$ is.

Now, let us go step by step to figure what this cotangent map is: we have the composition 
\begin{equation}
X_0(N) \overset{i_{N,M}^{(1)}}{\longrightarrow} X_0(M) \overset{\iota_M}{\longrightarrow} J_0(M)  \overset{\pi_f}{\longrightarrow} A_f,
\end{equation}
and since condition $(HV)$ holds, $i_{N,M}^{(1)} (\c) =  \infty$, and $i_{N,M}^{(1)}$ is unramified at $\c$ (at all reductions), so the (geometric) cotangent map of $i_{N,M}^{(1)}$ from $\c_\lambda$ to $\infty_\ell$ is surjective, so we only have to consider the (geometric) cotangent map of $\pi_f \circ \iota_M$ at $\infty_\ell$, but this is proven to be surjective by \cite[Proposition 3.1]{Mazur1977} since $\ell \neq 2$. Therefore, by Lemma \ref{lem:formimmcriterion} we have obtained that $G_{N,f}^*$ is a formal immersion at $\c^*_\lambda$.
\end{proof}

The following lemma will be used to control the error made when changing either cusps of possible (bad) reduction or level (which in both cases involves replacing the elliptic curve considered by an isogenous one).

\begin{lem}[Integrality under isogeny of elliptic curves]
\label{lem:integralityisogenousellcurves}
    Let $E$ and $E'$ be two elliptic curves defined over a nonarchimedean complete local field $K$ with valuation $v$, and $\varphi : E \rightarrow E'$ be a cyclic isogeny of degree $N$.

    Then, either $E$ and $E'$ have potentially good reduction, or they both have potentially multiplicative reduction and then, for some factorization $N = mn$ into integers, we have
    \begin{equation}
v(j(E')) = \frac{n}{m} v(j(E))<0.
    \end{equation}
\end{lem}
\begin{proof}
The potentially good reduction case is a consequence of Néron--Ogg--Shafarevich criterion (its negation implies that $|j(E)|>1$ if and only if $|j(E')|>1$). Let us thus assume that $E$ and $E'$ have potentially multiplicative reduction. By the theory of the Tate curve, one can then write $E(K) = K^*/q^\Z$ and $E' (K) = K^*/(q')^\Z$.

Following the arguments of \cite[p.~324 and 325]{Tate95} we must have $q^n = (q')^m$ for some pair of integers $m,n \in \Z_{>0}$ that we can and will assume are coprime. Up to inverse, there is only one cyclic isogeny between $E$ and $E'$ (given as $\alpha_{m,n}$ in Tate's notation), defined by making the following diagram commutative
\begin{equation}
    \begin{gathered}
        \xymatrix{
0 \ar[r] & \Z \ar[d]_{[m]} \ar[r]^{1 \mt q} & K^* \ar[d]^{z \mt z^n} \ar[r] & E(K) \ar[d]^{\alpha_{m,n}} \ar[r] & 0 \\
0 \ar[r] & \Z \ar[r]^{1 \mt q'} & K^* \ar[r] & E'(K) \ar[r] & 0 
}
    \end{gathered}
\end{equation}
Its kernel is then of order $mn$ by a standard application of the snake lemma, so we must have $N = mn$, and 
\begin{equation}
v(j(E')) = -v(q') = \frac{-1}{m} v(q^m) = \frac{-n}{m} v(q) = \frac{n}{m} v(j(E)). \qedhere
\end{equation}
\end{proof}

\begin{lem}
\label{lem:polyquadinertiaatp}
For any polyquadratic field $K/\Q$ and any prime number $p \neq 2$, any prime ideal $\gp$ of $\Ocal_K$ above $p$ has inertia $e(\gp/p) \leq 2$.
\end{lem}

\begin{proof}
 Consider the local extension $K_\gp/\Q_p$, it is Galois with Galois group $G$ equal to the decomposition group of $D_\gp$ in $\Gal(K/\Q)$ (thus itself a 2-torsion abelian group). Its two first ramification groups are $G_0 = I_\gp$ the inertia subgroup and $G_1$, and by \cite[Chapter II, Proposition 10.2]{Neukirch}, $G_0/G_1$ injects into the unit group $\kappa(\gp)^*$, where $\kappa(\gp)$ is the residue field of $\gp$. As follows from this reference, $G_1$ is actually a $p$-Sylow of $G_0$ and as $p$ is odd here, $G_1 = \{0\}$ and $G_0/G_1$ is cyclic, hence $I_\gp$ is cyclic and a subgroup of $G$, therefore it is of order 1 or 2.
\end{proof}

We are now ready to prove Theorem \ref{thm:integraljinv}. We start with the most general case (excluding in particular reduction at primes above 2).

Let $\widetilde{N}$ be a non-exceptional non-squarefree non-prime power level. 

By Theorem \ref{thm:rank0quotient}, there exists $N|_\square \widetilde{N}$ almost squarefree (whose powerful prime is denoted by $p$)  and $M|N$ squarefree such that $p|M$, and $J_0(M)^{-_p,+_{(p)}}$ has a nontrivial simple rank zero quotient $A$. We fix those from now on. By Lemma \ref{lem:integralityisogenousellcurves}, it is enough to prove that all non-cuspidal points on $X_0(N)^*(\Q)$ have potentially good reduction away from 2 so we work on $X_0(N)^*$ from now on.

\subsection{General case of Theorem \ref{thm:integraljinv}: integrality at $\ell \neq 2$ and for $v_2(N) \neq 2$}

As written in the title, we assume here that $\ell \neq 2$ and that either the powerful prime $p$ of $N$ is odd or that $p=2$ and $v_2(N)=3$.

\begin{proof}[Proof of Theorem \ref{thm:integraljinv} $(a)$]

Let $P^* \in Y_0(N)^* (\Q)$  and $P=(E,C_N) \in Y_0(N)(K)$ a lift of $P^*$. 
Let $\lambda$ be a prime ideal of $\Ocal_K$ with residual characteristic $\ell$. Let us assume that $P$ reduces modulo $\lambda$ to a cusp $\c$ of $\Xcal_0(N)$, and therefore $P^*$ reduces modulo $\ell$ to its image $\c^* \in \Xcal_0(N)^*$. Our goal is to obtain a contradiction.
Up to applying Atkin--Lehner involutions (which amounts to changing the choice of lift $P$), by Corollary \ref{cor:width1cuspsandAL} we can and will assume $\c$ is a width 1 cusp of $X_0(N)$.
Let us extend $K$ so that $\c$ is also defined over $K$, hence $\c_\lambda = P_\lambda$. First, $P_\lambda$ must be a smooth point of $\Xcal_0(N)_{\F_\ell}$ by Lemma \ref{lemregularsmooth}, so $\c_\lambda$ is a smooth point of $\Xcal_0(N)_{\F_\ell}$. Then, as $P^*$ is defined over $\Q$, we must have $\c^*_\lambda$ defined over $\F_\ell$ so we are in the third case of Proposition~\ref{prop:formalimmersionsec2} and $G^*_{N,f}$ is a formal immersion at $\c^*_\lambda$. Now, since $A_f$ is a rank zero quotient, and since $\ell \neq 2$ we obtain $P^* = \c^*$ by Proposition \ref{prop:formimmkeyresult}, which is absurd.
\end{proof}

\subsection{Special case $\ell=2$ with $v_2(N) \neq 2$}
\label{subsec:firstboundvaluation2}

The reduction above 2 raises several issues simultaneously, so at the cost of some loss of precision, we start by going back to Chabauty's method (of which Mazur's method is ultimately a deep refinement). 

We start the argument with the same notation as Proposition \ref{prop:formalimmersionsec2} ($N,M,A_f$ chosen with $N$ almost squarefree with powerful prime $p$). Let $P^* \in X_0(N)^*(\Q)$ choose a lift $P_0 \in X_0(N)(K)$ of $P^*$ with $K$ a polyquadratic field, and assume that modulo a prime ideal $\lambda|2$ of $\Ocal_K$, $P_0$ reduces to a cusp $\Ccal_0$ of $X_0(N)$.

Here also, one can use an Atkin--Lehner involution $w \in \Wcal(N)$ to ensure that $P := w \cdot P_0$ reduces to a width 1 cusp $\Ccal = w \cdot \Ccal_0$ (Corollary \ref{cor:width1cuspsandAL} $(ii)$).

First, arguments similar to \cite[paragraph (1.7)]{Momose1986} imply that as $k \leq 3$,  all cusps of $\Xcal_0(N)_{\F_2}$ are smooth points.

Choose $a,b \in \Rcal_N$ such that $\Ccal = \Ccal_N(a/b)$. If $p \nmid b$, $\Ccal$ is defined over $\Q$ and if $p|b$, by Proposition \ref{propactALoncusps}, if $\c_\lambda^*$ is defined over $\F_2$ it implies that $p=2$ or the multiplicative order of 2 modulo $p$ is 1 or 2, which only happens for $p=3$ but in all those cases $\c^*$ itself is defined over $\Q$ (and if $p\nmid b$, $\c$ is rational). In other words, $\c^*$ must be defined over $\Q$ in our situation.   
 
 Consider a good uniformizer $t_M$ at $\infty$ in $\Xcal_0(M)_{\Z_2}$ (see Definition \ref{def:goodunif}), whose image is a uniformizer at $\infty_{\F_2}$ in $X_0(M)^*_{\F_2}$. Its pullback by $i_{N,M}^{(1)}$, denoted by $t_N$, is then a good uniformizer at $\c$ in $X_0(N)_{\Z_2}$.

 In that situation, we can choose a Hecke operator $T \in \End J_0(M) \setminus 2 \End J_0(M)$ with image $A$ isogenous to $A_f$. We replace $(A_f,\pi_f)$ by $(A,T)$ below and define $G_{N,T}, G_{N,T}^*$ using formula \eqref{eqn:formalimmersion}.
 
 Applying Corollary \ref{corasqfcotangentmaps}, the cotangent map of $T_{-S'} \circ G_{N,T}$ at $\c$ is
 \begin{equation}
\Cot_\c(T \circ \iota_M \circ i_{N,M}^{(1)})
 \end{equation}
 as $p \neq 2$ or $8|N$ by hypothesis (see Corollary \ref{corasqfcotangentmaps}). By Mazur's arguments (mainly \cite[Proposition 3.1]{Mazur1978}), we now have that $T \circ \iota_M$ is a formal immersion at $\infty_{\F_2}$ (see also \cite{Merel07} where the idea of replacing quotients by abelian subvarieties originated, in particular for characteristic 2). Consequently, writing $\Acal$ the Néron model of $A$ over $\Z_2$, there is a integral 1-form $\omega \in \Omega^1_{\Jcal/\Z_2}$ such that
\begin{equation}
 (T \circ \iota_M)^* \omega = \sum_{n \geq 0} a_n t_M^n \textrm{d} t_M,
\end{equation}
and all $a_n$ belong to $\Z_2$, with $a_0 \neq 0 \mod 2$. Now, by construction, appealing to the theory of $2$-adic integration and 2-adic logarithm between $A(\Q_2)$ and $T_0(A_{\Q_2})$, we have the equality of 2-adic integrals
\begin{equation}
 \int_\c^P (G_{N,T})^* \omega = \int_{G_{N,T}(\c)}^{G_{N,T}(P)} \omega = 0 
\end{equation}
because $G_{N,T}(P) - G_{N,T}(\c) \in A(\Q)$ which is a torsion group. On the other hand, we can write 
\begin{equation}
(G_{N,T})^* \omega = \sum_{n \geq 0} a'_n t_N^n \textrm{d} t_N
\end{equation}
by expansion of this 1-form, and as this 1-form is integral, all coefficients $a'_n$ belong to $\Z_2$.

Consequently, we have the expansion of the 2-adic integral 
\begin{equation}
\label{eqfortP}
\int_\c^P (G_{N,T})^* \omega = \sum_{n \geq 0} a_n' \frac{t_N(P)^{n+1}}{n+1} = 0.
\end{equation}
Now, $P$ belongs to the 2-adic residue disk of $\c$ by hypothesis (so $|t_N(P)|_2<1$), and $a_0'$ is precisely given by the cotangent map of $G_{N,T}$ at $\c$, in other words $a'_0 = a_0$. By hypothesis $a_0$ has 2-adic valuation 0, so 
$v_2(a'_0)=0$. In that configuration, the most negative possible slope of the Newton polygon associated to the power series on the right-hand side of \eqref{eqfortP} would be $-1$ (obtained when $v_2(a'_0)=v_2(a'_1)=0$), so any (nonzero) solution in $\overline{\Z}_2$ for \eqref{eqfortP} satisfies $|t_N(P)|_2 \geq 1/2$, i.e. $|j(P)|_2 \leq 2$.

\subsection{Case $v_2(N)=2$}
\label{subsec:badcase4HalldivisorofN}
We now deal with the most involved case. We start with the same notation as in the previous paragraph, including $T \in \End J_0(M)$ which is not a strict integer multiple of any endomorphism, with image $A$ a subvariety of rank 0 of $J_0(M)$. We define as above $G_{N,T} \colon X_0(N) \rightarrow A$ the associated map.

For the reduction at width 1 cusps $\Ccal_N(a/b)$ with $4 \mid b$, there, the cotangent map has only one term and the proof proceeds exactly as above: we obtain potentially good reduction away from 2, and for every ideal $\lambda|2$, up to applying an Atkin--Lehner involution, 2-adic valuation at most $1$ for $t_N(P)$.

For width 1 cusps $\Ccal_N(a/b)$ with $v_2(b)=1$, the map $G_{N,T}$ is ramified, which prevents immediate application of the formal immersion criterion. We now explain how we can nevertheless obtain potentially good reduction.

For any prime $\ell$, at $\infty_{\F_\ell}$ in $X_0(M)$ a choice of good uniformizer is $q_M = q_\infty$, and at $\Ccal$ a choice of good uniformizer is $q_\Ccal := (i_{N,M}^{(1)})^* q_\infty$.
Take as above $\omega \in \Omega^1_{\Jcal/\Z_\ell}$ an integral 1-form whose expansion in $q_M$ is 
\[
(T \circ \iota_M)^* \omega =  \sum_{n \geq 0} a_n q_M^n \d q_M.
\]

By Corollary \ref{cor:fixedcuspsbyAtkinLehner} we have that $\Ccal$ is fixed by $w_4^{(N)}$, and a more precise statement: we have, for such a $\Ccal$, in the completed local ring at $\Ccal$, 
\[
(w_4^{(N)})^* q_\Ccal = - q_\Ccal,
\]
since $-1$ is the only root of unity of order 2 (here $r_2 = 2$). For any $d \neq 1,2$ dividing $N/M$,  by ramification arguments, $(i_{N,M}^{(d)})^* q_M$ is of valuation at least 3 in $q_\Ccal$ (since $p=2$ here). Using the previous formula, we also have $(w_4^{(N)})^* q_\Ccal^n \d q_\Ccal = (-1)^{n+1}  q_\Ccal^n \d q_\Ccal$. Now, up to translation, one can write $G_{N,T} = F_{N,T} +F_{N,T} \circ w_4^{(N)}$ given its definition (take $F_{N,T}$ the signed partial sum in \eqref{eqn:formalimmersion} for the $d \| N/M$ with $d$ odd), hence the even-order terms in the expansion of $G_{N,T}^* \omega$ vanish while the odd order terms appear twice. We thus have 
\[
(G_{N,T})^* \omega = (2 a_1 q_\Ccal + O(q_\Ccal^3)) \d q_\Ccal.
\]

We can then integrate between $\Ccal$ and $P$ by the same argument, and divide by $ q_\Ccal(P)^2$ to obtain
\[
a'_1 + \sum_{k \geq 1} a'_{k+1} \frac{q_\Ccal(P)^{2k}}{2k+2} = 0,
\]
with $a'_1 = a_1$ and $a'_k \in \Z_\ell$ for all $k$.

Notice now that $a_2(f)= - \varepsilon_f(w_2^{(M)}) = 1$ by \cite[Theorem 5]{AtkinLehner70} hence for any  1-form in $\Omega^1_{\Acal/\Z_\ell}$, we have $a_1 = a_0$ (see Proposition \ref{prop:isommodularforms1forms}). By hypothesis on $T$, one can now choose $\omega$ with $a_0 = 1$ hence $a_1 = 1$. For $\ell =2$, the first segment is of least possible slope $-2/2 = -1$, and for $\ell>2$, the first segment is of least possible slope $- \frac{1}{2\ell-2}>-1/2$ (obtained if $v_\ell(a'_{\ell})=0$), but by Lemma \ref{lem:polyquadinertiaatp}, we must have $v_\ell(j(P)) \in \frac{1}{2} \Z$, so we obtain again that $v_\ell(j(P)) \geq 0$, i.e. $j(P)$ is integral over $\Z_\ell$.

\subsection{Gathering all the cases and going back to general level for Theorem \ref{thm:integraljinv}$(a)$ and $(b)$}

In all cases, we have thus obtained potentially good reduction away from 2, and for any lift $P$ of $P^*$ and any prime $\lambda$ above 2, there is $w \in \Wcal(N)$ such that $|j(w \cdot P)|_\lambda \leq 2$. By Lemma \ref{lem:integralityisogenousellcurves}, we thus have $|j(P)|_\lambda \leq 2^N$.

Finally, going back to level $\widetilde{N} = Nd^2$ for some $d \geq 1$, lifting a point $\widetilde{P}^* \in Y_0(\widetilde{N})^*(\Q)$ to $\widetilde{P} \in Y_0(\widetilde{N})$ and on another hand going down to a $P \in Y_0(N)(\Qb)$ via a succession of morphisms $\psi_{N',N'/q^2}$, $\widetilde{P}$ and $P$ are isogenous via a cyclic isogeny of degree $d$, so another application of Lemma \ref{lem:integralityisogenousellcurves} at all primes dividing 2 gives that $2^{Nd} j(E)$ is algebraic above $\Z$, and for the statement of the theorem we roughly bound $Nd$ by $\widetilde{N}$.

\begin{rem}
 The use of Lemma \ref{lem:integralityisogenousellcurves} seems to  artificially multiply $\log |j(P)|_\lambda$ by a factor up to $N$ even when one has bounded it by $1/2$ for a given choice of $P_0$. In fact, in explicit examples, there are cases where this is exactly what happens: there can be four conjugate points $P$ of a given $P_0$, with $\log |j(P_0)|_\lambda = 2$ but for each of the four conjugates $P$, $\log |j(P)|_\lambda = 2 Q$ for a given Hall divisor of $N$ (and there are four them, all obtained in this way for the four conjugates of $P_0$). 
 For an example, see the points \eqref{eq:x075}.
\end{rem}

\subsection{Proof of Theorem \ref{thm:integraljinv} $(c)$: integrality at primes above 2}
\label{subsec:improvementsvaluation2}
Given that Theorem \ref{thm:integraljinv} $(a)$ is already proven above, we only need to obtain the potentially good reduction at 2 for this case, but this requires particularly careful work.

The principal obstruction to Mazur's strategy is the fact that a 2-torsion point generating a $\mu_2$ reduces to 0 modulo 2, and is compatible with all formal immersion type results (see Proposition \ref{prop:formimmkeyresult}). Therefore, the goal is to prove that the image of a rational torsion point $G_{N,T}(P)$ does not generate a $\mu_2$. There are two natural strategies: either one can prove that $G_{N,T}(P)=0$ in some cases, or that $G_{N,T}(P)$ has large order in other cases.

In what follows, we will prove that $G_{N,T}(P)=0$ by proving that its image by an injective morphism (related to the reduction to a group of components at the special fiber of a Néron model) must be 0.

In this subsection, we make the following assumption: 
\begin{center}
$p^2 |N$ for some prime $p \notin \{2,3,5,7,13\}$, and we fix $M=p$.
\end{center}

\begin{prop}
\label{prop:cuspidalsubgroupinjectsintocomponentsgroupprimelevel}
    Let $p \notin \{2,3,5,7,13\}$ prime, $K$ a finite extension of $\Q_p$ of absolute ramification $e$, $(R,\gp)$ its integral ring and $\gp$ its maximal ideal. 

    Then: 
    
    $(a)$ The group of irreducible components $\Phi_{K,p}$ of the Néron model $\Jcal_0(p)_R$ of $J_0(p)$ above $R$ is of the following shape 
    \begin{equation}
\Phi_{K,p} \cong \Z/ne\Z \times (\Z/e\Z)^{S-2}
    \end{equation}
    where $n = \operatorname{num}((p-1)/12)$ and $S$ the number of supersingular points in $\Xcal_0(p)(\Fpb)$.

    $(b)$ The cuspidal subgroup of $J_0(p)$ (generated by $[0]-[\infty]$) injects in $\Phi_{K,p}$ via the following composite map: 
    \begin{equation}
J_0(p)(\Q) \overset{\subset}{\rightarrow} J_0(p)(K) = J_0(p)(R) {\rightarrow} \Jcal_0(p)_R(k) \rightarrow \Phi_{K,p},
    \end{equation}
 and its image is thus a cyclic group of order $n$ in $\Phi_{K,p}$.
\end{prop}

\begin{proof}
    The results are immediate consequences of  \cite[Proposition 2.11]{LeFourn1}, itself a generalization of \cite[Theorem A.1$(b)$]{Mazur1977} (which treated the case $e=1$) describing completely the group of components $\Phi_{K,p}$. (To be clear, $Z = [\Ecal_0^{(p)}], Z' = [\Ecal_1^{(p)}]$ and $\overline{Z}=Z-Z'$ in the notation of that statement). 
\end{proof}

\begin{prop}
    Let $\widetilde{N} \geq 1$ such that some prime $p \notin \{2,3,5,7,13\}$ has multiplicity at least 2 in $\widetilde{N}$.

    For any non-cuspidal point $P^* \in X_0(\widetilde{N})^*(\Q)$ and $P=(E,C_{\widetilde{N}}) \in X_0(\widetilde{N})(K)$ a lift of $P^*$, any prime ideal $\lambda$ above 2 in $\Ocal_\lambda$, $P$ has potentially good reduction at $\lambda$. In other words, $j(E) \in K$ is integral at all primes above 2.
\end{prop}

\begin{proof}
    First, after descending by square factors towards an almost squarefree $N$ with powerful prime $p$, one can assume $\widetilde{N} = N$ (this will not change integrality of $j(E)$ at 2 by Lemma \ref{lem:integralityisogenousellcurves}).

    In the setup of Proposition \ref{propdefGNf}, we can now fix $M=p \notin \{2,3,5,7,13\}$ for a choice $p$ such that $p^2 | \widetilde{N}$, and instead of using a quotient map $\pi_f : J_0(p) \rightarrow A_f$, use some endomorphism $T \in \End(J_0(p))$ factoring through $A_f$ (we will denote $G_{N,T}$ instead of $G_{N,f}$ for the composite map built the same way).

The key to tackle the reduction at 2 case here is the following lemma, inspired from Momose's arguments in the prime power case \cite[Proposition (3.3)]{Momose1986}.

    \begin{lem}
    With the previous notation, we have $G_{N,T}(P)=0$ in $J_0(p)(\Q)$.
    \end{lem}

    \begin{proof}
           First, $x := G_{N,T}(P)$ is by construction a torsion point of $J_0(p)(\Q)$, thus it belongs to the cuspidal subgroup of $J_0(p)$ by Mazur's proof of Ogg's conjecture in the prime level case \cite[Theorem 1]{Mazur1977}. 

        Let us fix a prime ideal $\gp$ of $\Ocal_K$ above $p$, and consider $K_\gp$ the finite extension of $\Q_p$, $R$ its ring of integers, and $k$ its residue field. By injectivity of the map from the cuspidal group to $\Phi_{K_\gp,p}$ (Proposition \ref{prop:cuspidalsubgroupinjectsintocomponentsgroupprimelevel}), it is enough to prove that the reduction of $x$ modulo $\gp$ is 0 in $\Phi_{K_\gp,p}$.

Since $K$ is a polyquadratic field (and $p>2$), the absolute ramification $e$ of $K_\gp$ above $\Q_p$ is at most 2 by Lemma \ref{lem:polyquadinertiaatp}.  Now, defining $r = v_p(N) \in \{2,3\}$, for every $d'|(N/(p^r))$, the reduction of 
        \begin{equation}
        P_{d'} := i_{N,p^r}^{(d')}(P) \in X_0(p^r)(K)
        \end{equation}
        in the minimal model $\widetilde{\Xcal_0(p^r)}_\gp(k)$ belongs to either $\Ecal_0^{h,(p ^r)}$ or $\Ecal_{r}^{h,(p^r)}$ by Proposition \ref{prop:irredcomponentsX0pr} since $p>7$.

        Since the only prime factor of $M=p$ is $p$ here, one can rewrite $G_{N,T}(P)$ as 
        \begin{equation}
G_{N,T}(P) = 
T\left( \sum_{d' \| (N/p^r)} \operatorname{cl} \left[ i_{p^r,p}(P_{d'}) - i_{p^r,p} ^{(p^{r-1})}(P_{d'}) \right] \right).
        \end{equation}
        Now, $T$ factors through $J_0(p)^-$ by the properties of the Eisenstein quotient, so one can write $T = T' \circ (1 - w_{p}^{(p)})$ for some $T' \in \End J_0(p)$. By Corollary \ref{cor:key0caseforextremecomponents}, since each reduction of $P_{d'}$ modulo $\gp$ in the minimal model belongs to $\Ecal_{0}^{h,(p^r)}$ or $\Ecal_r^{h,(p^r)}$ the image of each $(1 - w_p^{(p)})(i_{p^r,p}(P_{d'}) - i_{p^r,p} ^{(p^{r-1})}(P_{d'}))$ gives 0 in the formal sum of components, i.e.\ in the group of components $\Phi_{K_\gp,p}$ by Proposition \ref{prop:Neronmodelminmodel}. Consequently, the image of $x$ in $\Phi_{K,p}$ is 0, which concludes the proof of the lemma.
    \end{proof}
With this lemma, we now have $G_{N,T}(P) = G_{N,T}(\infty) = 0$.

Assume now that $P$ has the same reduction modulo $\lambda$ as some cusp $\Ccal$ of $X_0(N)$. 

By the arguments of  \cref{subsec:firstboundvaluation2}, since $P^* \in X_0(N)^*(\Q)$, this imposes that $\F_\lambda=\F_2$ and $\Ccal^*$ is $\Q$-rational, and under our hypothesis on $N$, by  \cref{propactALoncusps}, $\Ccal^*$ is the Atkin--Lehner orbit of $\infty$. Therefore, up to applying an Atkin--Lehner involution, let us assume that $\Ccal = \infty$.

The arguments of  \cref{subsec:firstboundvaluation2} then prove that $G_{N,T}$ is a formal immersion at $\infty_{\F_2}$ for $T \notin 2 \End J_0(M)$. By the lemma above, $G_{N,T}(P)=0$, hence by Proposition \ref{prop:formimmkeyresult} we obtain  $P=\infty$,  a contradiction. We have thus proven that $P$ has potentially good reduction at $\lambda$ for every $\lambda|2$.
\end{proof}

\begin{rem}
    The strategy could in principle be adapted to non-prime $M$, but we have not done it in this paper for the following reason. In the simplest case we are interested in, $M=pq$ with $p \in \{2,3,5,7,13\}$ and $q \neq p$ prime such that $(p,q)$ is not exceptional. 

    The main roadblock here is that the generalized Ogg's conjecture is unknown for such $M$: one knows the cuspidal subgroup (made up with only rational cusps as $M$ is squarefree here), one can compute the group of components of the Jacobian at primes above $p$ ($\Phi_{K,p}^{(M)} $) or $q$ ($\Phi_{K,p}^{(M)} $) in the same fashion, but one cannot ensure that the map induced by reductions at $p$ and $q$
    \begin{equation}
    J_0(pq)^{-_p,+_q}(\Q)_{\rm{tors}} \rightarrow \Phi_{K,p}^{(M)} \times \Phi_{K,q}^{(M)}
    \end{equation}
    is injective, which is the key to the argument above.

    Precise results exists about the $\ell$-primary part of the torsion rational subgroup when $\ell$ is an odd prime \cite{Ohta14}, but the 2-primary part is yet unknown (it was already the most difficult part of the prime level case).
\end{rem}

\section{Dealing with exceptional levels}
\label{sec:exclevels}

For the exceptional levels $N$ (introduced in Definition \ref{defexceptionallevel}), it is impossible to reduce to an almost squarefree level with an admissible rank zero quotient. We do an analysis on those levels in this section to obtain either good reduction type results or a full determination of the nontrivial rational points on those levels.

\subsection{Determining minimal exceptional levels}
\label{subsec:detexclevels}
\begin{prop} \label[prop]{exceptional levels}
	The exceptional (non-squarefree, not prime powers) levels are exactly the levels of the following shape: 
	
	\begin{enumerate}
		\item[(1)] One multiple factor: $\widetilde{N} = p^\ell Q$ with $\ell> 1$ and $(p,Q)$ exceptional or $p=2$, $Q=15$.
		
		\item[(2)] Only two (small) prime factors: $\widetilde{N} = p_1^{\ell_1} p_2^{\ell_2}$ with $\ell_1,\ell_2>1$ and $p_1,p_2 \in \{2,3,5,7,13\}$. If $\ell_1$ (resp.\ $\ell_2$) is odd, $(p_2,p_1)$ (resp.\ $(p_1,p_2)$) is exceptional (if both are odd, both pairs are exceptional). 
		
		\item[(3)] Two small prime factors:  $\widetilde{N} = p_1^{\ell_1} p_2^{\ell_2} q$ with $\ell_1 >1, \ell_2>1$ and 
		\begin{equation}
		(p_1,p_2,q) \in \{(2,3,5),(2,3,11),(2,7,3),(3,5,2)\}.
		\end{equation}
		Furthermore, $\ell_1$ must be even in all cases, and $\ell_2$ even in the last three cases.
		
		\item[(4)] Square: $\widetilde{N} = R^2$ with all primes dividing $R$ belonging to $\{2,3,5,7,13\}$.
		
		\item[(5)] Powerful and an odd power: $\widetilde{N} = 2^{\ell_1} 3^{\ell_2} 5^{\ell_3}$ or $5^{\ell_1} 2^{\ell_2} 3^{\ell_3}$ or $3^{\ell_1} 2^{\ell_2} 7^{\ell_3}$ with  $\ell_1 \geq 3$ odd and $\ell_2,\ell_3 \geq 2$ even.
	\end{enumerate}

 To the list of exceptional levels, we add the two prime powers $p^k$ such that $k \geq 2$ and $X_0(p^\alpha)^+$ is of positive genus but its Jacobian has no quotient of rank zero (which are $5^3$ and $13^2$).
\end{prop}

\begin{proof}
	Assume $\wN$ is exceptional. As there is no $p \notin \{2,3,5,7,13\}$ such that $p^2 | \wN$, we can write it as follows:
	\begin{equation}
	\wN = p_1^{\ell_1} \cdots p_{r}^{\ell_r} q_1 \cdots q_s
	\end{equation}
	with $r+s \geq 2$ as $\widetilde{N}$ is not a prime power, $p_1, \cdots, p_r \in \{2,3,5,7,13\}$, $\ell_i >1$ for all $i$ and $q_1, \cdots, q_s$ other primes (which can be in $\{2,3,5,7,13\}$ or not).
	
	For each $i$, after dividing by squares we can find $N |\wN$ almost squarefree such that $\wN/N$ is a square and of the shape 
	\begin{equation}
	N = p_i^{k_i} q_1 \cdots q_s P_i
	\end{equation}
	where $k_i=2$ or 3 and $P_i$ is the product of $p_j$'s different from $p_i$ such that $\ell_j$ is odd. Denote $Q_i = q_1 \cdots q_s P_i$, which is in particular squarefree.
	
	By hypothesis, each $(p_i,q)$ with $q | Q_i$ has to be exceptional and this for every choice of $i$, and if $Q_i$ has more than one prime factor we must have $p_i=2$ and $Q_i=15$.
	
	This immediately forbids the factorization of $\tilde{N}$ from simultaneously having $r \geq 2$ and $s \geq 2$: indeed, there are no choices of distinct primes $p_1,p_2,q_1,q_2$ such that $(p_i,q_j)$ is exceptional for every $i,j \in \{1,2\}$ and that $(p_i,q_1q_2)$ is exceptional for $i=1,2$. We thus have $r=1$ or $s \leq 2$. Using similar arguments, we can fill the following table for possible values or $r$ and $s$ (and which cases of the five claimed possibilities they cover):
	\begin{equation}
	\begin{array}{|c|c|c|c|}
		\hline
		(s,r) & 1 & 2 & \geq 3 \\
		\hline
		0 & p^\alpha & (2),(4) & (4),(5)\\
		\hline
		1 & (1) & (3) & \rm{X} \\
		\hline
		2 & (1) & \rm{X} & \rm{X} \\
		\hline
		\geq 3 & \rm{X} & \rm{X} & \rm{X} \\
		\hline
		
	\end{array}
	\end{equation}
	let us give some more details. The case $(s,r)=(2,1)$ gives $p_1=2, Q_1=15$. The case $(s,r)=(1,1)$ gives the rest of $(1)$. The case $(s,r)=(0,2)$ gives $(4)$ if both powers are even, $(2)$ otherwise. The case $(s,r) = (0,3)$ gives $(4)$ if all powers are even and $(5)$ if one is odd (two or more odd powers is impossible because there is no triple of distinct primes $(p_1,p_2,p_3)$ from which all pairs are exceptional).
	This just leaves to study the two cases $\widetilde{N} = p_1^{\ell_1} p_2^{\ell_2} q$ or $s=0$.  Finally, the case $(s,r)=(1,2)$ gives us $(3)$. In each of those case, a precise analysis then gives the exact possibilities of prime factors.
\end{proof}

	\begin{defi}
		Let $\Ecal$ be the set of exceptional levels in $\N$. 
		
		A \emph{minimal family} of $\Ecal$ is a family $\Fcal \subset \Ecal$ such that for every $\widetilde{N} \in \Ecal$, if $g(X_0(\wN)^*)>0$ there exists $N \in \Fcal$ such that $g(X_0(N))^*)>0$ and $\wN/N$ is a square.
		
		For example, if $\Ecal$ was only the set of $2^\ell \cdot  7$, a minimal family would be given by $\{ 2^4\cdot7, 2^5 \cdot 7\}$, because $g(X_0(2^i \cdot 7)^*) = 0$ for $i < 4$.
	\end{defi}

A Magma computation, starting with a set of levels that for each of the above five shapes contains minimal ones and removing levels that reduced to others in the set, gives the following initial minimal family of minimal exceptional levels
\begin{eqnarray}
    \Lcal_0 & = & \{40, 48, 72, 80, 88, 96, 99, 100, 108, 112, 120,125, 135, 144, 147, 162, 169,176, \notag\\
    & & 180, 184, 196, 200, 216, 224, 225, 240, 250, 297, 324, 368, 405, 441, 450, 486,  
  \\
&  & 1029, 1225,1250, 1372\}.\notag
\end{eqnarray}

We modify the list slightly, because of the following facts (see Proposition \ref{exceptional levels} for determining which level is exceptional).

\begin{itemize}
    \item $X_0(99)^*$ is an elliptic curve with positive Mordell--Weil rank \cite[Newform orbit 99.2.a.a]{lmfdb}, and $X_0(99p^2)^*$ is a nonexceptional level except for $p=2,3$. 
    \item  $X_0(125)^+$ has exactly one nontrivial rational point \cite{ArulMuller22} and for all primes $p$ except $2,3$, $X_0(125p^2)^*$ is nonexceptional. 
    \item  $X_0(169)^+(\Q)$, $X_0(81)^+(\Q)$ and $X_0(7^3)^+(\Q)$ are known to be trivial \cite{Balakrishnanetc,BiluParent}, hence also $X_0(324)^*(\Q)$ and $X_0(1372)^*(\Q)$ by Corollary \ref{cor:trivialityandabovelevels}.
\end{itemize}
We thus define $\Lcal_1 = (\Lcal_0 \setminus \{99,125,169,324,1372\} )\cup \{ 396, 500,891,1125\}$.

\subsection{Large exceptional levels: adapting Mazur's method to bound the denominator}

For levels $N \geq 250$ in $\Lcal_1$ (this threshold corresponds to the computability of models on smaller levels, see next subsections), we want to reuse Mazur's method with the caveats brought by the multiplicity of $p$, using in all those cases the presence of a rank zero quotient in the Jacobian but with some complications on the degeneracy maps.

Checking the LMFDB and looking at newforms of level $M$ for divisors $M$ of $N$ satisfying $(HV)$ (and their Atkin--Lehner signs and central value of L-function), we obtain the following.

\begin{prop}[Admissible divisors for exceptional levels]
For each $M,N$ as in the following table, $M$ is an admissible divisor of $N$ (see Definition \ref{def:adimissible}).
\begin{table}[h]
    \centering
    \renewcommand{\arraystretch}{1.3} 
    \begin{tabular}{|c|c|c|c|c|c|c|c|c|c|c|c|c|c|c|}
        \hline
        \( N \) & \( 250 \) & \( 297 \) & \( 368 \) & \( 396 \) & \( 405 \) & \( 441 \) & \( 450 \) & \( 486 \) & \( 500 \) & \( 891 \) & \( 1029 \) & \( 1125 \) & \( 1225 \) & \( 1250 \) \\
        \( M \) & \( 50 \) & \( 99 \) & \( 92 \) & \( 66 \) & \( 45 \) & \( 21 \) & \( 15 \) & \( 54 \) & \( 50 \) & \( 99 \) & \( 147 \) & \( 75 \) & \( 35 \) & \( 50 \) \\
        \hline
    \end{tabular}
    \caption{Values of \( N \) and \( M \)}
    \label{tab:NM_values}
\end{table}
\end{prop}
\begin{rem}
    All $N \geq 250$ in $\Lcal_1$ are in the table above, and for the lower levels in the list, only $80,88,96,135, 176, 184,224$ have admissible divisors. 
\end{rem}

\begin{defi}[Integrality factor for admissible divisors]
\label{defintfactor}
    For $N$ and $M$ belonging to the list above, we define $G_{N,f}$ following this choice of $(N,M)$ and an admissible quotient $A_f$. Then, define $m_{N,M}$ as the smallest integer such that for every cusp $\c$ of $X_0(N)$ of width 1, any sum of roots of unity given by Proposition \ref{propcotangentGNf} (with the exact formula for each roots of unity given in \eqref{eq:rootsofunitywd}) divides $m_{N,M}$.

    The extended number $m'_{N,M}$ is then $m_{N,M}$ multiplied by the lcm of its radical and 2.
    Explicitly, we have the following table:

\end{defi}

  \begin{table}[h]
      \centering
  \begin{tabular}{|c|c|c|c|c|c|c|c|c|c|c|c|c|c|c|}
          \hline
    $N$ & 250 & 297 & 368 & \textbf{396} & 405 & 441 & 450 & 486 & \textbf{500}
  & 891 & 1029 & 1125 & 1225 & 1250 \\
    $M$ & 50 & 99 & 92 & 66 & 45 & 21 & 15 & 54 & 50 & 99 & 147 & 75 & 35 & 50
  \\
  $m_{N,M}$ & 1 & 1 & 2 & \textbf{0} & 1 & 1 & 61 & 1 & \textbf{0} & 1 & 1 & 1 &
  1 & 1 \\
  $m_{N,M}'$ & 2 & 2 & $2^2$ & - & 2 & 2 & $2 \cdot 61^2$ & 2 & - & 2 & 2 & 2 & 2 & 2 \\
          \hline
    \end{tabular}
    \caption{Table of values for \( N \), \( M \), \( m_{N,M} \), and \(
  m'_{N,M} \)}
  \label{tab:values}
    \end{table}

\begin{rem}
    For $N=396$ and $N=500$ (where $4 \| N$), we encounter exactly the same issue as in the analogous almost non-squarefree case: in other words, the cotangent map is 0 at some cusps, and computations reveal that the sum of roots of unity involved is once again $1 + \zeta_2=0$.

    We leave aside these two cases for now, and work out the other ones.

    For almost all other ones, the norm is $\pm 1$, and there is a heuristic explanation for that: for any odd squarefree $D$, if we choose $\zeta_d$ a $d$-th root of unity for $d|D$  such that $\zeta_{dd'} = \zeta_d \zeta_{d'}$ for coprime $d,d'$ dividing $D$, the sum of these roots is then
    \[
\sum_{d|D} \zeta_d = \prod_{p|D} (1 + \zeta_p) \in \Z[\zeta_D]^\times,
    \]
    and its algebraic norm is $\pm 1$. The roots of unity appearing are not always multiplicative though: for $N=450$, we obtain, for the cusp $\Ccal_N(7/30)$, we obtain $1 + \zeta_{30}^{10} + \zeta_{30}^{18} + \zeta_{30}^{2}$ whose algebraic norm is $\pm 61$.

    Finally, for $N=368$, at the width 1 cusp $\Ccal_N(1/92)$, we get $1 + \zeta_4 = 1 \pm i$, of norm 2. 
\end{rem}

\begin{lem}
\label{lemNewtonpolygons}
        Let $p$ be a prime number and $f= \sum_{n \geq 0} \frac{a_n}{n+1} T^{n+1}.$ with each $a_n$ in $\overline{\Z_p}$, and $v(a_0) = \alpha < + \infty$ (the valuation being normalized so that $v(p)=1$).

    Then, the smallest nonzero root of $f$ in $\overline{D}(0,1)$ has valuation at most $\alpha + \log(2)/\log(p)$.
\end{lem}

\begin{proof}
    The first segment of the Newton polygon has slope
    \begin{equation}
\frac{v(a_n/(n+1)) - \alpha}{n} \geq \frac{- \alpha - v(n+1)}{n} \geq \frac{- \alpha - \log(n+1)/\log(p)}{n}
    \end{equation}
    for some $n \geq 1$.

    The right-hand side is always increasing in $n$ on $[1,+\infty[$, so the lowest possible value of that slope is at least 
    \begin{equation}
- \alpha - \log(2)/\log(p) .
    \end{equation}
    The smallest nonzero root of $f$ thus has valuation at most $\alpha + \log(2)/\log(p)$, by the theory of Newton polygons.
\end{proof}

\begin{thm}
\label{prop:integralityexclevel}
    Let 
    \begin{equation}
    N \in \{250, 297,368,396, 405,441,450,486,500,891,1029,1125,1225,1250\}
    \end{equation}
    and let the corresponding $M,m_{N,M}$ and $m'_{N,M}$ be as in Definition \ref{defintfactor}.

    Let a non-cuspidal point $P^* \in X_0(N)^* (\Q)$ and a lift $P=(E,C_N) \in X_0(N)(K)$ of $P$. We then have:
    
    \begin{itemize}
        \item For $N \neq 396,500$, the number $(m'_{N,M})^N j(E) \in \Ocal_K$.
        \item For $N=396,500$, $2^N j(E) \in \Ocal_K$.
        
        \end{itemize}
    In particular, for every odd prime $p \nmid m_{N,M}$, $E$ has potentially good reduction at every prime above $p$.
\end{thm}

\begin{proof}
Let us assume first $N \neq 396,500$.

    We reproduce the structure of proof of Proposition \ref{prop:formimmkeyresult} and Theorem \ref{thm:integraljinv}, going over the small differences in this situation. 

    For Proposition \ref{prop:formalimmersionsec2}, the assumption that $M$ is an admissible divisor allows us to use the previous result and obtain the same trichotomy in characteristic $\ell$ when $\ell$ is odd and does not divide $m_{N,M}$. In that situation (which is the generalization of case $(a)$ in the proof of the Theorem \ref{thm:integraljinv}) we proceed exactly in the same way, and obtain potentially good reduction of $E$ at every prime of residual characteristic $\ell$. 

    Assume now that $\ell$ divides $m'_{N,M}$. We will use the strategy of case $(b)$ of the proof of Theorem \ref{thm:integraljinv}. Extend the base field $K$ such that the cusp $\c$ is defined over $K$, fix $\lambda$ a prime ideal of $K$ above $\ell$ and assume that $P_\lambda = \c_\lambda$. Apply an Atkin-Lehner involution $w \in \Wcal(N)$ such that $w \cdot \c $ is of width 1 (and replace $P$ by $w \cdot P$). As $P_\lambda^* = \c_\lambda^*$ must be defined over $\F_\ell$ (as the reduction of $P^*$ modulo $\ell$), we can apply the same arguments to prove that $\c_\lambda$ is a smooth point of $\Xcal_0(N)_{\F_\ell}$, and it is sent by $i_{N,M}^{(1)}$ to $(\infty_M)_{\F_\ell}$. We again fix a good uniformizer $t_M$ at $(\infty_M)_{\F_\ell}$ (this cusp is always a smooth point) and $t_N$ its pullback by $i_{N,M}^{(1)}$, which is then a good uniformizer at $\c_\lambda$ (over $\Ocal_\lambda$).

    The notations $|\cdot|$ and $v$ below denotes the norm on $K_\lambda$ and its valuation, normalized by $|\ell|=1/\ell$ and $v(\ell)=1$.
    Applying Proposition \ref{propcotangentGNf} in the same way, the cotangent map of $G_{N,f}$ at $\c$ is of the shape 
    \begin{equation}
S_\c \Cot_\c(\pi_f \circ \iota_M \circ i_{N,M}^{(1)})
    \end{equation}
    where $S_\c$ is a sum of roots of unity, dividing $m_{N,M}$ by construction. As $\pi_f \circ \iota_M$ is a formal immersion at $(\infty_M)_{\F_\ell}$, we can choose again an integral 1-form $\omega \in \Omega^1_{\Acal_f/\Ocal_\lambda}$ such that 
    \begin{equation}
(\pi_f \circ \iota_M)^* \omega = \sum_{n \geq 0} a_n t_M^n d t,
    \end{equation}
    with every $a_n \in \Ocal_\lambda$ and $a_0$ a unit, and as in the almost-squarefree case we have 
    \begin{equation}
    (G_{N,f})^* \omega = \sum_{n \geq 0} a'_n t_N^n d t_N
    \end{equation}
    where all $a'_n$ belong to $\Ocal_\lambda$ and $v(a'_0) \leq v(m_{N,M})$. by construction, we again have the $\lambda$-adic integral of this 1-form between $\c$ and $P$ which is 0 (as $G_{N,f}(P) - G_{N,f}(\c) \in A_f(\Q)$ which is torsion), hence $t_N (P)$ gives a root of the power series
    \begin{equation}
    \sum_{n \geq 0} a'_n \frac{T^{n+1}}{n+1}.
     \end{equation}
     By Lemma \ref{lemNewtonpolygons} (and bounding roughly $\log(2)/\log(p)$ by 1), we thus know that either $P = \c$ (which is impossible) or $|t_N(P)|\geq \ell^{- v(a'_0) - 1}$ so that $|j(P)| \leq \ell^{v(a'_0)+1}$ by the same argument, which implies by definition of $m'_{N,M}$ that $|m'_{N,M} j(P)| \leq 1$, as claimed. 

     We have proven that for an arbitrary lift $P$ of $P^*$, for every prime $\lambda$, there is $w \in \Wcal(N)$ such that $|j(w \cdot P)|_\lambda \leq |m'_{N,M}|^{-1}$. Applying Lemma \ref{lem:integralityisogenousellcurves}, we thus have $(m'_{N,M})^N j(P) \in \Ocal_K$ for every lift $P$ of $P^*$ in $X_0(N)$.

     Finally, let us deal with the cases $N=396$ or $N=500$.
In that case, one has $4 \| N$ and $v_2(N/M)=1$, so that at width 1 cusps $\Ccal_N(a/b)$ with $v_2(b)=1$, the corresponding sum of roots of unity is $1+ \zeta_2=0$ for 500 and for $N=396$ if $3^2 | b$, and $1 + \zeta_2 + \zeta_3 + \zeta_6=0$ (after computational verification) for $N=396$ and $b = 2^2 \cdot 3 \cdot 11$ (for width 1 cusps with $v_2(b) = 2$, the sum of roots of unity is 1 so the formal immersion method applies without issue). 
We repeat the strategy of paragraph \ref{subsec:badcase4HalldivisorofN} and refer to it. In particular, we can then again rewrite $G_{N,T} = F_{N,T} + F_{N,T} \circ w_4^{(N)}$, and the obtain $G_{N,T}^* \omega = 2 u (a_2(f) q_\Ccal + O(q_\Ccal^3) )\d q_\Ccal$ where $u=1$ unless $N=396$ and $b = 2^2 \cdot 3 \cdot 11$ then it is $1 + \zeta_3$, an algebraic unit in all cases. We finally obtain that $j(P)$ is integral outside primes above 2, and $|j(w \cdot P)|_\lambda \leq 2$ for $\lambda|2$ and for some $w \in \Wcal(N)$.

As before, gathering the cases, we deduce that $2^N j(P)$ is integral for every lift $P$ of $P^*$.
\end{proof}

\section{Rational Heegner points and the analysis of the small exceptional levels}
\label{sec:heegnerpts}

For the small exceptional levels $N \leq 240$ in $\Lcal$, we can compute models for $X_0(N)^*$ and determine their rational points.
However, the modular interpretation of the rational points is not obvious, since the $j$-map $X_0(N) \rightarrow \P^1$ and the maps $X_0(N) \ra X_0(N)^*$ can be quite complicated even for small levels.

This motivates the current section, whose goal is to explain the existence of most of these rational points as $\Q$-rational cusps or as coming from Heegner points.

Unfortunately, the common references for Heegner points do not cover the case where the discriminant of the CM order is not coprime to the level of the modular curve. 
To generate enough rational Heegner points, we need to consider this level of generality. 
Hence, we start with some theory of Heegner points below, before presenting the algorithm to compute all Heegner points on $X_0(N)$ yielding rational points on $X_0(N)^*$.
This might also be of independent use to find all rational Heegner points on some families of modular curves.

\subsection{Theory of Heegner points}

 \begin{defi}[Heegner points]
     For any level $N \geq 1$, a \emph{Heegner point of level $N$ and discriminant $D$} is a point $P$ on $Y_0(N)(\C)$ associated to an isogeny $E \ra E'$ (with cyclic kernel of degree $N$) of complex elliptic curves such that $\End(E) \cong \End(E') \cong \Ocal$ the quadratic imaginary order (not necessarily maximal) of discriminant $-D$.
 \end{defi}

The following lemma will be used without specific mention throughout, and allows to reduce all problems of cyclicity and invertibility of ideals of a quadratic order to the case of prime power index.

\begin{lem}[Factoring norm-coprime ideals in orders]
\label{lemfactoringnormideals}
    Let $\Ocal$ be an order in a number field and $I$ a (not necessarily invertible) nonzero ideal of $\Ocal$. Assume its norm $N(I) \colonequals |\Ocal/I|$ can be written as a product of coprime integers $N(I) = mn$.

    Then, there is a unique factorization into ideals 
    \begin{equation}
    I = I_m I_n
    \end{equation}
    where $N(I_m )=m, N(I_n)=n$ and $I$ is invertible in $\Ocal$ if and only if both $I_m$ and $I_n$ are.

    Furthermore,  $\Ocal/I$ is cyclic if and only if $\Ocal/I_m$ and $\Ocal/I_n$ are.
\end{lem}

\begin{proof}
    Let us define $I_m := m^{-1}I \cap \Ocal$ and $I_n := n^{-1} I \cap \Ocal$,  ideals of $\Ocal$ containing $I$.

    Define also $G = \Ocal/I$ and $\pi : \Ocal \ra G$ the canonical projection, so that $I_m := \pi^{-1} (G[m])$ and $I_n := \pi^{-1} (G[n])$ by definition.

    Now, $I_m + I_n$ has index dividing both $m$ and $n$ in $\Ocal$ so $I_m + I_n = \Ocal$, and $\Ocal/I_m I_n \cong \Ocal/I_m \oplus \Ocal/I_n$ as an isomorphism of $\Ocal$-modules by the Chinese remainder theorem. By definition again, $G$ is of order $mn$ with $m,n$ coprime (and commutative) so $G \cong G[m] \oplus G[n]$, hence $I_m I_n \subset I$ with index equal to $mn = N(I)$, and $I = I_m I_n$ by index equality. This proves the existence of $I_m$ and $I_n$. For uniqueness, it is enough to say that ideals containing $I$ induce subgroups of $G$, and $G[m]$ is the unique subgroup of $G$ of order $m$ and $G[n]$ is the unique subgroup of order $n$ (since $G$ is a finite commutative group of order $mn$ and $m,n$ are coprime).

Now, if $I$ is invertible with inverse $J$, $I_m$ and $I_n$ are also invertible (with inverses $J I_n$ and $J I_m$ respectively), and conversely if $I_m$ and $I_n$ are invertible, $I$ is a product of invertible ideals hence an invertible ideal.

    Finally, the last sentence comes from the fact that $G = G[m] \oplus G[n]$ with the previous notation, and if $G$ is cyclic, any subgroup of $G$ is cyclic.
\end{proof}

\begin{defi}
    An \emph{admissible ideal} of norm $M$ in a quadratic order $\Ocal$ is an invertible ideal of $\Ocal$ such that $\Ocal/I \cong \Z/M\Z$ as an additive group.
\end{defi}

 \begin{prop}
 For fixed discriminant $D$ and level $N$, 
     there is a bijective correspondence between Heegner points of level $N$ and discriminant $D$ and triples 
     $(\Ocal,\eta,[\mathfrak{a}])$ where $\Ocal$ is the imaginary quadratic order of discriminant $D$, $\eta$ an admissible ideal of norm $N$ of $\Ocal$ and $[\mathfrak{a}]$ is a class of invertible ideals in $\Pic(\Ocal)$.

     In particular, there are exactly $\Ncal_N(D) h_\Ocal$ (distinct) Heegner points of discriminant $D$ on $X_0(N)$, where $\Ncal_N(D)$ is the number of admissible ideals of norm $N$ in $\Ocal$ and $h_\Ocal$ is the order of $\Pic (\Ocal)$.
 \end{prop}
 
 \begin{defi}[Heegner points as triples]
       By abuse of notation, in light of the previous proposition, a triple $(\Ocal,\eta,[\mathfrak{a}])$ will now also denote the corresponding Heegner point in $X_0(N)$.
 \end{defi}

 \begin{proof}
     Consider an isogeny $\phi: E \ra E'$ associated to a Heegner point of discriminant $D$ and level $N$. By CM theory, $\End(E) \cong \Ocal$ implies that $E \cong \C/\mathfrak{a}$ with $\mathfrak{a}$ an invertible ideal in $\Ocal$, and by the same argument $E' \cong \C/\mathfrak{a}'$ with $\mathfrak{a'}$ an invertible ideal in $\Ocal$.

     Lifting the isogeny, we get a multiplication by $\lambda$ map $\C \ra \C$ descending to $\phi$, so in particular $\lambda \mathfrak{a} \subset \mathfrak{a'}$ and the quotient is cyclic of order $N$. Replacing $\mathfrak{a}$ by $\lambda \mathfrak{a}$ (which does not change the ideal class), we can assume $\lambda=1$, and then multiplying $\mathfrak{a} \subset \mathfrak{a'}$ by $\mathfrak{a'}^{-1}$, we obtain 
     \begin{equation}
     \eta := \mathfrak{a} \mathfrak{a'}^{-1} \subset \Ocal
     \end{equation}
     and the quotient must still be isomorphic to $\mathfrak{a'}/\mathfrak{a}$ so cyclic of order $N$ since it is the kernel of~$\phi$.

     The ideal $\eta$ is a product of invertible ideals and so is invertible, and we have associated to $E$ a unique triple $(\Ocal,\eta,[\mathfrak{a}])$ (neither $\eta$ nor $[\mathfrak{a}]$ depend on our choices).   Conversely, taking such a triple $(\Ocal, \eta, [\mathfrak{a}])$, one can define the isogeny 
     \begin{equation}
\C/\mathfrak{a} \ra \C/\mathfrak{a} \eta^{-1}
     \end{equation}
     which is cyclic of degree $N$ and up to isomorphism only depends on $[\mathfrak{a}]$. The two maps are inverses on one another, therefore we do have our bijection and we do get $h_\Ocal \Ncal_N(D)$ different points on $X_0(N)$.
 \end{proof}

Now, the goal is to give  conditions for the existence of  Heegner points.

\begin{prop}
\label{propconditionsHeegnerpoints}
    Let $D = D_K c^2$ be the discriminant of an order $\Ocal \subset K$ an imaginary quadratic field (so that $c$ is the conductor of $\Ocal$).

    Let $N \geq 1$ be a level, $m = \gcd(N,c)$, $N_m = \prod_{p|m} p^{v_p(N)}$ and $N'_m = N/N_m$ so that $N = N_m N'_m$ with the prime factors $p$ of $N$ dividing $N_m$ (resp.\ $N'_m$) when $p | c$ (resp.\ $p \nmid c$).

    Then, there exists a Heegner point of level $N$ and with discriminant $D$ if and only if both conditions hold:
    
    $(a)$ For every prime $p |N'_m$, $p$ is split or ramified in $K$, and split if $p^2|N$.

    $(b)$ There exists an admissible ideal $\eta_{N_m}$ of norm $N_m$ in $\Ocal$.
\end{prop}

\begin{rem}
    Condition $(a)$ can be rephrased equivalently as $(b)$ for an ideal for quotient isomorphic to $\Z/N'_m \Z$ (as will be clear in the proof below), but it is more convenient and closer to the original Heegner points theory to state it as such.
\end{rem}
\begin{proof}
    Let us prove this is a sufficient condition first.  

    By condition $(a)$, we can write for every prime $p\nmid N_m$ that $p \Ocal_K = \gP_+^{(K)} \gP_-^{(K)}$ with two prime ideals in $\Ocal_K$ (which are equal in the ramified case) of norm $p$, and we define $\gP_{\pm} := \gP_{\pm}^{(K)} \cap \Ocal$ which is a prime ideal of $\Ocal$ of index $p$ as $p \nmid c$.

    Since $p$ does not divide $c$, $\gP_{\pm} \cap \Ocal$ is still an (invertible) prime ideal of index $p$ in $\Ocal$. We can then write 
\begin{equation}
\eta = \eta_{N_m} \prod_{\substack{p |N \\ p \nmid m}} \gP_{\pm}^{v_p(N)}
\end{equation}
where for each $p|N$, we choose $\gP_+$ or $\gP_-$. This is a product of invertible prime ideals and they are pairwise relatively prime, so by the Chinese remainder theorem (and properties of nonmaximal orders), 
\begin{equation}
\Ocal/ \eta \cong \Ocal/\eta_{N_m} \oplus \bigoplus_{\substack{p |N \\ p \nmid m}} \Ocal_K/\gP_{\pm}^{v_p(N)}
\end{equation}
which proves that $\eta$ is indeed admissible of norm $N$.

Conversely, let $(\Ocal,\eta,[\mathfrak{a}])$ be a Heegner point. Since $\Ocal/\eta \cong \Z/N\Z$, there is an ideal $\eta'$ of $\Ocal$ containing $\eta$ such that $\Ocal/\eta' \cong \Z/(N/N_m) \Z$. This ideal is relatively prime to the conductor $c$, so we can write it as a product of $\gP_+$ and $\gP_-$ with the above notations, but we cannot have any factor $\gP_+ \gP_-$ because this would give $p \Ocal$, which would break the cyclicity of the quotient. Similarly, looking at the prime factors dividing $m$, there is an ideal $\eta_{N_m}$ of $\Ocal$ containing $\eta_{N_m}$ such that $\Ocal/\eta_{N_m} \cong \Z/N_m \Z$ and such that $\eta = \eta_{N_m} \eta'$ (by argument of inclusion and norm), so that $\eta_{N_m}$ is invertible because $\eta$ and $\eta'$ are.
\end{proof}

We postpone further discussion of condition $(b)$ to \cref{sec:conditionb}, where we will outline an explicit procedure to check this condition.

Finally, we discuss the action of the Atkin--Lehner involutions and Galois on Heegner points, in order to enumerate rational  points on the quotient curve $X_0(N)^*$ that are orbits of Heegner points.
\begin{prop}[Action of the Atkin--Lehner involutions]
For a level $N \geq 1$, Hall divisor $Q$ of $N$ and discriminant $D$ of some order $\Ocal$, the Atkin--Lehner involution acts as follows: 
\begin{equation}
w_Q^{(N)} (\Ocal,\eta,[\mathfrak{a}]) = (\Ocal,\eta'_Q,[\mathfrak{a} \eta_Q^{-1}])
\end{equation}
where we have the unique decomposition $\eta = \eta_Q \eta_{N/Q}$ in admissible ideals of respective norms $Q$ and $N/Q$, and $\eta'_Q = \overline{\eta_Q} \eta_{N/Q}$.
\end{prop}

\begin{proof}
    First, since $\eta$ is invertible (and we are in an imaginary quadratic field), we have $\eta \overline{\eta} = N \Ocal$ by \cite[Proposition 7.4 and Lemma 7.14]{Cox22}.

Now, writing $\eta = \eta_Q \eta_{N/Q}$ as a product of invertible ideals with respective norms $Q$ and $N/Q$, it is only a matter of tracking the functorial definition of $w_Q$ on the Heegner point, and using that $(\eta_{N/Q}^{-1} + Q^{-1} \eta_Q )(\overline{\eta_Q} \eta_{N/Q}) =  \overline{\eta_Q} + \eta_{N/Q} = \Ocal$.
\end{proof}

\begin{prop}[Galois action on Heegner points]
    Let $H_\Ocal$ be the ring class field of $\Ocal$.

    Then, the Heegner points $(\Ocal,\eta,[\mathfrak{a}])$ of level $N$ and discriminant $D$ are defined over $H_\Ocal$, and with the following Galois actions: 

    $(a)$ For $\tau$ the complex conjugation, 
    \begin{equation}
    (\Ocal,\eta,[\mathfrak{a}])^\tau = (\Ocal,\overline{\eta},[\mathfrak{a}^{-1}]).
     \end{equation}

    $(b)$ The group $\Gal(H_\Ocal/K) \cong \Pic(\Ocal)$ acts as follows on Heegner points: for every $[\mathfrak{b}] \in \Pic(\Ocal)$ to which corresponds $\sigma(\mathfrak{b}) \in \Gal(H_\Ocal/K)$ via the Artin map, 
    \begin{equation}
(\Ocal,\eta,[\mathfrak{a}])^{\sigma(\mathfrak{b})} = (\Ocal,\eta,[\mathfrak{a} \mathfrak{b}^{-1}]).
    \end{equation}
\end{prop}

\begin{proof}
   The proof proceeds exactly as in the case $c$ prime to $N$, and those formulas can be found in \cite[paragraph 4]{Gross84}. 
\end{proof}

\begin{cor}
\label{corboundclassnumber}
    If $(\Ocal,\eta,[\mathfrak{a}])$ defines a rational point on $X_0(N)^*$ with $c$ the conductor of $\Ocal$ and $m = \gcd(N,c)$, $\Pic(\Ocal)$ is a 2-torsion group and $h_\Ocal|2^{\omega(N)}$. 
\end{cor}

\begin{proof}
    Recall $\Wcal(N)$ is the group of Atkin--Lehner involutions on $X_0(N)$.

    Assume $P = (\Ocal,\eta,[\mathfrak{a}])$ defines a rational point on $X_0(N)^*$. 

    Every element of $\Gal(H_\Ocal/\Q)$ can be written as $\sigma(\mathfrak{b})$ or $\tau \sigma(\mathfrak{b})$ for some $[\mathfrak{b}] \in \Pic(\Ocal)$, so for every $[\mathfrak{b}] \in \Pic(\Ocal)$, there exist $Q$ and $Q'$ Hall divisors of $N$ such that 
    \begin{equation}
P^{\sigma(\mathfrak{b})} = w_Q^{(N)} (P), \qquad P^{\tau \sigma(\mathfrak{b})} = w_{Q'}^{(N)}(P).
    \end{equation}
   This defines a map from $\Gal(H_\Ocal/\Q)$ to $\Wcal(N)/\Stab_{\Wcal(N)}(P)$, which is readily checked to be a group morphism, injective on $\Gal(H_\Ocal/K)$ by definition of Heegner points formula $(b)$ (it is even injective on the whole Galois group unless $\overline{\eta} = \eta$ by formulas $(a)$ and $(b)$ combined).

    Consequently, $\Pic(\Ocal)$ is a 2-torsion group of order at most $2^{\omega(N)}$ as $\Wcal(N)$ is (and even $2^{\omega(N)-1}$ if we know that $\overline{\eta} \neq \eta$).
\end{proof}

\begin{rem}
    This result also allows us to explicitly describe the Atkin--Lehner involutions that fix a particular Heegner point: indeed, by direct application of the formula, for the Heegner point $P$ associated to a triple $(\Ocal,\eta,[\mathfrak{a}])$, 
    \begin{equation}
\Stab_{\Wcal(N)}(P) = \{ w_Q^{(N)} \, | \,  Q \| N \, \textrm{such that} \,   \overline{\eta_Q} = \eta_Q \textrm{ and } \eta_Q \textrm{  is principal}\}.
    \end{equation}
\end{rem}
\begin{rem}
    Comparing the formulas of Galois and Atkin--Lehner actions more closely, we can even prove under the same hypotheses that $\Pic(\Ocal)$ is generated by the primary factors of $\eta$ and the ramified ideals of $K$ dividing $N$.
\end{rem}

\subsection{Criterion and algorithm for existence of rational Heegner points}
\label{sec:conditionb}
Throughout this section, we fix the following notation.
\begin{itemize}
    \item $D_K<0$ is a fundamental discriminant, and $c \in \Z_{\geq 1}$ a conductor, $D = c^2 D_K$.
    \item $\Ocal$ is the imaginary quadratic order of discriminant $D$, and 
    \begin{equation}
\alpha_\Ocal = \left\{ \begin{array}{rcl}
     \sqrt{D/4} & \textrm{ if } &  4|D_K \\
    c \left(\frac{ 1 + \sqrt{D_K}}{2}\right) & \textrm{ if } & 4 \nmid D_K
\end{array}\right.
    \end{equation}
    so that $\Ocal = \Z[\alpha_\Ocal]$. 
    \item $P_\Ocal$ is the minimal polynomial of $\alpha_\Ocal$ over $\Q$, i.e.\ $X^2 - D/4$ if $4|D_K$ and $X^2 - c X + c^2 \left( \frac{1-D_K}{4}\right)$ otherwise.
\end{itemize}

\begin{prop}
    For every prime $p|c$ and every integer $k \geq 1$, the admissible ideals of norm $p^k$ in $\Ocal$ are in bijective correspondence with the integers $\lambda$ modulo $p^k$ such that $P_\Ocal(\lambda)=0 \mod p^k$ but not $p^{k+1}$, via the bijection 
    \begin{equation}
\lambda \longmapsto I := p^k \Z + (\lambda - \alpha_\Ocal) \Z = (p^k, \lambda - \alpha_\Ocal).
    \end{equation}
\end{prop}

\begin{rem}
    This condition is invariant by congruence modulo $p^k$: indeed, if $\lambda' = \lambda + p^k \ell$ with $\lambda$ a root of $P_\Ocal$ modulo $p^k$, in particular $p|\lambda$ and then $P_\Ocal(\lambda') -P_\Ocal(\lambda) = 2 p^k \lambda + p^{2k} \ell^2$ if $4|D_k$ and $-c p^k \ell + 2p^k \lambda + p^{2k}\ell^2$ if $4 \nmid D_K$, so it zero modulo $p^ {k+1}$, since $p|c$ and $p|\lambda$.
\end{rem}

\begin{proof}
    Ideals of norm $p^k$ in $\Ocal$ must contain $p^k \Ocal$, so we are looking for ideals in the quotient ring 
\begin{equation}
\overline{\Ocal} := \Ocal/p^k \Ocal \cong \Z[X]/(p^k, P_\Ocal) \cong \Z/p^k \Z \oplus \Z/p^k \Z \overline{X}.
\end{equation}
An ideal $\overline{I}$ in $\overline{\Ocal}$ such that $\overline{\Ocal}/\overline{I}\overline{\Ocal}$ is cyclic of order $p^k$ must be a free $\Z/p^k \Z$-module of rank 1, in other words of the shape
\begin{equation}
\overline{I} = \left( \Z/p^k \Z \right) (a + b \overline{X}),
\end{equation}
with $a$ or $b$ invertible in $\Z/p^k \Z$. It is an ideal of $\overline{\Ocal}$ if and only if it is stable by multiplication by $\overline{X}$, in other words for $M_\Ocal$ the matrix of multiplication by $\alpha_\Ocal$ in the basis $(1,\alpha_\Ocal)$, there must be a $\overline{\lambda}$ such that 
\begin{equation}
M_\Ocal \cdot \begin{pmatrix} a \\ b \end{pmatrix} = \overline{\lambda} \begin{pmatrix} a \\ b \end{pmatrix} \Leftrightarrow (M_\Ocal - \overline{\lambda} I_2) \cdot \begin{pmatrix} a \\ b \end{pmatrix} = 0.
\end{equation}
Multiplying by the comatrix we get $P_\Ocal(\lambda) \begin{pmatrix} a \\ b \end{pmatrix} =0$, hence $P_\Ocal(\lambda)=0$ in $\Z/p^k \Z$.  

Looking more precisely at the matrices $M_\Ocal$, we obtain that $a= \lambda b$ if $4|D_K$ and $a = (\lambda-c)b$ so $b$ must be invertible, multiplying by $- b^{-1}$ we can write after lifting the ideal $I = p^k \Z + (\lambda + \alpha_\Ocal)\Z$ (if $4|D_K$) and $I = p^k \Z + ((\lambda-c) + \alpha_\Ocal)\Z$ (if $4 \nmid D_K$) (and conversely we can check that for $\lambda$ a root of $P_\Ocal$ in $\Z/p^k \Z$, this does define an ideal of norm $p^k$ in $\Ocal$).

Finally, replacing $\lambda$ by $-\lambda$ if $4|D_K$ and $c-\lambda$ if $4 \nmid D_K$ (i.e.\ the conjugate root of $P_\Ocal \mod p^k$), which satisfies the same property, we obtain the ideal $I = (p^k, \lambda - \alpha_\Ocal)$ in both cases.

Now, we need to check the invertibility of $I$, i.e.\ we need to ensure that the ring of elements of $K$ stabilizing $I$ is not larger than $\Ocal$. This is equivalent to asking that the matrix $M'_\Ocal$ of multiplication by $\alpha_\Ocal$ in the basis $(p^k,\lambda -  \alpha_\Ocal)$ is not a strict multiple of another integer matrix. Now, writing precisely what this matrix is, we obtain 
\begin{equation}
M'_\Ocal = \left\{ \begin{array}{rcl}
\begin{pmatrix} 
\lambda & P_\Ocal(\lambda)/p^k \\ 
- p^k & -\lambda 
\end{pmatrix} & \textrm{ if  } & 4|D_K \\
\begin{pmatrix} 
\lambda & P_\Ocal(\lambda)/p^k \\ 
- p^k & c-\lambda 
\end{pmatrix}  & \textrm{ if } & 4 \nmid D_K.
\end{array}\right.
\end{equation}
In both cases, the gcd of all coefficients of the matrix is 1 if $v_p(P_\Ocal(\lambda))=k$ and a multiple of $p$ otherwise (recall $p$ divides $c$ and $\lambda$), so the ideal is invertible if and only if $v_p(P_\Ocal(\lambda))=k$.
\end{proof}

\begin{rem}
   As explained by the theory of nonmaximal orders, one finds that there are never admissible ideals of norm $p$ if $p|c$. 
\end{rem}

\subsection{Results}

With these results, and using bounds on the sizes of the class numbers of orders \cite{Klaise} and
Corollary \ref{corboundclassnumber}, one can determine algorithmically for every small level $N$ the exact list of all possible Heegner points defining rational points on $X_0(N)^*(\Q)$. We have written up and run this algorithm (which is probably of independent interest), with the following results for our list of minimal exceptional levels. 

Combining this with a sharp upper bound on the number of $X_0(N)^*$ coming from Chabauty--Coleman, we are able to completely classify the rational points on almost all small exceptional levels and show that they are trivial.

To carry out the algorithm in practice, we first need to list all orders $\Ocal$ with $h(\Ocal)$ bounded by $2^{\omega(N)}$. A table can be found in \cite[p.~19]{Klaise}.
Instead of working explicitly with the ideals and reduction of ideals to ideal classes in the ideal class group of of $\Ocal$, which seems difficult, we work with the form class group of reduced primitive positive definite binary quadratic forms of  discriminant $D$. By \cite[7, Theorem 7.7\,(i)]{Cox22} these are isomorphic.

For levels $N \leq 240$ and $N = 396$, excluding the cases $N \neq 99, 125, 169$ (where the Mordell--Weil rank is equal to genus), we constructed models of the modular curves $X_0(N)^*$ and explicitly classified their rational points.
\begin{rem}
    We warn the reader that in Magma V2.28-18, the intrinsic \texttt{X0NQuotient} does not always produce the correct equations for the specified quotient of $X_0(N)$. Building on code from  \cite{AKMJNOV} we wrote code to compute the correct models and their associated data. This bug has now been fixed in Magma V2.28-20.
\end{rem}
The Jacobians $\Jac(X_0(N)^*)$ have Mordell--Weil rank less than the genus of $X_0(N)^*$. 
When the genus is $1$, the rank is $0$, and we simply computed the torsion subgroup to find the rational points. Otherwise, we used the Chabauty--Coleman method, in combination with a Mordell--Weil sieve if necessary.
\Cref{table:heegnerpts} summarizes the results of our computations.

\begin{center}
\begin{table}
\begin{tabular}{cccccl}
     level& genus & $\Q$-pts & $\Q$-cusps & $\Q$-Heegner points & CM discriminants  \\
     \midrule
40	&1	&6	&2& 4& $-15,-16, -60, -160$ \\
48	&1&	8&	3&  5&  $-15, -48, -48, -60, -192$\\
72	&1&	6&	4&	2&	$-32, -288$\\
80  &1&	6&	3&	3&	$-15, -60, -64$\\
88	&2	&6&	2&	4&	$-7, -28, -32, -352$\\
96	&1&	4&	2&	2&	$-15, -60$\\
100	&1&	5&	2&	2&	$-16, -64$\\
108 &1	&3	&2&	1&	$-32$\\
112	&2&	6&	3&	3&	$-7, -28, -448$\\
120	&1&	6&	2&	4&	$-15, -60, -96, -480$\\
135	&2&	3&	1&	2&	$-11, -35$\\
144	&3&	6&	6&	0& \\
147	&2&	12&	1&	9&	$-3, -12, -27, -24, -48, -75, -147, -147, -147$\\
162 &3&	5&	1&	4&	$-8, -20, -72, -72$\\
176	&4&	5&	3&	2&	$-7, -28$\\
180 &2&	4&	4&	0& \\
196	&3	&4&	2&	2&	$-12, -48$\\
200	&3&	3&	2&	1&	$-16$\\
216	&5&	3&	2&	1&	$-32$\\
224 &4&	6&	2&	4&	$-7, -28, -112, -112$\\
225	&4&	4&	2&	2&	$-11, -99$\\
240	&3&	6&	3&	3&	$-15, -60, -960$\\
396 & 5 & 6 & 4 & 2& $-32,-288$
\end{tabular}
\caption{Rational points on small exceptional levels}
\label{table:heegnerpts}
\end{table}
\end{center}
There are two levels $N$  where the number of $\Q$-rational cusps plus the number of Heegner points on $X_0(N)^*$ does not add up to the number of $\Q$-points, $N = 100$ and  $147$. We will investigate and explain these remaining points in the following sections.

\subsection{Rational non-Heegner points}
\label{sec:CMpoints}
The following pair of propositions allows us to deduce the existence of certain rational CM points on $X_0(N)^*$ that arise from lifting rational Heegner points on $X_0(M)^+$.
We will apply some of the corollaries of these propositions to the remaining small exceptional levels $N = 100$ and $N = 147$ to explain two of the missing CM points.
We then show that $X_0(147)^*(\Q)$ contains two exceptional (non-CM, non-cuspidal) points.

Before we prove this, we recall a fact about isogeny volcanoes from \cite{SutherlandIsogenyVolcanoes} and \cite{KohelThesis}. 
Let $\ell$ be a prime number.
Recall that an $\ell$-isogeny of CM elliptic curves $\phi: E_{D}\to E_{D'}$ is \emph{horizontal} if $D = D'$, \emph{ascending} if $D'|D$ and $D' \neq D$,  and \emph{descending} if   $D|D'$ and $D' \neq D$.

\begin{prop}[{\cite[\S 2.7, \S 2.9, Lemma 6, Remark 8]{SutherlandIsogenyVolcanoes} and \cite[Proposition 23]{KohelThesis}}]\label[proposition]{prop:isogvocanoes}
Let $E/\C$ be an elliptic curve with CM by an order $\Ocal$ of discriminant $D$ in an imaginary quadratic field $K$.
Let $\ell$ be a prime number. Then:
\begin{enumerate}
    \item If $\ell$ divides $[\Ocal_K : \Ocal]$, then $E$ has a unique ascending $\ell$-isogeny, no horizontal $\ell$-isogenies, and $\ell$ descending ones.
    \item Otherwise  $E$ has $1+  \left( \frac{D}{\ell} \right)$ horizontal $\ell$-isogenies,  $1-  \left( \frac{D}{\ell} \right)$ descending $\ell$-isogenies,  and no ascending $\ell$-isogenies.
    \item If $\varphi:E \to E'$ is an $\ell$-isogeny towards some elliptic curve $E'$, if $f,f'$ are the respective conductors of endomorphism rings of $E$ and $E'$, then $f=f'$ or $f = \ell^{\pm 1} f'$.
\end{enumerate}
\end{prop}

\begin{prop}[Lifting rational Heegner to rational CM non-Heegner points]\label{prop:CMpointslift}
 Let $d$ and $M$ be positive coprime integers.
 
 Let $ (\Ocal,\eta,[\mathfrak{a}])$ be a Heegner triple  of level $M$ (with $\Ocal$ a nonmaximal order) such that the associated point $P=(E,C_M) \in X_0(M)(\Qb)$ defines a rational point $P^* \in X_0(M)^*(\Q)$.

 Assume that there is an elliptic curve $E_0$ defined over $\Q$ with CM order $\Ocal_0$ an order containing strictly $\Ocal$, such that there is a unique (up to automorphism) cyclic $d$-isogeny $\varphi: E \rightarrow E_0$, with kernel denoted by $C_d$.

 Then, $(E,C_M + C_d) \in X_0(dM)(\Qb)$ defines a rational point in $X_0(dM)/\langle w_{Q}^{(dM)}, Q \| M \rangle$, hence a rational (non-Heegner) point in $X_0(dM)^*(\Q)$.
\end{prop}

\begin{proof}
 First, for every $\sigma \in \GalQ$, there is a unique cyclic $d$-isogeny between $E^\sigma$ and $E_0$, and its kernel is $C_d^\sigma$. Indeed, $\varphi^\sigma$ provides such an isogeny as $E_0$ is defined over $\Q$, and conversely if there is another $d$-cyclic subgroup $C'_d$ of $E^\sigma$ such that $E^\sigma/C'_d \cong E_0$, then by applying $\sigma^{-1}$, one obtains another $d$-cyclic isogeny $E \rightarrow E_0$ for the same reason, hence $(C'_d)^{\sigma^{-1}} = C_d$ by uniqueness, so $C'_d = C_d$.

 Now, by hypothesis, for every $\sigma \in \GalQ$, there exists $Q \| M$ such that $P^\sigma = w_Q^{(M)} (P)$, i.e.\ 
     \begin{equation}
(E^\sigma, C_M^\sigma) = w_Q^{(M)} (E,C_M) = (E/C_M[Q],(E[Q] + C_M)/C_M[Q]).
     \end{equation}
     Denote by $\Qcal$ the set of $Q$'s obtained as such.
     
 Second, denote for each $Q\in \Qcal$ by $\psi_Q$ the quotient isogeny $E \rightarrow E/C_M[Q]$. Then, $(E/C_M[Q])/\psi_Q(C_d)$ is isomorphic to $E_0$. Indeed, by hypothesis $E_0$ is the unique elliptic curve with CM by $\Ocal_0$. 
 Let us define $f_0,f,f'$ the respective conductors of $\Ocal_0, \Ocal$ and of the geometric endomorphism ring of $E_0 \cong (E/C_M[Q])/\psi_Q(C_d)$. By~\cref{prop:isogvocanoes}\,$(3)$, $f'/f$ must be supported on the prime factors of $d$ (as is $f/f_0$ because of the isogeny $\varphi$), hence $f'/f_0$ is supported on the prime factors of $d$, but on another hand $f'/f_0$ must be supported on the prime factors of $Q$ because of the isogeny $E_0 \rightarrow E_0/\varphi(C_M[Q])$.
 
 We thus have $f' = f_0$, which proves the elliptic curve is necessarily $E_0$. 
 
Then, $C_d^\sigma$ is the unique kernel of cyclic $d$-isogeny between $E^\sigma \cong E/C_M[Q]$ and $E_0$, but we have shown that $(E/C_M[Q])/\psi_Q(C_d) \cong E_0$, therefore  $\psi_Q(C_d) = C_d^\sigma$ by uniqueness. 
To conclude, consider 
\begin{equation}
P_{dM} = (E,C_d+C_M) \in X_0(dM)(\Qb).
\end{equation}
For every $\sigma \in \GalQ$, choosing $Q \in \Qcal$ such that $P^\sigma = w_Q^{(M)}(P)$, 
\begin{eqnarray}
P_{dM}^\sigma & = & (E^\sigma,C_d^\sigma + C_M^\sigma) \notag \\
& =&  (E/C_M[Q],C_d^\sigma + \psi_Q(E[Q] + C_M)) \\
& = & (E/C_M[Q], \psi_Q(C_d) + \psi_Q(E[Q] + C_M)) \notag
\end{eqnarray}

Combining the equalities, we then obtain an equality of the type
\begin{equation}
P_{dM}^\sigma = (E/C_M[Q], \psi_Q(C_d + E[Q] + C_M)) = w_Q^{(dM)}(P_{dM}),
\end{equation}
for every $\sigma \in \GalQ$, therefore $P_{dM}$ defines a rational point in $X_0(dM)/\langle w_Q^{(dM)}, Q \| M \rangle$ as claimed.
\end{proof}

\begin{prop}\label{prop:CMpointskernels}
Let $E_{-D}$ denote (a choice of) elliptic curves with CM discriminant $-D$. We have unique (cyclic) kernels of $d$-isogenies $E_{-D} \to E_{-D'}$ in the following cases:
\begin{enumerate}[(a)]
    \item $E_{-16} \to E_{-4}$, $d = 4$ 
    \item $E_{-27} \to E_{-3}$, $d = 3$
    \item $E_{-27} \to E_{-3}$, $d = 9$ 
    \item $E_{-12} \to E_{-3}$, $d = 2$
    \item $E_{-48} \to E_{-3}$, $d = 4$ 
    \item $E_{-48} \to E_{-12}$, $d = 2$
    \item $E_{-28} \to E_{-7}$, $d = 2$
    \item $E_{-112} \to E_{-7}$, $d = 4$
    \item $E_{-36} \to E_{-4}$, $d = 3$
\end{enumerate}
\end{prop} 

\begin{proof}[Proof of Proposition \ref{prop:CMpointskernels}]
The proofs are straightforward applications of Proposition \ref{prop:isogvocanoes}. For $(b), (d), (f), (g),$ and $(i)$, these are just the unique ascending cyclic isogeny of degree $d$ between the listed elliptic curves.

For $(e),$ and $(h)$, the unique degree 4 cyclic isogeny arises as the composition of two unique degree 2 ascending isogenies.

The cases $(a)$ and $(c)$ are more interesting.
In $(a)$, we start with the unique ascending $2$-isogeny $E_{-16} \to E_{-4}$ (it is crucial here that there is no horizontal 2-isogeny $E_{-16} \to E_{-16}$). Combining this with the unique horizontal $2$-isogeny $E_{-4} \to E_{-4}$, we obtain a unique cyclic $4$-isogeny $E_{-16} \to E_{-4}$. For $(c)$, we start with the a unique ascending isogeny $E_{-27} \to E_{-3}$ of degree $3$ (there is no horizontal 3-isogeny $E_{-27} \rightarrow E_{-27}$). Combining this with the unique horizontal isogeny $E_{-3} \to E_{-3}$ of degree 3, we obtain a unique cyclic $9$-isogeny  $E_{-27} \to E_{-3}$.
\end{proof}

We now apply this to explain the extra points in \Cref{table:heegnerpts}.
At $N = 100$, we are missing one point. The point is a CM point not arising from Heegner points on $X_0(100)^*$. By Propositions \ref{prop:CMpointslift} and \ref{prop:CMpointskernels}, after verifying that there is a Heegner point of discriminant $-16$ on $X_0(25)^*$ we have proven that this is a point given by a cyclic $4$-isogeny from the elliptic curve with CM discriminant $-16$ to the elliptic curve with CM discriminant $-4$.

There are three missing rational points on $X_0(147)^*$.
First, there is a CM point which, by Propositions \ref{prop:CMpointslift} and \ref{prop:CMpointskernels} and the fact that $X_0(49)^*$ has a Heegner point of discriminant $-27$,  is  given by a $3$-isogeny from the elliptic curve with CM discriminant $-27$ to the elliptic curve with CM discriminant $-3$.

There are two exceptional points on $X_0(147)^*$ (rational points which are non-CM and not cusps).
These are given by $j$-invariants satisfying the following minimal polynomials
\label{subsec:147}
{\scriptsize
\begin{align}
    x^4 &- 436294632387445525302369131483180799074870132944408913417997146895640211367 \nonumber \\
    &+ 3030000749471339072/(3^{147}\cdot 83^3) \cdot x^3 \nonumber \\
    &+ 1853081114889740152670552404166064400785666244456581901795788811689973173885 \nonumber \\
    &+ 601266913798688224019931114783980466429363219239891148132478178080/(3^{196} \cdot 83^6) \cdot x^2 \nonumber \\
    &- 818921749683932551797439647036751279693219934658565270534946851782413187330 \normalsize{\tag{$\bigstar$}} \label{eq:exceptional1471}  \\
    &+ 827060235054694333188892862612248136656343736742062058004841946438930591846 \nonumber \\
    &+ 4/(3^{199} \cdot 83^7) \cdot x  \nonumber \\
    &+ 124247059244019094779328266651744106659092539602338608324551530837920070424 \nonumber \\
    &+ 137428796981374268229284981398361952272218833829858351835228139179075171149 \nonumber \\
    &+ 7675809024/(3^{200} \cdot 83^8) \nonumber
\end{align}
}

and
{\scriptsize
\begin{align*}
x^4 & -  
2185834256476398919667661699680057386538756568156375/2^{147}\cdot x^3 \nonumber\\ & +   4068652760832854210651704588306466871045151058510583281739967466905718625 / 2^{196} \cdot x^2  \normalsize{\tag{$\bigstar \bigstar$}} \label{eq:exceptional1472}\\ & +  4406473070233200489910589891873135102505191238796926491650077401072362078125  /2^{199} \cdot x \nonumber\\ & + 
2235837537492274292081138826498765155100228727086319246963199330645437211390625/2^{200}. \nonumber
\end{align*} 
}

The points on \eqref{eq:exceptional1471} are a lift via $\psi_{147,3}$ of a Galois orbit of points on $X_0(3)^*$, whose $j$-invariants are the roots of the polynomial
{ \scriptsize
\begin{align}
&x^2 + 8843539984052578516745954230680528468514848296750944/(3^{21}\cdot 83^{21}) \cdot x \\
&-1260732867722543499520886291779027044098193555921947309839666025326704368/(3^{28}\cdot 83^{28}).\notag
\end{align}}
The points on \eqref{eq:exceptional1472} are a lift via $\psi_{147,3}$ of a Galois orbit of points on $X_0(3)^*$, whose $j$-invariants are the roots of the polynomial
{\scriptsize
\begin{align}
    x^2 - 75403995064430432067361575/2^{21} \cdot x + 31863271848515801577905234193374625/2^{28}.
\end{align}
}

For each of the two polynomials \eqref{eq:exceptional1471} and \eqref{eq:exceptional1472}, we checked with Magma that $P$ defines a biquadratic extension of $\Q$.
Let $\Phi_d$ be the modular polynomial for level $d$.
We also checked that for the four roots $j_1,j_2,j_3,j_4$ of $P$ (corresponding to non-CM elliptic curves because $P \notin \Z[X]$), for each pair of indices $m,n \in \{1,2,3,4\}$ we have $\Phi_d(j_m,j_n)=0$ for some $d \in \{1,3,49,147\}$.
Since all the corresponding elliptic curves are non-CM, this means that, up to reindexing the roots, we have a commuting diagram (up to sign) of isogenies as follows:
\begin{equation}
\begin{tikzcd}
	{E_{j_1}} & {E_{j_2}} \\
	{E_{j_3}} & {E_{j_4}}
	\arrow["3", dash, from=1-1, to=1-2]
	\arrow["49", dash, from=1-2, to=2-2]
	\arrow["49", dash, from=2-1, to=1-1]
	\arrow["3", dash, from=2-2, to=2-1]
 \arrow["147", dash, from=1-1, to=2-2]
\end{tikzcd}
\end{equation}

We can now define for each $m \in \{1,2,3,4\}$ the point $P_m = (E_{j_m}, \ker (E_{j_m} \rightarrow E_{j_n})) \in X_0(147)(K)$ where  $j_n$ is the invariant opposite to $j_m$ in the square. For each $\sigma \in \Gal(K/\Q)$, by uniqueness of the isogenies, $P_m^\sigma = P_n$ for some $n$, and for each $d \nmid 147$, $w_d (P_m) = P_{n'}$ for some $n'$. The Galois orbit and Atkin--Lehner orbit of $P_1$ are then identical, which implies that $P_1$ defines a rational point on $X_0(147)^*$.

\begin{remark}
The exceptional points are hyperelliptic involutions of CM points and this implies that the hyperelliptic involution on $X_0(147)^*$ is not modular, which did not seem to be known before.
\end{remark}

\subsection{Dealing with levels square-above exceptional points and conclusion}

In all our family of (minimal) exceptional levels for each we either worked out Mazur's method, or determined fully the rational points, the levels allowing for a nontrivial point with no result of potentially good reduction of the $j$-invariant are $125 = 5^3$ and $147 = 3 \cdot 7^2$.

\begin{rem}
    By Corollary \ref{cor:trivialityandabovelevels}, if $\widetilde{N}$ is square-above $N$, if $X_0(N)^*(\Q)$ is trivial then $X_0(\widetilde{N})^*(\Q)$ is too, but the converse is not true (in fact we generally expect any level strictly square-above even an exceptional level to have only trivial rational points).
\end{rem}

\begin{prop}
\label{prop:exccasesratpoints}
    Let $\widetilde{N}$ be an exceptional level such that $g(X_0(\widetilde{N})^*)>0$ and assume $P=(E,C_{\widetilde{N}})$ defines a nontrivial rational point on $X_0(\widetilde{N})^*$. Then, we are in one of the following possible cases: 

    \begin{itemize}
        \item $\widetilde{N} = 99$: $X_0(99)^*$ is a rank one elliptic curve, so has infinitely many rational points.
        \item $\widetilde{N} = 125$ and $E$ has $j$-invariant 
        {\tiny 
        \[
\frac{-2^{18}\cdot 3^3 \cdot  23447 \cdot 483165695017 \cdot 26700980784488201 \pm 2^{20} \cdot 3^7 \cdot 5 \cdot 7^3 \cdot 13 \cdot 17 \cdot 19 \cdot 29 \cdot 31 \cdot 59 \cdot 101 \cdot 113 \cdot 179 \cdot 463 \cdot 563 \cdot 1553 \sqrt{509}}{11^5} .
        \]
        }
        \item $\widetilde{N} =147$: $j(E)$ is a root of one of the two degree 4 polynomials given in  \cref{subsec:147}. 

    \item $\widetilde{N}$ is only square-above minimal exceptional levels $N \in \Lcal$ with
    $N \geq 250$ 
    and for every such $N$ we have $N \notin \{396, 500, 1125\}$ and $(m'_{N,M})^{\widetilde{N}} j(E)$ is integral with the notations of  Theorem \ref{prop:integralityexclevel}.
    
    \end{itemize}
\end{prop}

\begin{proof}
    The first three cases come from the previous results (see \cite[Proposition 2.1]{ArulMuller22} for $\widetilde{N}=125$).

For the fourth case, $\widetilde{N} \neq 99, 125, 147$ and because $\Lcal$ is a minimal family, there is some $N \in \Lcal$ such that $\widetilde{N}$ is square-above $N$, and $g(X_0(N)^*)>0$. 
By hypothesis, $X_0(\widetilde{N})^*(\Q)$ has a nontrivial point $P = (E, C_{\widetilde{N}})$, so that $X_0(N)^*(\Q)$ does as well by Corollary \ref{cor:trivialityandabovelevels}. 

By \Cref{table:heegnerpts} and Table \ref{tab:smallcurves}, the cases  $N < 250$ excluding $N = 147$, as well as the case $N = 396$ all have $X_0(N)^*(\Q)$ trivial. So we are left with $N = 147$ or $N \geq 250$, $N \ne 396$.
    
Suppose $N \geq 250$, and $N \neq 396$. 
Write $\widetilde{N} = N d^2$ and let $P_N = (E_N, C_N)$ be  the image of $P$ under the composition of the maps $i^{(q)}_{N', N'/q^2}$ of Proposition \ref{prop:constructionofdegeneracymappsi} from level $\widetilde{N}$ to level $N$, so that $E$ and $E_N$ are related by a cyclic isogeny of degree $d$. It lifts a nontrivial point of $X_0(N)^*(\Q)$  and $E \to E_N$ is a cyclic isogeny of degree $d$.
If $N \in \{500,1125\}$, then $P_N$ is in turn isogenous to the exceptional point of $X_0(125)^*$, whose $j$-invariant is not integral at $11$, and so $j(E_N)$ is not integral at some prime above 11 by Lemma \ref{lem:integralityisogenousellcurves}; this contradicts Theorem \ref{prop:integralityexclevel} which gives $E_N$ potentially good reduction at all primes above $11$. Hence $N \notin \{500, 1125\}$.
For the remaining $N$, by Theorem \ref{prop:integralityexclevel} and Lemma \ref{lem:integralityisogenousellcurves} we have that $(m'_{N,M})^{N d} j(E) \in \mathcal{O}_K$ and $Nd \leq \widetilde{N}$. 

Finally, we consider the case where $N = 147$.
If $N=147 = 3\cdot7^2$, checking the full description of exceptional levels in \cref{subsec:detexclevels}, there are only three minimally exceptional levels square-above $147$: $\widetilde{N} = 2^2\cdot3\cdot7^2$, $3^3\cdot7^2$ or $3 \cdot 7^4$. Let us thus take $q \in \{2,3,7\}$, so that $\widetilde{N}= 147 q^2$ in all cases.

Assuming that there is a nontrivial $P^* \in X_0(\widetilde{N})^*(\Q)$ for one of those levels, we will reach a contradiction. 
Let $P = (E,C) \in X_0(147q^2)$ a lift of $P^*$, with $K$ its polyquadratic field.
 Using the degeneracy map, $\psi_{\widetilde{N},147}(P^*)$ is one of the two nontrivial points on $X_0(147)^*(\Q)$
 so up to a good choice of lifts, one can assume $P =(E,C_{\widetilde{N}})$ above $P^*$ is such that $i_{\widetilde{N},147}^{(q)}(P) = P'  $ where $P' = (E/C[q], C[147q]/C[q]) $ has $j$-invariant $j_0$ one of the eight roots of the two quartics ~\eqref{eq:exceptional1471}, ~\eqref{eq:exceptional1472}. 
 
Let $K'$ be the field of definition of $j_0$.
Proposition \ref{prop:fieldsofdefsquarebelow} implies $K' \subseteq K$  with $[K: K'] \leq 2$ and moreover $K = K'$ when $q^3$ divides $\widetilde{N}$, i.e. for $q = 3,7$.

Since $E$ is $q$-isogenous to $E/C[q]$ the number $j(E) \in K$ is a root of the modular polynomial $\Phi_q(j_0,X)$. We compute:
\begin{itemize}
\item $\Phi_2(j_0,X)$ is an irreducible cubic over $K'$;
\item $\Phi_3(j_0,X)$ factors over $K'$ as a linear polynomial times an irreducible quartic, so it has exactly one root in $K'$;
\item $\Phi_7(j_0,X)$ factors over $K'$ as a linear polynomial times an irreducible degree $7$ polynomial, so it has exactly one root in $K'$.
\end{itemize}

For $q = 2$, this implies that $\Phi_2(j_0,X)$  has no root in any extension of $K'$ of degree at most $2$, contradicting the fact that $j(E)$ would be such a root.

For $q = 3$ or $7$, we consider the two curves $E$ and $E/C[q^2]$. These are $q$-isogenous to $E/C[q]$, through the dual of the map $E \to E/C[q]$ and the map $E/C[q] \to E/C[q^2]$ respectively. 
Both curves are defined over $K = K'$ so these yield two roots $j(E)$ and $j(E/C[q^2])$ of $\Phi_{q}(j_0,X)$ lying in $K'$.
However, there is only one such root corresponding to the linear factor, which implies $j(E) = j(E/C[q^2])$ and the two curves are isomorphic $E \simeq E/C[q^2]$ over $\overline{\Q}$. 
Thus, composing the map $E \to E/C[q^2]$ with the isomorphism yields an endomorphism of $E$ of degree $q^2$ whose kernel is the cyclic group $C[q^2]$. However, $E$ does not have CM, so this gives a contradiction.
\end{proof}

\section{Classifying rational points on small levels} \label{sec:smalllevels}

In this section we discuss the classification of rational points on small levels.

Table \ref{tab:smallcurves} gives a classification of rational points on all star curves of non-squarefree non-exceptional level $N$, where $X_0(N)^*$ has genus between $1$ and $5$, and the level is minimal (in the sense that there is no $N'|_{\square} N$ with finitely many rational points). 

The $\Q$-points on almost all curves are computed by first obtaining an upper bound with the  Chabauty--Coleman method, and then exhibiting enough CM points and cusps.
The Chabauty--Coleman method gives a sharp upper bound on the count of the number of $\Q$-points, except in the cases of $N = $ 261, 294, 308, 376, 378, and 572.
For levels $294$ and $308$, the rank is one less than the genus and greater than $0$, and the upper bound from Chabauty--Coleman is not sharp.
In this case, one needs to exclude some points, for example, with a Mordell--Weil sieve \cite{BruinStollMW}. Such a sieve is only feasible if one can find generators and execute Jacobian arithmetic, so this is currently only implemented in genus 2.
For levels $N = 261, 376, 378$ and $572$, the number of known rational points minus 1 is less than the rank.
In this case, we fail to construct a vanishing differential: the current implementation of Chabauty--Coleman requires that the Mordell--Weil group is generated up to finite index by the known rational points in order to construct a differential form whose Coleman integral vanishes on the rational points of the curve embedded in the Jacobian.
In these cases, we do not have sufficiently many rational points. For these 6 levels, we instead use the rank zero quotients of the Jacobian to determine the rational points, which we discuss in the next section.

The classification of the CM discriminants here mainly arises from results about Heegner points as well as the propositions in \cref{sec:CMpoints}. In some small cases, we were also able to compute $j$-maps from $X_0(N) \to X(1)$ along with quotient maps $X_0(N) \to X_0(N)^*$, then applied this to find the possible $j$-invariants of the lifts of rational points on $X_0(N)^*$. 
Computing the full quotient map from $X_0(N)$ became too computationally expensive after some very small examples. Write $D_1 \to D_2$ in the table for a CM point arising from an isogeny of elliptic curves with CM discriminants $D_1, D_2$.

In order to apply Proposition \ref{prop:CMpointslift}, we check that $-16$ is the discriminant of a Heegner point on the (star quotient) levels $N = 52/4, 68/4, 116/4, 164/4, 148/4, 260/4, 212/4$, that $-27$ is the discriminant of a Heegner point on the levels $N = 63/3, 171/9, 279/9$, that $-48$ is the discriminant of a Heeegner point on the levels $N = 84/4, 156/4, 228/4, 294/2$, that $-12$ is the discriminant of a Heegner point on the levels $N = 98/2, 294/2$, that $-28$ is the discriminant of a Heegner point on the level $N = 242/2$, and that $-112$ is the discriminant of a Heegner point on the level $N = 308/4$ using the algorithms described in \cref{sec:heegnerpts}.

The two rational exceptional points on $X_0(63)^*$ were discovered by Elkies \cite{ElkiesKCurves}. The rational exceptional point on $X_0(125)^*$ was a computation of Arul and M\"uller \cite{ArulMuller22}. As far as we know, the rational exceptional point on $X_0(75)^*$ has not appeared in the literature before. The minimal polynomial of the $j$-invariant of this point is
\begin{align}\label{eq:x075}
    x^4 &-& 3^6 \cdot 5^2 \cdot 15671 \cdot 62062657 \cdot 2594434657/2^{75}x^3 \notag \\ 
    &+& 3^9 \cdot 5^2 \cdot 89 \cdot 337 \cdot 941 \cdot 93497 \cdot 4409922434705659247/2^{100}x^2  \tag{$\bigstar \bigstar \bigstar$} \\&-& 3^{15}\cdot 5^2 \cdot 11^4 \cdot 152459\cdot 187786109447 \cdot 553303027813/2^{103}x \notag\\&+& 3^{18} \cdot 5^2 \cdot 11^6 \cdot 2276505289^3/2^{104}. \notag
\end{align}

\begin{center}
    \centering
    \begin{longtable}{lllllll}
    \caption{Rational points on small levels of genus $\leq 5$} \\\label{tab:smallcurves}
        N & genus & $\Q$-pts & $\Q$-cusp & $\Q$-Heeg. & Except. & CM Disc  \\ 
        \hline
        48 & 1 & 8 & 3 & 5 &  & $-15,-48,-48,-60,-192$ \\ 
        52 & 1 & 7 & 2 & 4 & &  $-12, -16, -48, -64,$ $(-16 \to -4)$  \\ 
        63 & 1 & 8 & 2 & 3 & 2& $-27,-35,-315$, $(-27 \to -3)$  \\ 
        64 & 1 & 4 & 2 & 2 &  & $-7,-28$ \\ 
        68 & 1 & 6 & 2 & 3 & & $-16,-32,-64$, $(-16 \to -4)$\\ 
        75 & 1 & 8 & 1 & 6 & 1&  $-11,-24,-51,-75,-75,-75$ \\ 
        76 & 1 & 5 & 2 & 3 &  & $-12, -32, -48$ \\ 
        81 & 1 & 3 & 1 & 2 & & $-8,-11$ \\ 
        84 & 1 & 8 & 2 & 5 & & $-12, -48, -96, -192, -672$, $(-48 \to -3
)$  \\ 
        90 & 1 & 6 & 2 & 4 & & $-20,-36,-36,-180$ \\ 
        98 & 1 & 6 & 1 & 4 & & $-20,-24,-40,-52$, $(-12 \to-3 )$ \\ 
        124 & 1 & 4 & 2 & 2 & & $-12,-48$ \\ 
        126 & 1 & 4 & 2 & 2 & & $-20,-180$ \\ 
        132 & 1 & 4 & 2 & 2 & & $-32, -96$ \\ 
        140 & 1 & 4 & 2 & 2 & & $-160, -1120$ \\ 
        150 & 1 & 4 & 1 & 2 & & $-24,-84, (-36 \to -4)$ \\ 
        156 & 1 & 7 & 2 & 4 & & $-12, -48, -192, -1248$, $(-48 \to -3
)$  \\ 
        188 & 1 & 2 & 2 & 0 & & ~ \\ 
        220 & 1 & 3 & 2 & 1 & & $-160$  \\ 
        104 & 2 & 3 & 2 & 1 & & $-16$ \\ 
        116 & 2 & 8 & 2 & 5 & & $-7,-16,-64,-112,-928$, $(-16 \to -4)$ \\ 
        117 & 2 & 4 & 2 & 1 & & $-27$, $(-3 \to -27)$\\ 
        121 & 2 & 6 & 1 & 5 & & $-7,-8,-19,-28,-43$ \\ 
        125 & 2 & 6 & 1 & 4 & 1&$-4,-11,-16,-19$ \\ 
        153 & 2 & 4 & 2 & 2 &  &$-8,-72$ \\ 
        168 & 2 & 4 & 2 & 2 & &$-96,-672$ \\ 
        198 & 2 & 4 & 2 & 2 & &$-8,-72$ \\ 
        204 & 2 & 6 & 2 & 4 & &$-15, -32, -240, -1632$ \\ 
        276 & 2 & 4 & 2 & 2 & &$-15, -240$ \\ 
        284 & 2 & 4 & 2 & 2 & &$-7, -112$ \\ 
        380 & 2 & 6 & 2 & 4 & &$-15, -160, -240, -3040$ \\ 
        128 & 3 & 4 & 2 & 2 & &$-7,-28$ \\ 
        136 & 3 & 4 & 2 & 2 & &$-16, -32$ \\ 
        152 & 3 & 3 & 2 & 1 & &$-32$ \\ 
        164 & 3 & 6 & 2 & 3 & &$-16,-32,-64$, $(-16 \to -4)$ \\ 
        171 & 3 & 6 & 2 & 3 & &$-8,-27,-72$, $(-27 \to -3)$ \\ 
        175 & 3 & 4 & 1 & 3 & &$-19,-75,-91$ \\ 
        189 & 3 & 4 & 1 & 3 & &$-27,-27,-35$ \\ 
        207 & 3 & 4 & 2 & 2 & & $-11,-99$ \\ 
        234 & 3 & 4 & 2 & 2 & & $-36,-36$ \\ 
        236 & 3 & 3 & 2 & 1 & & $-32$ \\ 
        245 & 3 & 4 & 1 & 3 & & $-19,-40,-115$ \\ 
        248 & 3 & 2 & 2 & 0 & & ~ \\ 
        252 & 3 & 4 & 4 & 0 & & ~ \\ 
        270 & 3 & 4 & 1 & 3 & & $-20,-180,-180$ \\ 
        294 & 3 & 6 & 1 & 3 & & $-20,-24,-132, (-48 \to -12), (-12 \to -3)$\\ 
        312 & 3 & 3 & 2 & 1 &  &$-1248$ \\ 
        315 & 3 & 4 & 2 & 2 &  &$-35,-315$ \\ 
        348 & 3 & 3 & 2 & 1 &  &$-96$ \\ 
        420 & 3 & 4 & 2 & 2 & & $-96, -3360$ \\ 
        476 & 3 & 2 & 2 & 0 & & ~ \\ 
        148 & 4 & 9 & 2 & 6 & & $-7,-12,-16,-48,-64,-112$, $(-16 \to -4)$ \\ 
        172 & 4 & 7 & 2 & 5 & & $-7,-12,-32,-48,-112$ \\ 
        228 & 4 & 9 & 2 & 6 & & $-12,-15,-32,-48,-192,-240$, $(-48\to -3)$ \\ 
        242 & 4 & 8 & 1 & 6 & & $-7,-8,-24,-40,-52,-72, (-28 \to -7)$ \\ 
        260 & 4 & 7 & 2 & 4 & & $-16,-64,-160,-2080$, $(-16 \to -4)$  \\ 
        261 & 4 & 2 & 2 & 0 & & ~ \\ 
        264 & 4 & 4 & 2 & 2 & & $-32,-96$ \\ 
        275 & 4 & 4 & 1 & 3 & & $-11,-19,-99$ \\ 
        280 & 4 & 4 & 2 & 2 & & $-160,-1120$ \\ 
        300 & 4 & 3 & 2 & 1 & & $-96$ \\ 
        306 & 4 & 6 & 2 & 4 & & $-8,-36,-36,-72$ \\ 
        308 & 4 & 6 & 2 & 3 & & $-7,-112,-448, (-112 \to -7)$\\ 
        342 & 4 & 4 & 2 & 2 & & $-8,-72$ \\ 
        350 & 4 & 3 & 1 & 2 & & $-24,-84$ \\ 
        208 & 5 & 4 & 3 & 1 & & $-64$ \\ 
        212 & 5 & 7 & 2 & 4 & & $-7,-16,-64,-112$, $(-16 \to -4)$  \\ 
        279 & 5 & 6 & 2 & 3 & & $-11,-27,-99$, $(-27\to -3)$ \\ 
        316 & 5 & 6 & 2 & 4 & & $-7,-12,-48,-112$ \\ 
        364 & 5 & 4 & 2 & 2 & & $-12,-48$ \\ 
        376 & 5 & 2 & 2 & 0 & & ~ \\ 
        378 & 5 & 2 & 1 & 1 & & $-20$ \\ 
        396 & 5 & 6 & 4 & 2 & & $-32,-288$ \\ 
        414 & 5 & 4 & 2 & 2 & & $-20,-180$ \\ 
        440 & 5 & 3 & 2 & 1 & & $-160$ \\ 
        444 & 5 & 6 & 2 & 3 & & $-12,-48,-192$, $(-48\to -3)$ \\ 
        495 & 5 & 6 & 2 & 4 & & $-11,-35,-99,-315$ \\ 
        572 & 5 & 3 & 2 & 1 & & $-352$ \\ 
        630 & 5 & 4 & 2 & 2 & & $-20,-180$ \\ 
    \end{longtable}
\end{center}

\subsection{Rational points via rank zero quotients} \label{ssec:rat pts via RZQ}
Finally, we describe the computation of the rational points on $X_0(N)^*$ for $N \in \{261,294,308,376,378,572\}$, the last six remaining cases. In \Cref{table:rkzeroquots} we list various data for the levels $N$ and their rank zero quotients. For every level except $N = 261$ we find a rank zero elliptic curve quotient, which we list by the LMFDB label. For $N = 261$ the rank zero quotient is a dimension two abelian variety, which can be identified as the Jacobian of the genus 2 modular curve $\mathrm{Jac}(X_0(87)/\langle w_{29} \rangle)$, with LMFDB label \texttt{7569.a.68121.1}.

\begin{table}[h!]
\centering
\begin{tabular}{ccccc}
\toprule
\textbf{Level} & \textbf{Factorization} & \textbf{Genus} & \textbf{Rank zero quotient $A_f$} & \textbf{Order of $A_f(\mathbb{Q})$} \\
\midrule
261 & $3^2 \cdot 29$ & 4 & $\mathrm{Jac}(X_0(87)/\langle w_{29} \rangle)$ & 5 \\
294 & $2 \cdot 3 \cdot 7^2$ & 3 & $E_{14.a1}$ & 2 \\
308 & $2^2 \cdot 7 \cdot 11$ & 4 & $E_{154.c4}$ & 4 \\
376 & $2^3 \cdot 47$ & 5 & $E_{94.a2}$ & 2 \\
378 & $2 \cdot 3^3 \cdot 7$ & 5 & $E_{21.a5}$ & 8 \\
572 & $2^2 \cdot 11 \cdot 13$ & 5 & $E_{26.b2}$ & 7 \\
\bottomrule
\end{tabular}
\caption{Data for various levels and their rank zero quotients}
\label{table:rkzeroquots}
\end{table}

To compute the rational points, we need to construct the morphism $\phi: X_0(N)^* \to A_f$ described in \eqref{eqn:formalimmersion}. 
However, we want to avoid computing the morphism $X_0(N) \to A_f$, because the genus of $X_0(N)$ is much larger (and therefore working with the associated modular forms space and canonical ring is difficult). 
This is why we first go through modular parametrizations $X_0(M) \to A_f$ (in terms of $q$-expansions) which are computationally easy, and then proceed to algebraize the maps we obtain using the canonical ring of $X_0(N)^*$. 
We were inspired by the computational techniques in \cite{DerickxOrlic}.

\textbf{Elliptic curve quotient cases}

Assume first that $E = A_f$ is an elliptic curve of conductor $M$, and adopt the notation from Proposition \ref{propdefGNf}. 
The map $\pi_f \circ \iota_M: X_0(M) \to E$  is a modular parametrization of $E$ and this map can be computed in terms of $q$-expansions with arbitrary precision.
More precisely, fixing a Weierstrass model $F_E(x_E,y_E)=0$ of $E$ with $x_E \in \Q(E), y_E \in \Q(E)$ having poles of order 2 and 3 respectively at the zero of $E$ (and nowhere else) $(\pi_f \circ \iota)^* x_E, (\pi_f \circ \iota)^* y_E$ are modular functions of weight 0 on $X_0(M)$. This is implemented in Sage with \texttt{modular\_parametrization}.

The composition of $\pi_f \circ \iota_M$ with degeneracy maps $i_{N/M}^{(d)}$ is immediate, and we can also compute their signed sum using the group law of $E$ in its Weierstrass model. This provides two modular functions on $X_0(N)^* $ whose $q$-expansions $x_N,y_N \in \Q((q))$ satisfy the Weierstrass equation of $E$, such that $G_{N,f} : X_0(N)^* \rightarrow E$ is $(x_N : y_N : 1)$.

Now, by computing monomials in a  basis $(f_1, \cdots, f_g)$ of cusp forms of $S_2(\Gamma_0(N),\Q)^*$ in terms of $q$-expansions, we can solve for $x_N,y_N$ as rational fractions $x_N = F_X(f_1, \cdots, f_g)$, $y_N = F_Y(f_1, \cdots, f_g)$ for explicit $F_X,F_Y \in \Q(X_1, \cdots, X_g)$ homogeneous of degree 0.  

In all five cases, $X_0(N)^*$ is not hyperelliptic so we can compute the canonical model of $X_0(N)^*$ using the basis $(f_1, \cdots, f_g)$. Then $F_X$ and $F_Y$ are readily interpreted as defining rational maps $X_0(N)^* \dashrightarrow \P^1$, and in this model the map $G_{N,f}$ is explicitly defined as $(F_X: F_Y:1)$  outside the poles of $F_X$ and $F_Y$ on $X_0(N)^*$.

For every rational point $P$ of $X_0(N)^*$ either $P$ is outside the poles of $F_X$ and $F_Y$ and $(F_X(P),F_Y(P))$ is a rational point of $E$, or it is in the poles of $F_X$ and $F_Y$.
Since we can compute $E(\Q)$ explicitly, the rational points of the poles of $F_X$ and $F_Y$, and the fibers of $G_{N,f}$, we thus obtain the rational points of $X_0(N)^*$.

\textbf{Abelian surface quotient case}

The computation for the dimension two quotient is slightly more involved: one begins in a similar way. Let $H$ be the hyperelliptic curve $X_0(87)/\langle w_{29} \rangle $.
Following \cite[\S 6 and 7]{Hasegawa97}, let $f_1, f_2$ be a basis for $S_2(\Gamma_0(87))^{+_{29}}$ with rational coefficients such that $f_1 = q+ \dots $ and $f_2 = q^2 + \dots $.
Letting $x_H = f_1/f_2 \in \Q(H)$ and $y_H = q \mathrm{d} x/(f_2 \mathrm{d} q) \in \Q(H)$, $x_H$ and $y_H$ satisfy the equation 
\begin{equation}y^2 =  x^6 - 2x^4 - 6x^3 - 11x^2 - 6x - 3.\end{equation}
We have $N = 261$, $M = 87$, and $N/M = 3$. Then
applying the degeneracy maps $\iota^{(1)}$ and $\iota^{(3)}$ to $x_H, y_H$ is straightforward on $q$-expansions.
We then subtract $((\iota^{(1)})^* x_H, (\iota^{(3)})^* y_H) - ((\iota^{(3)})^* x_H, (\iota^{(3)})^* y_H)$ using Cantor's algorithms for arithmetic in Jacobians of hyperelliptic curves.
 
The end result of this is $q$-expansions of the Mumford representation of a point in $\mathrm{Jac}(H)$.
In other words, we obtain (in terms of $q$-expansions up to arbitrary precision) a monic quadratic polynomial $A(X) = X^2 + A_1 X + A_0 \in \Q(X_0(N)^*)[X]$ and a linear polynomial $B(X)  = B_1 X + B_0\in \Q(X_0(N)^*)[X]$.
We can solve for $A_1, A_0, B_1, B_0$ in a similar manner as above: by constructing a basis $(f_1, \cdots, f_4)$ for $S_2(\Gamma_0(N))^*$ with rational coefficients and solving for $A_i, B_i$ in terms of degree zero homogeneous rational fractions in $(f_1, \cdots, f_4)$.
We can also find the canonical model for $X_0(N)^*$ in this basis.

This yields a map $\phi: X_0(N)^* \dashrightarrow \Sym^2(H)$ given (in Mumford-like coordinates) by
\begin{equation}
(x_1 : \cdots : x_4) \mapsto (A(x_1, \cdots, x_4)(X), Y-B(x_1, \cdots, x_4)(X)).
\end{equation}
Note that $H$ has two rational points at infinity, $\infty_+$ and $\infty_-$. Write $D_\infty = \infty_+ + \infty_-$ and define $j :  \Sym^2(H) \to \Jac(H)$ sending $D\mapsto [D - D_\infty]$.
By composition we obtain a rational map
\begin{equation}  \phi' \coloneqq j \circ \phi :X_0(N)^* \dashrightarrow \Jac(H), \end{equation}
defined on the open subset $U = X_0(N)^* \setminus \{ {\text{poles of }} A_0,A_1,B_0,B_1\}$.

The map $D \mapsto [D]$ from $\Sym^2(H) \to \Pic^2(H)$ is injective away from divisors $D$ of the shape $P + \iota(P)$ (which are linearly equivalent to the canonical divisor of $H$), hence $j$ is injective outside $j^{-1}(0)$, and $j^{-1}(0)$ is the set of $P + \iota(P), P \in H$.

We can compute the rational points of $\Jac(H)(\Q) \simeq \Z/5 \Z$. Two of the (nonzero) points are given by 
\begin{equation} (X^2 + X + 1, Y \pm (X + 2))  \in \Sym^2(H)(\Q).\end{equation}
There are no rational points in the fiber of $\varphi$ above these points.

The two other nonzero points are given by the difference of the two points at infinity,
\begin{equation} \pm (\infty_- - \infty_+)\end{equation}
and are not in $\phi(U)$ by construction.
Therefore there are no rational points in the fiber above these points.

Next we consider $0 \in A_f$, and assume  for some $P \in X_0(N)^*(\Q)$, that $\varphi'(P)=0$. We then have $D \coloneqq \varphi(P) \in \varphi(U) \cap j^{-1}(0)$. We can write $D = W + \iota(W)$ for some point $W=(x_W,y_W) \in H$, but then by definition of $A$ and $B$, we have $A(P)(X) = (X - x_W)(X-x_{\iota(W)}) = (X- x_W)^2$ and $y_W = B(P)(x_W)$, but the same relation holds for $\iota(W)$ so $-y_W = b_1(P) x_W + b_0(P)$ also, hence $y_W = 0$ and $W$ is a Weierstrass point of $H$. In particular $A_1(P)=- 2 x_W$, so $x_W \in \Q$, but there are no rational Weierstrass points on $H$ hence $(\varphi')^{-1}(0) \cap X_0(N)^*(\Q) = \emptyset$.

Finally, we check the subset of $X_0(N)^*$ where $\phi$ is not defined, the poles of $A_0,A_1,B_0,B_1$. This is a dimension zero scheme that has two rational points.

\subsection{Algorithms and Code} \label{subsec:algorithms and code}
Many of the theorems contained in this paper rely on computational results and code.
Our Magma and Sage code and its output can be found in~\cite{CodeForOurArticle}.
Below is an overview of the main algorithms which are implemented in the code.

\begin{itemize}
    \item (Models of modular curves) We compute a canonical model of an Atkin--Lehner quotient of $X_0(N)$ using relations between $q$-expansions. Our code is based on that of~\cite{ACKP,AKMJNOV}.
    
    \item (Rational points) We have a script to automate the checking of rational points on $X_0(N)^*$ by first computing a suitable plane model (by successively projecting down from rational points on the curve)\footnote{Note that this generally produces models with smaller coefficients and is simpler than using products of two gonal maps that were used in~\cite{AABCCKW}.} for Balakrishnan--Tuitman's Chabauty--Coleman algorithm \cite{BalakrishnanTuitman},
    then finding a suitable Chabauty prime to prove that the known small $\Q$-points are complete. Then it checks the Chabauty condition. It iterates over primes until it (hopefully) finds a prime where it succeeds in determining the rational points.

    \item (Sums of roots of unity) We compute the rational primes that lie below prime divisors of the sums of roots of unity in~\eqref{eq:sums of roots of unity}  as described in the statements of~\cref{subsec:formal immesions}. For the output, see~\cref{tab:NM_values} and \cref{tab:values}. 

    \item (Exceptional levels) We compute a list of minimal exceptional levels as described in~\cref{sec:exclevels}.
    
    \item (Heegner points) We implement an algorithm to compute the rational points on the star quotient which arise as the image of Heegner points on $X_0(N)$ as described in~\cref{sec:heegnerpts}.

    \item (Rational points via rank zero quotients) We implement the computation of rational points on $X_0(N)^*$ via pullback from a map $\phi:X_0(N)^* \to A_f $ where $A_f$ is rank zero and in the new part of $J_0(M)$ in certain cases when Chabauty--Coleman fails.

    \item (LMFDB search) We search the LMFDB using a Sage script to establish the existence of newforms with certain small levels and with certain Atkin--Lehner eigenvalues. See the proof of~\cref{proprankzeroquotient}.
\end{itemize}

    \bibliographystyle{amsalpha}
	\bibliography{biblioX0star}

\end{document}